\documentclass[a4paper, DIV=15, abstract, toc=bibliography]{scrartcl}
\pdfoutput=1
\newif\ifprivate
\privatefalse

\usepackage[utf8]{inputenc}
\usepackage[T1]{fontenc}
\usepackage{lmodern}
\usepackage[draft=false, kerning=true]{microtype}

\usepackage[english]{babel}
\usepackage{graphicx}
\usepackage{amsmath, amsfonts, amsthm}
\allowdisplaybreaks
\usepackage{mathtools}
\usepackage{multirow}
\usepackage[shortlabels]{enumitem}
\usepackage{ifthen}
\usepackage{float}
\usepackage{physics}
\usepackage{stmaryrd}
\usepackage[dvipsnames]{xcolor}
\usepackage{tikz}
\usetikzlibrary{decorations.markings}
\usetikzlibrary{arrows.meta,calc,decorations.markings,math,arrows.meta}
\usetikzlibrary{shapes,snakes}
\usetikzlibrary{shapes.multipart}
\usepackage{diagbox}
\usepackage{todonotes}
\usepackage{enumitem}
\usepackage{xparse}

\ifprivate
\usepackage[mark]{gitinfo2}
\renewcommand{\gitMark}{\jobname\,\textbullet{}\,\gitFirstTagDescribe\,\textbullet{}\,\gitAuthorName,\,\gitAuthorIsoDate}
\fi

\newcommand{\R}{\mathbb{R}}

\newcommand{\N}{\mathbb{N}}
\newcommand{\Z}{\mathbb{Z}}
\newcommand{\C}{\mathbb{C}}

\newcommand{\calG}{\mathcal{G}}

\DeclareMathOperator{\ind}{ind}

\DeclareMathOperator{\diag}{diag}
\DeclareMathOperator{\kernel}{\mathcal{K}}

\DeclareMathAlphabet{\mcal}{OMS}{cmsy}{m}{n}
\DeclarePairedDelimiter\set\{\}
\DeclarePairedDelimiterX{\setm}[2]{\{}{\}}{#1\,\delimsize\vert\,\mathopen{}#2}
\let\abs\relax
\DeclarePairedDelimiter{\abs}{\lvert}{\rvert}
\let\norm\relax
\DeclarePairedDelimiter{\norm}{\lVert}{\rVert}
\DeclarePairedDelimiter{\fractional}{\{}{\}}
\DeclarePairedDelimiter{\iverson}{\llbracket}{\rrbracket}
\DeclarePairedDelimiter{\floor}{\lfloor}{\rfloor}
\DeclarePairedDelimiter{\ceil}{\lceil}{\rceil}
\newcommand{\stirling}[2]{\genfrac{\{}{\}}{0pt}{}{#1}{#2}}
\newcommand{\twoldots}{\,.\,.\,}
\newcommand{\oeis}[1]{\cite[\href{https://oeis.org/#1}{#1}]{OEIS:2021}}
\newlength{\mb}
\newcommand{\downtosymb}[3][0em]{\setlength{\mb}{\widthof{#2}/2 - #1}%
\stackrel{\hspace*{-\mb}\stackrel{\text{#2}}{\downarrow}\hspace*{-\mb}}{#3}}
\newcommand{\downtoeq}[2][0em]{\downtosymb[#1]{#2}{=}}

\makeatletter
\newcommand\undisp[1]{\bgroup\@displayfalse #1\egroup}
\makeatother
\newcommand{\tpmod}[1]{\ensuremath{\undisp{\pmod{#1}}}}

\newcommand{\tikzmark}[1]{\tikz[overlay,remember picture] \node (#1) {};}
\NewDocumentCommand{\DrawBox}{
  O{} 
  O{0em}
  O{0em}
  O{0em}
  O{0em}
  m 
  m 
  }{%
  \tikz[overlay,remember picture]{
    \draw[#1]
    ($(#6)+(-0.2em,0.8em)+(#5,#2)$) rectangle
    ($(#7)+(0.2em,-0.2em)+(#3,#4)$);}
}
\newcommand{\coolrightbrace}[2]{\mathclap{\left.\vphantom{\begin{matrix} #1 \end{matrix}}\right\}}\hspace{0.75em}#2}

\newtheorem{theorem}{Theorem}

\newtheorem{definition}{Definition}[section]
\newtheorem{lemma}[definition]{Lemma}
\newtheorem{corollary}[theorem]{Corollary}
\newtheorem{proposition}[definition]{Proposition}

\theoremstyle{remark}
\newtheorem{examplex}[definition]{Example}
\newenvironment{example}
    {\pushQED{\qed}\examplex}
    {\popQED\endexamplex}
\newtheorem{remarkx}[definition]{Remark}
\newenvironment{remark}
    {\pushQED{\qed}\remarkx}
    {\popQED\endremarkx}

\usepackage{hyperref}
\hypersetup{
  pdftex,
  a4paper,
  bookmarks,
  bookmarksopen=true,
  bookmarksnumbered=true,
  pdftitle={Asymptotic Analysis of Recursive Sequences},
  pdfauthor={Clemens Heuberger, Daniel Krenn, Gabriel F. Lipnik},
  pdflang=english,
  colorlinks = false,
  linkbordercolor = {MidnightBlue},
  urlbordercolor = {MidnightBlue},
  citecolor = {ForestGreen},
  citebordercolor = {ForestGreen},
  plainpages=false
}

\makeatletter
\newcommand{\makeanchor}{\hyper@anchor{\@currentHref}}
\makeatother

\numberwithin{equation}{section}
\numberwithin{figure}{section}
\numberwithin{table}{section}

\begin{document}
\title{Asymptotic Analysis of $q$-Recursive Sequences}
\author{Clemens Heuberger, Daniel Krenn, Gabriel F.\ Lipnik}
\date{}

\maketitle

\begin{abstract}
  For an integer $q\ge2$, a $q$-recursive sequence is defined by recurrence
  relations on subsequences of indices modulo some powers of~$q$.
  In this article, $q$-recursive sequences are studied
  and the asymptotic behavior of their summatory functions is analyzed.
  It is shown that every $q$-recursive sequence is $q$-regular in
  the sense of Allouche and Shallit and that a $q$-linear representation of the
  sequence can be computed easily by using the coefficients from the recurrence relations.
  Detailed asymptotic results for $q$-recursive sequences are then
  obtained based on a general result on the asymptotic
  analysis of $q$-regular sequences.

  Three particular sequences are studied in detail:
  We discuss the asymptotic
  behavior of the summatory functions of\setlist{nolistsep}
  \begin{itemize}[noitemsep]
  \item Stern's diatomic sequence,
  \item the number of non-zero elements in some generalized Pascal's triangle
    and
  \item the number of unbordered factors in the Thue--Morse sequence.
  \end{itemize}
  For the first two sequences, our analysis even leads to precise formul\ae{} without error terms.
\end{abstract}

\vfill
{\footnotesize
  \begin{description}
  \item [Clemens Heuberger]
    \href{mailto:clemens.heuberger@aau.at}{\texttt{clemens.heuberger@aau.at}},
    \url{https://wwwu.aau.at/cheuberg},
    Alpen-Adria-Universität Klagenfurt, Austria
  \item [Daniel Krenn]
    \href{mailto:math@danielkrenn.at}{\texttt{math@danielkrenn.at}},
    \url{http://www.danielkrenn.at},
    Paris Lodron University of Salzburg, Austria
  \item [Gabriel F.\ Lipnik]
    \href{mailto:math@gabriellipnik.at}{\texttt{math@gabriellipnik.at}},
    \url{https://www.gabriellipnik.at},
    Graz University of Technology, Austria
  \item [Support]
    Clemens Heuberger and Daniel Krenn are supported by the
    Austrian Science Fund (FWF): P\,28466-N35. Gabriel F.\ Lipnik is supported by the Austrian Science Fund (FWF): W\,1230.
  \item [Acknowledgment]
    The authors thank Helmut Prodinger for
    drawing their attention to the counting function of unbordered
    factors in the Thue--Morse sequence.
  \item [2020 Mathematics Subject Classification]
    05A16; 
    11A63, 
    11B37, 
    30B50, 
    68Q45, 
    68R05, 
    68R15  
  \item[Key words and phrases]
    regular sequence,
    recurrence relation,
    digital function,
    summatory function,
    asymptotic analysis,
    Dirichlet series,
    Stern's diatomic sequence,
    Pascal's triangle,
    Thue--Morse sequence
  \end{description}}

\newpage
\tableofcontents
\newpage

\section{Introduction}
\label{chap:intro}

\paragraph{$q$-Recursive Sequences.}
We study a special class of recursively defined sequences, the
so-called \emph{$q$-recursive sequences}. Here $q$ is an integer and
at least $2$, and
$q$-recursive sequences are sequences which satisfy a specific type
of recurrence relation: Roughly speaking,
every subsequence whose indices run through a residue class modulo~$q^M$
is a linear combination of subsequences where for each of these subsequences, the indices run through a residue class modulo~$q^m$ for some $m < M$.

It turns out that this property is quite
natural and many combinatorial sequences are in fact $q$-recursive.
A simple nontrivial example of such a sequence%
\footnote{Throughout this paper, we let $\N$ denote the set of positive integers and write $\N_{0} \coloneqq \N\cup\set{0}$. Moreover, we write sequences in functional notation, i.e., we write a (complex-valued) sequence~$x$
  as a function~$x\colon\N_{0}\to\C$, and as a consequence, the $n$th element
  of the sequence is denoted by $x(n)$. In addition, we consider a vector of
  complex-valued sequences~$v = (x_{1}, \dots, x_{D})^{\top}$ as a
  vector-valued function~$v\colon\N_{0}\to\C^{D}$, and its evaluation is
  defined component-wise, i.e., $v(n) = (x_{1}(n),\dots,x_{D}(n))^{\top}$.}
is when $h(n)$ is the
largest power of $2$ less than or equal to~$n$; see~\oeis{A053644}.
Then we have $h(2n)=2 h(n)$ and $h(2n+1)=2 h(n)$ for $n\ge 1$ as well as $h(1)=1$, so
clearly $q=2$, and we set $M=1$ and $m=0$. This is because the
left-hand sides of the two recurrence relations contain $2^1n$
shifted by $0$ and $1$, and because the right-hand sides only contain
$2^0n$ (and no shifts). Another example are divide-and-conquer recurrences,
see~\cite[Equation~(1.2)]{Hwang-Janson-Tsai:2017:divide-conquer-half}.

\paragraph{$q$-Regular Sequences.}
The concept of $q$-recursive sequences is related to $q$-regular
sequences introduced by Allouche and
Shallit~\cite{Allouche-Shallit:1992:regular-sequences}. One definition
of a $q$-regular sequence $x$ is that every subsequence of the form
$x(q^jn+r)$ can be written as a linear combination of the same finite
number of sequences; see Section~\ref{sec:reg-seq} for more details
and a precise description. Again the sequence~$h$ (\oeis{A053644})
is an example, now for a $2$-regular sequence; it satisfies
$h(2^jn+r)=2^j h(n)$ for\footnote{By the
  definition of $q$-regular sequences as given in
  Section~\ref{sec:reg-seq}, when representing every subsequence of
  the form $h(2^jn+r)$ as a linear combination of the same finite
  number of sequences, this needs to hold for all $n\geq 0$.
  In particular, this needs to hold for
  $n = 0$, which is not fulfilled in the example. However, this is only a minor
  technical issue and can be fixed by adding some appropriate sequence
  to the linear combination. Details about this repair are provided in
  Theorem~\ref{thm:offset-correction}.}
$n \geq 1$, $j\ge 0$ and $0\le r<2^j$, so every $h(2^jn+r)$ can be written in terms of~$h(n)$.

Equivalently, every $q$-regular sequence
can be modeled by a $q$-linear representation. Here $x(n)$ is one
component of a vector $v(n)$, and there exist matrices~$A_r$ with
$v(qn+r)=A_r v(n)$ for all $0 \le r < q$ and~$n \geq 0$; see also
Section~\ref{sec:reg-seq}.

\paragraph{Linear Representation of $q$-Recursive Sequences.}

One main result of this paper is that every $q$-recursive sequence is
indeed $q$-regular; see Section~\ref{sec:rec-seq-main}. Even more can
be said: Theorem~\ref{thm:remark-2.1-general} provides an explicit
$q$-linear representation of the sequence.

In Section~\ref{sec:recursive-special}, we significantly
improve our results for a special case of $q$-recursive sequences,
namely where~$M=m+1$.

\paragraph{Asymptotics of $q$-Recursive Sequences.}

After exploring the concept of $q$-recursive sequences itself,
we investigate the asymptotic behavior of the summatory
functions of $q$-recursive sequences, i.e., sequences of partial sums.
There exist explicit results for many particular regular sequences and
also some quite general results.
Dumas~\cite{Dumas:2013:joint} as well as the first two authors
of this paper together with
Prodinger~\cite{Heuberger-Krenn-Prodinger:2018:pascal-rhombus, Heuberger-Krenn:2018:asy-regular-sequences}
studied the asymptotics of $q$-regular sequences in general. The two
works~\cite{Heuberger-Krenn-Prodinger:2018:pascal-rhombus, Heuberger-Krenn:2018:asy-regular-sequences}
will also be
one of the main ingredients for obtaining the asymptotic behavior of
$q$-recursive sequences. We present details in Section~\ref{sec:asymp}.
We investigate an important special case where the asymptotic behavior can be
directly determined from the $q$-recursive sequence without constructing the
representation as a $q$-regular sequence.

\paragraph{Explicit Precise Asymptotics for Three Particular Sequences.}

We also investigate three specific $q$-recursive sequences in-depth.
In particular, we derive
asymptotic results for their summatory functions as well as explain
and illustrate the connection between these results and the fact that
the sequences are $q$-recursive. To be more specific, we analyze
\begin{itemize}
\item Stern's diatomic sequence in Section~\ref{sec:stern-brocot},
\item the number of non-zero entries in a generalized Pascal's
  triangle in Section~\ref{sec:pascal}, and
\item the number of unbordered factors in the Thue--Morse sequence in Section~\ref{sec:ub}.
\end{itemize}
For the first two sequences, our analysis even leads to precise formul\ae{}
without error terms.

\paragraph{Proofs.}

We finally complete this paper by giving proofs of our
results; these are collected in Section~\ref{chap:proofs}.

\section{Brief Introduction to $q$-Regular Sequences}
\label{sec:reg-seq}

The concept of $q$-regular sequences\footnote{In the standard
  literature~\cite{Allouche-Shallit:1992:regular-sequences,
    Allouche-Shallit:2003:autom}, these sequences are called
  $k$-regular instead of $q$-regular.} was first introduced by Allouche and Shallit~\cite{Allouche-Shallit:1992:regular-sequences} in
1992, and a lot of research on them has
been done since then; see for
example Allouche and Shallit~\cite{Allouche-Shallit:2003:autom} and%
~\cite{Allouche-Shallit:2003:regular-sequences-2},
Bell~\cite{Bell:2005:regular-sequences-values},
and Coons and Spiegelhofer~\cite{Coons-Spiegelhofer:2018:number-theoretic-regular-sequences}.

The parameter~$q$ acts as a base (or radix); therefore the term
\emph{digital function} arises in context of such sequences.
We start by giving a definition;
see Allouche and Shallit~\cite{Allouche-Shallit:2003:autom}.

Let $q\geq 2$ be a fixed integer and $x\colon\N_{0}\to\C$ be a
sequence\footnote{The results given in Sections~\ref{sec:rec-seq-main}
  and~\ref{sec:recursive-special} are valid for sequences $x\colon \N_{0}\to R$,
  where~$R$ is a commutative Noetherian ring.  However, in places where we
  speak about the asymptotics of a sequence, we always consider complex-valued
  sequences.}.  Then~$x$ is called \emph{$q$-regular} if the complex vector
space generated by its $q$-kernel
\begin{equation*}
  \kernel_{q}(x)\coloneqq \setm[\big]{x\circ (n\mapsto q^{j}n +
    r)}{\text{integers }j\geq 0,
    0\leq r < q^{j}}
\end{equation*}
has finite dimension.
In other words, a sequence~$x$ is \emph{$q$-regular} if there are $\Delta\in\N_{0}$
and sequences $x_{1}$, $\dots$, $x_{\Delta}$ such that for every
$j\in\N_{0}$ and $0\leq r < q^{j}$ there exist $c_{1}$, $\dots$,
$c_{\Delta}\in\C$ with
\begin{equation*}
  x(q^{j}n + r) = \sum_{i=1}^{\Delta}c_{i}x_{i}(n)
\end{equation*}
for all $n \geq 0$.

By Allouche and
Shallit~\cite[Theorem~2.2]{Allouche-Shallit:1992:regular-sequences}, a
complex-valued sequence~$x$ is $q$-regular if and only if there exist a
vector-valued sequence~$v\colon \N_{0}\to \C^{D}$ for some $D\in\N$ whose first
component coincides with~$x$ and matrices $A_{0}$, \dots, $A_{q-1}$ such that
\begin{equation*}
  v(qn + r) = A_{r}v(n)
\end{equation*}
holds for all $0\leq r < q$ and $n \geq 0$.
If this is the case, the tuple $(A_{0},\dots,A_{q-1},v)$ is called
\emph{$q$-linear representation} of~$x$, and~$D$ is said to be its \emph{dimension}.

At this point, we note that a $q$-linear
representation~$(A_{0}, \dots, A_{q-1},v)$ of a sequence~$x$ immediately leads
to an explicit expression for~$x(n)$ by induction: Let $d_{L-1}\ldots d_{0}$
be the $q$-ary digit expansion of~$n$. Then we have
\begin{equation}
  \label{eq:xn-lin-rep}
  x(n) = e_{1}A_{d_{0}}\ldots A_{d_{L-1}}v(0),
\end{equation}
where $e_{1} = (1, 0, \ldots, 0)$. In particular, the $n$th element~$x(n)$ can
be computed in~$O(\log n)$ operations in $\C$.

The prototypical and probably best-known example of a $q$-regular sequence is the binary
sum of digits.

\begin{example}[Binary Sum of Digits~\oeis{A064547}]
  \label{ex:binary-sum-of-digits}
  For $n\in\N_{0}$, let $s(n)$ denote the number of ones in the binary
  expansion of~$n$. Then we clearly have
  \begin{equation}\label{eq:binary-sum-of-digits}
    s(2n) = s(n)\qq{and}
    s(2n + 1) = s(n) + 1
  \end{equation}
  for all $n \geq 0$ and $s(0)=0$. By induction we obtain
  \begin{equation*}
    s(2^{j}n + r) = s(n) + s(r)\cdot 1
  \end{equation*}
  for all $n \geq 0$, $j \geq 0$ and $0\leq r < 2^{j}$. This means that the complex
  vector space generated by the kernel~$\kernel_{q}(s)$ is also generated by~$s$ and
  $n\mapsto 1$ and thus, the sequences~$s$ is $2$-regular.

  A $2$-linear representation $(A_{0}, A_{1}, v)$ of~$s$ is given by
  \begin{equation*}
    A_{0} =
    \begin{pmatrix}
      1 & 0\\
      0 & 1
    \end{pmatrix},\quad
    A_{1} =
    \begin{pmatrix}
      1 & 1\\
      0 & 1
    \end{pmatrix}
    \qq{and}
    v =
    \begin{pmatrix}
      s\\
      n\mapsto 1
    \end{pmatrix},
  \end{equation*}
  where the corresponding recurrence relations $v(2n) = A_{0}v(n)$ as well as
  $v(2n + 1) = A_{1}v(n)$ can be checked easily by
  using~\eqref{eq:binary-sum-of-digits}.
\end{example}



\section{$q$-Recursive Sequences}
\label{chap:lin-rep}

In this section, we introduce the concept of $q$-recursive sequences,
investigate how linear representations of these sequences look like
and thereby conclude that these sequences are $q$-regular.
We will also investigate a more restricted set-up.
This special case
(as we will call it from now on in contrast to the general case)
is important; two instances will be discussed
in Sections~\ref{sec:pascal} and~\ref{sec:ub}.

\subsection{Definitions}
\label{sec:rec-seq-definition}
We start by giving the definition of our sequences of interest.

\begin{definition}[$q$-Recursive Sequence]
  \label{definition:q-reqursive-sequence}
  Let $q\geq 2$, $M > m \geq 0$, $\ell\leq u$ and $n_{0}\geq \max\set{-\ell/q^{m},0}$ be fixed
  integers. Let~$x$ be a sequence.

    If there are constants~$c_{s,k}\in\C$ for
    all~$0\leq s < q^{M}$ and $\ell\leq k\leq u$ such that
    \begin{equation}\label{eq:def-q-recursive}
      x(q^{M}n + s) = \smashoperator{\sum_{\ell\leq k\leq u}} c_{s,k}\,x(q^{m}n + k)
    \end{equation}
    holds for all $n\geq n_{0}$ and $0\le s<q^{M}$, then
    we say that the sequence~$x$
    is \emph{$q$-recursive with offset~$n_{0}$, exponents~$M$ and~$m$,
      index shift bounds~$\ell$ and~$u$, and coefficients~$(c_{s,k})_{0\le s<q^M,\ell\le k\le u}$}.
  \end{definition}

    We use the convention that
    if any of the parameters~$q$, $n_{0}$, $M$, $m$, $\ell$, $u$, $(c_{s,k})_{0\le s<q^M,\ell\le k\le u}$
    is not mentioned for a recursive sequence,
    then we assume that a value of this parameter exists
    such that the sequence is recursive with this value of the parameter.

  The sequence where~$h(n)$ is the largest power of $2$ less than or equal to~$n$
  (\oeis{A053644}), that was mentioned in the introduction, is indeed a
  $2$-recursive sequence with offset~$1$, as $h(2n)=2 h(n)$ and
  $h(2n+1)=2 h(n)$ hold for $n\ge 1$. On the other hand, the
  binary sum of digits as introduced in Example~\ref{ex:binary-sum-of-digits}
  does not directly\footnote{While the binary sum of digits~$s$ does not directly fit into the framework with $M=1$ and $m=0$, it is a
    $2$-recursive sequence with exponents $M = 2$ and $m = 1$, since the
    recurrence relations $s(4n)=s(2n)$, $s(4n + 1)=s(2n + 1)$,
    $s(4n + 2)=s(2n + 1)$ and $s(4n + 3)= -s(2n) + 2s(2n+1)$ follow
    from~\eqref{eq:binary-sum-of-digits} for all $n\geq 0$. This means that in
    this and similar cases we can get rid of the inhomogeneity by increasing the exponents.}
  fit into this framework, because the constant sequence appears on the
  right-hand side; see~\eqref{eq:binary-sum-of-digits}. For a discussion of
  such inhomogeneous $q$-recursive sequences we refer to
  Corollary~\ref{cor:inhomogeneities}.

Before considering a slightly more involved example, we clarify the role of the restriction on $n_0$.

\begin{remark}
    The condition $n_{0}\geq \max\set{-\ell/q^{m},0}$ in
    Definition~\ref{definition:q-reqursive-sequence} is necessary because for
    $n=n_0$,~\eqref{eq:def-q-recursive} reduces to
    \begin{equation*}
      x(q^{M}n_{0} + s) = \smashoperator{\sum_{\ell\leq k\leq u}} c_{s,k}\,x(q^{m}n_{0} + k),
    \end{equation*}
    and so the smallest argument of~$x$ on the right-hand side is
    $q^{m}n_{0} + \ell$, which is non-negative by the given condition and
    therefore indeed a valid argument.
\end{remark}

\begin{example}[Odd Entries in Pascal's Triangle~\oeis{A006046}]
  \label{ex:pascal-odd}
  Let $p(n)$ be the number of odd entries in the first~$n$ rows of Pascal's
  triangle. The first few elements are given in Table~\ref{tab:pascal-odd}.

  \begin{table}[htpb]
    \centering
    \begin{tabular}{c|cccccccccccccccccccccc}
      $n$ & $0$ & $1$ & $2$ & $3$ & $4$ & $5$ & $6$ & $7$ & $8$ & $9$ & $10$\\\hline
      $p(n)$ & $0$ & $1$ & $3$ & $5$ & $9$ & $11$ & $15$ & $19$ & $27$ & $29$ & $33$
    \end{tabular}
    \caption{First few elements of $p$}
    \label{tab:pascal-odd}
  \end{table}

  By Lucas' theorem on binomial coefficients modulo a prime,
  the number of odd entries in row~$n$ of Pascal's triangle is given
  by~$2^{s(n)}$, where $s(n)$ is the binary sum of digits of~$n$; see also Fine~\cite[Theorem~2]{Fine:1947:binomial-coefficients-modulo-p}. This
  implies that
  \begin{align*}
    p(2n) = \smashoperator{\sum_{0\leq k < 2n}}2^{s(k)}
    &= \smashoperator{\sum_{0\leq k < n}}2^{s(2k)} + \smashoperator{\sum_{0\leq k < n}}2^{s(2k + 1)}\\
    &\downtoeq{\eqref{eq:binary-sum-of-digits}} \smashoperator{\sum_{0\leq k < n}}2^{s(k)} + \smashoperator{\sum_{0\leq k < n}}2^{s(k) + 1}
      = p(n) + 2p(n) = 3p(n)
  \end{align*}
  as well as
  \begin{align*}
    p(2n + 1) = \smashoperator{\sum_{0\leq k < 2n + 1}}2^{s(k)}
    &= \smashoperator{\sum_{0\leq k \leq n}}2^{s(2k)} + \smashoperator{\sum_{0\leq k < n}}2^{s(2k + 1)}\\
    &\downtoeq{\eqref{eq:binary-sum-of-digits}} \smashoperator{\sum_{0\leq k \leq n}}2^{s(k)} + \smashoperator{\sum_{0\leq k < n}}2^{s(k) + 1}
    = p(n + 1) + 2p(n)
  \end{align*}
  hold for all $n \geq 0$.
  Thus, the sequence~$p$ is $2$-recursive with exponents $M=1$ and $m=0$,
  index shift bounds~$\ell=0$ and $u=1$, and offset~$n_{0}=0$.

From Allouche and Shallit~\cite[Example~14 and Theorem~3.1]{Allouche-Shallit:1992:regular-sequences}, we know that the sequence~$p$
is $2$-regular as well. This is no coincidence: In the following, we will show that each
$q$-recursive sequence is $q$-regular. Furthermore, if the recurrence relations
in~\eqref{eq:def-q-recursive} are known, we can even give an explicit
$q$-linear representation of~$x$.
\end{example}

In analogy to Definition~\ref{definition:q-reqursive-sequence}, we also introduce \emph{$q$-regular sequences
    with offset}.
  \begin{definition}[$q$-Regular Sequence with Offset]
    Let $q \geq 2$ and $n_0 \geq 0$ be fixed integers.
    A sequence~$x$ is said to be \emph{$q$-regular with
    offset~$n_{0}$} if there exist a vector-valued
  sequence~$v\colon \N_{0}\to \C^{D}$ for some $D\in\N$ whose first component
  coincides with~$x$ and matrices $A_{0}$, \dots, $A_{q-1}$ such that
  $v(qn + r) = A_{r}v(n)$ holds for all $0\leq r < q$ and $n \geq n_{0}$. In
  this case, we say that $(A_{0},\ldots, A_{q-1}, v)$ is a \emph{$q$-linear
    representation with offset~$n_{0}$} of~$x$.
\end{definition}

\begin{remark}
  A $q$-regular sequence with offset~$0$ is $q$-regular in the usual
  sense. Likewise, every $q$-linear representation with offset~$0$ is a
  $q$-linear representation as introduced in Section~\ref{sec:reg-seq}.
\end{remark}

\subsection{Reduction to $q$-Regular Sequences in the General Case}
\label{sec:rec-seq-main}

It turns out that every $q$-recursive sequence with any offset is indeed
$q$-regular (Corollary~\ref{corollary:main:is-q-regular}). This is an
implication of the following two results:
\begin{itemize}
\item Theorem~\ref{thm:remark-2.1-general} explicitly constructs a $q$-linear
  representation with offset~$n_{1}$ of $q$-recursive sequences with
  offset~$n_{0}$, where~$n_{1}\in\N$ is explicitly given. This means that such
  sequences are $q$-regular with offset~$n_{1}$.
\item Theorem~\ref{thm:offset-correction} states that every $q$-regular
  sequence with some offset is $q$-regular (without offset) as well. Also
  here, an explicit $q$-linear representation of~$x$ is given.
\end{itemize}

\begin{theorem}\label{thm:remark-2.1-general}
  Let $x$ be a $q$-recursive sequence with offset~$n_{0}$, exponents~$M$ and~$m$ and
  index shift bounds~$\ell$ and~$u$. Furthermore, set\footnote{We use Iverson's convention:
    For a statement~$S$, we set $\iverson{S} = 1$ if~$S$ is true and~$0$
    otherwise; see also Graham, Knuth and Patashnik~\cite[p.~24]{Graham-Knuth-Patashnik:1994}.}
  \begin{subequations}
    \label{eq:ell-and-u}
  \begin{align}
    \label{eq:ell}
    \ell' &\coloneqq \floor[\bigg]{\frac{(\ell+1) q^{M-m} - q^{M}}{q^{M-m}-1}} \iverson{\ell < 0}
    \intertext{and}
    \label{eq:u}u' &\coloneqq q^{m} - 1 + \ceil[\bigg]{\frac{uq^{M-m}}{q^{M-m}-1}} \iverson{u > 0}.
  \end{align}
  \end{subequations}
  Then $x$ is $q$-regular with offset~$n_{1} = n_{0} - \floor{\ell' /q^{M}}$, and a $q$-linear representation $(A_{0}, \dots,
  A_{q-1}, v)$ with offset~$n_{1}$ of~$x$ is
  given as follows:
  \begin{enumerate}[(a)]
    \item The vector-valued sequence~$v$ is given in block form by
    \begin{equation}
      \label{eq:recursive-blocks-of-v}
      v =
      \begin{pmatrix}
        v_{0}\\
        \vdots\\
        v_{M-1}
      \end{pmatrix},
    \end{equation}
    where the blocks are of the following form:
    For $0\leq j < m$, the block~$v_{j}$ has the form
    \begin{subequations}\label{eq:recursive-components-of-v}
      \begin{equation}
        \label{eq:recursive-components-of-v-1}
        v_{j} =
        \begin{pmatrix}
          x\circ(n\mapsto q^{j}n)\\
          \vdots\\
          x\circ(n\mapsto q^{j}n + q^{j} - 1)
        \end{pmatrix},
      \end{equation}
    and for $m\leq j < M$, the block~$v_{j}$ has the form\footnote{We
      set $x(n)=0$ for $n < 0$ to ensure that all blocks are well-defined.
      By the choice of $n_1$, this does not have any factual consequences.}
    \begin{equation}
      \label{eq:recursive-components-of-v-2}
      v_{j} =
      \begin{pmatrix}
        x\circ(n\mapsto q^{j}n + \ell')\\
        \vdots\\
        x\circ(n\mapsto q^{j}n + q^{j} - q^{m}  + u')
      \end{pmatrix}.
    \end{equation}
  \end{subequations}

\item The matrices $A_{0}$, \ldots, $A_{q-1}$ of
  the $q$-linear representation with offset~$n_{1}$ can be computed by
  using the coefficients in~\eqref{eq:def-q-recursive}; an explicit formula for
  the rows of these matrices is given
  in~\eqref{eq:Ar-explicitly}.
\end{enumerate}
  The linear representation $(A_{0}, \ldots, A_{q-1}, v)$ does not depend on the
  offset~$n_{0}$.
\end{theorem}

\begin{remark}
  \makeanchor
  \label{rem:main-thm}
  \begin{enumerate}
  \item We easily verify that $\ell'\le 0$ holds and it is clear that $u'\ge q^m-1 \geq 0$. Thus $\ell' \leq u'$. This implies that the blocks $v_j$ for $m\le j<M$ in~\eqref{eq:recursive-components-of-v-2} are indeed non-empty.
  \item It is easy to check that $x$ itself is a component
    of~$v$. For $m=0$, this is due to the fact that we
    have~$\ell'\leq 0 \leq u'$. However, it can happen that~$x$ is not
    the \emph{first component} of~$v$ (as it is required for a linear
    representation). Then a simple permutation of the components
    of~$v$ brings~$x$ to its first component.
  \item\label{item:minimization}
    The dimension of the $q$-linear representation is
    \begin{equation*}
      \frac{q^{M} - 1}{q - 1} + (M - m)\bigl(u' - \ell' - q^{m} + 1\bigr),
    \end{equation*}
    which is possibly very big. However, we can always apply a
    minimization algorithm in order to decrease the dimension of
    the linear representation as far as possible.
    Such an algorithm is presented
    in Berstel and
    Reutenauer~\cite[Chapter~2]{Berstel-Reutenauer:2011:noncommutative-rational-series}
    for recognizable series, but can be applied on regular sequences
    as well; see~\cite{Heuberger-Krenn-Lipnik:2022:minimisation-notes}.
    SageMath~\cite{SageMath:2021:9.4} provides an implementation
    of this minimization algorithm.
  \item The statement of Theorem~\ref{thm:remark-2.1-general}
    for $M = 1$ and $m = n_{0} = 0$
    is mentioned by the first two
    authors of this article in~\cite[Remark~5.1]{Heuberger-Krenn:2018:asy-regular-sequences}.
    \qedhere
  \end{enumerate}
\end{remark}

In order to put the main aspects of the previous result across, we present
two examples: The first one is a simple continuation of
Example~\ref{ex:pascal-odd}, and the second one discusses a $q$-recursive sequence
with more involved parameters. While the latter might not seem to be very
natural, it is an intentionally made choice
to keep things illustrative and
comprehensible. For further natural combinatorial examples we refer to
Sections~\ref{sec:stern-brocot},~\ref{sec:pascal} and~\ref{sec:ub}.

\begin{example}[Odd Entries in Pascal's Triangle, continued]
  \label{ex:odd-pascal-after-main-result}
  Let $p(n)$ again be the number of odd entries in the first~$n$ rows
  of Pascal's triangle. As already mentioned (Example~\ref{ex:pascal-odd}), $p$ is
  $2$-recursive with exponents~$M=1$ and~$m=0$ and index shift bounds $\ell=0$ and $u=1$ as well as
  \begin{equation}
    \label{eq:pascal-odd-recs}
    p(2n) = \colorbox{SpringGreen}{3}p(n) + \colorbox{LimeGreen}{0}p(n+1) \qq{and} p(2n+1) = \colorbox{Goldenrod}{2}p(n) + \colorbox{Dandelion}{1}p(n+1)
  \end{equation}
  for all $n \geq 0$. Due to Theorem~\ref{thm:remark-2.1-general}, $p$ is also
  $2$-regular (with offset $n_{1} = 0$) and a $2$-regular representation of~$p$
  can be found as follows. We have $\ell' = 0$ and $u' = 2$, and it is due to the
  relation $m = M - 1 = 0$ that the vector~$v$ only consists of one block, namely
  \begin{equation*}
    v = v_{M-1} = v_{0} =
    \begin{pmatrix}
      p\\
      p\circ(n\mapsto n + 1)\\
      p\circ(n\mapsto n + 2)
    \end{pmatrix}.
  \end{equation*}
  Moreover, we can determine the matrices~$A_{0}$ and~$A_{1}$ in various ways:
  By~\eqref{eq:Ar-explicitly}, these matrices are
  \begin{equation*}
    A_{0} =
    \begin{pmatrix}
      \colorbox{SpringGreen}{3} & \colorbox{LimeGreen}{0} & 0\\
      \colorbox{Goldenrod}{2} & \colorbox{Dandelion}{1} & 0\\
      0 & \colorbox{SpringGreen}{3} & \colorbox{LimeGreen}{0}
    \end{pmatrix}
    \qq{and}
    A_{1} =
    \begin{pmatrix}
      \colorbox{Goldenrod}{2} & \colorbox{Dandelion}{1} & 0\\
      0 & \colorbox{SpringGreen}{3} & \colorbox{LimeGreen}{0}\\
      0 & \colorbox{Goldenrod}{2} & \colorbox{Dandelion}{1}
    \end{pmatrix}.
  \end{equation*}
  However, these matrices can also be obtained in an ad hoc fashion, namely by
  inserting $2n$ and $2n + 1$ into~$v$ and then
  component-wise applying~\eqref{eq:pascal-odd-recs}. For example, let us take a look at
  the third row of~$A_{0}$: We have to consider the third component
  of~$v$, which is $p\circ(n\mapsto n+2)$. We insert $2n$, which
  results in $p\circ(n\mapsto 2n+2)$, and we obtain
  \begin{equation*}
    p\circ(n\mapsto 2n+2) = p\circ(n\mapsto 2(n+1)) =
    \colorbox{SpringGreen}{3}p\circ(n\mapsto n + 1) +
    \colorbox{LimeGreen}{0}p\circ(n\mapsto n + 2)
  \end{equation*}
  by~\eqref{eq:pascal-odd-recs}. Thus, we have a~$3$ in the second column,
  because $p\circ(n\mapsto n + 1)$ is the second component of~$v$, and a~$0$ in
  the third column, because $p\circ(n\mapsto n + 2)$ is the third component
  of~$v$. Generally speaking, the rows of~$A_{r}$ that correspond to the last
  block~$v_{M-1}$ always consist of shifted copies of the coefficients in the
  recurrence relations.

  The ``step'' between the second row and the third row of $A_{0}$ and between
  the first row and the second row of $A_{1}$ is caused by the following fact:
  After inserting $2n$ or $2n + 1$, it can happen that the remainder
  is too large to apply the given recurrence relations directly. For instance, this was the
  case when determining the third row of~$A_{0}$ above: After inserting~$2n$, we have
  obtained $p\circ(n\mapsto 2n+2)$, and we had to rewrite this to
  $p\circ(n\mapsto 2(n+1))$ to be able to
  apply~\eqref{eq:pascal-odd-recs}. This increase of the argument by~$1$ causes
  the shift of the entries in the matrix to the right by~$1$.  For a more detailed
  description of this effect, we refer to the two different cases in Part~3 of
  the proof of Theorem~\ref{thm:remark-2.1-general} in
  Section~\ref{sec:proof-main-result} and to~\eqref{eq:Ar-explicitly}.
  
  Note that the dimension of this linear representation is not minimal since
  the sequence $p\circ (n\mapsto n+2)$ can be omitted. This is
  due to the following two facts: The third columns of~$A_{0}$ and~$A_{1}$ correspond
  to $p\circ (n\mapsto n+2)$. All non-zero elements of these columns are in the last row, which again
  corresponds to $p\circ (n\mapsto n+2)$. This reduction is possible because
  the coefficient of $p(n+1)$ in the left recurrence relation
  of~\eqref{eq:pascal-odd-recs} is zero.
\end{example}

\begin{example}
  Consider the $2$-recursive sequence~$x$ with exponents $M = 3$ and
  $m = 1$ given by the recurrence relations
  \begin{equation}
    \label{eq:artificial-1}
    \begin{split}
      x(8n) &= -\phantom{0}1x(2n - 1) + \phantom{0}0x(2n) + \phantom{0}1x(2n + 1),\\
      x(8n + 1) &= -11x(2n - 1) + 10x(2n) + 11x(2n + 1),\\
      x(8n + 2) &= -21x(2n - 1) + 20x(2n) + 21x(2n + 1),\\
      x(8n + 3) &= -31x(2n - 1) + 30x(2n) + 31x(2n + 1),\\
      x(8n + 4) &= -41x(2n - 1) + 40x(2n) + 41x(2n + 1),\\
      x(8n + 5) &= -51x(2n - 1) + 50x(2n) + 51x(2n + 1),\\
      x(8n + 6) &= -61x(2n - 1) + 60x(2n) + 61x(2n + 1),\\
      x(8n + 7) &= -71x(2n - 1) + 70x(2n) + 71x(2n + 1)
    \end{split}
  \end{equation}
  for all~$n \geq 0$. So for the sake of recognition, the
  coefficients $(c_{s,k})_{0\le s<8, -1\le k\le 1}$ are given by
  $c_{s,k} = (-1)^{\iverson{k < 0}}10s + k$. The index shift bounds
  of~$x$ are $\ell = -1$ and $u = 1$, and its offset is $n_{0} =
  0$. With the notation of Theorem~\ref{thm:remark-2.1-general}, we
  further find $\ell' = -3$ and $u' = 3$.

  Due to Theorem~\ref{thm:remark-2.1-general}, $x$ is $2$-regular with offset
  $n_{1} = 1$, and by~\eqref{eq:recursive-components-of-v} and~\eqref{eq:Ar-explicitly}, a $2$-linear representation with offset $n_{1} = 1$ of~$x$
  is given by $(A_{0}, A_{1}, v)$ with
    \begin{equation*}
      v =
      {\small
      \begin{pmatrix}
        x\\
        x\circ(n\mapsto 2n - 3)\\
        x\circ(n\mapsto 2n - 2)\\
        \vdots\\
        x\circ(n\mapsto 2n + 3)\\
        x\circ(n\mapsto 4n - 3)\\
        x\circ(n\mapsto 4n - 2)\\
        \vdots\\
        x\circ(n\mapsto 4n + 5)\\
      \end{pmatrix}
      {\color{gray}
        \begin{matrix}
          \coolrightbrace{x}{v_{0}}\\
          \coolrightbrace{x\\ x\\ \vdots\\ x}{v_{1}}\\
          \coolrightbrace{x\\ x\\ \vdots\\ x}{v_{2}}
        \end{matrix}}}
    \end{equation*}
  as well as
  \begin{equation*}
    A_{0} =
    \left(\begin{smallmatrix}
        0 & 0 & 0 & 0 & 1 & 0 & 0 & 0 & 0 & 0 & 0 & 0 & 0 & 0 & 0 & 0 & 0 \\
        0 & 0 & 0 & 0 & 0 & 0 & 0 & 0 & 1 & 0 & 0 & 0 & 0 & 0 & 0 & 0 & 0 \\
        0 & 0 & 0 & 0 & 0 & 0 & 0 & 0 & 0 & 1 & 0 & 0 & 0 & 0 & 0 & 0 & 0 \\
        0 & 0 & 0 & 0 & 0 & 0 & 0 & 0 & 0 & 0 & 1 & 0 & 0 & 0 & 0 & 0 & 0 \\
        0 & 0 & 0 & 0 & 0 & 0 & 0 & 0 & 0 & 0 & 0 & 1 & 0 & 0 & 0 & 0 & 0 \\
        0 & 0 & 0 & 0 & 0 & 0 & 0 & 0 & 0 & 0 & 0 & 0 & 1 & 0 & 0 & 0 & 0 \\
        0 & 0 & 0 & 0 & 0 & 0 & 0 & 0 & 0 & 0 & 0 & 0 & 0 & 1 & 0 & 0 & 0 \\
        0 & 0 & 0 & 0 & 0 & 0 & 0 & 0 & 0 & 0 & 0 & 0 & 0 & 0 & 1 & 0 & 0 \\
        0 & \tikzmark{l1}-51 & 50 & 51 & 0 & 0 & 0 & 0 & 0 & 0 & 0 & 0 & 0 & 0 & 0 & 0 & 0 \\
        0 & -61 & 60 & 61 & 0 & 0 & 0 & 0 & 0 & 0 & 0 & 0 & 0 & 0 & 0 & 0 & 0 \\
        0 & -71 & 70 & 71 & 0 & 0 & 0 & 0 & 0 & 0 & 0 & 0 & 0 & 0 & 0 & 0 & 0 \\
        0 & 0 & 0 & -1 & 0 & 1 & 0 & 0 & 0 & 0 & 0 & 0 & 0 & 0 & 0 & 0 & 0 \\
        0 & 0 & 0 & -11 & 10 & 11 & 0 & 0 & 0 & 0 & 0 & 0 & 0 & 0 & 0 & 0 & 0 \\
        0 & 0 & 0 & -21 & 20 & 21 & 0 & 0 & 0 & 0 & 0 & 0 & 0 & 0 & 0 & 0 & 0 \\
        0 & 0 & 0 & -31 & 30 & 31 & 0 & 0 & 0 & 0 & 0 & 0 & 0 & 0 & 0 & 0 & 0 \\
        0 & 0 & 0 & -41 & 40 & 41 & 0 & 0 & 0 & 0 & 0 & 0 & 0 & 0 & 0 & 0 & 0 \\
        0 & 0 & 0 & -51 & 50 & 51 & 0 & 0\tikzmark{r1} & 0 & 0 & 0 & 0 & 0 & 0 & 0 & 0 & 0
      \end{smallmatrix}\right)
    \DrawBox[ForestGreen, thick]{l1}{r1}
    \qq{and}
    A_{1} =
    \left(\begin{smallmatrix}
        0 & 0 & 0 & 0 & 0 & 1 & 0 & 0 & 0 & 0 & 0 & 0 & 0 & 0 & 0 & 0 & 0 \\
        0 & 0 & 0 & 0 & 0 & 0 & 0 & 0 & 0 & 0 & 1 & 0 & 0 & 0 & 0 & 0 & 0 \\
        0 & 0 & 0 & 0 & 0 & 0 & 0 & 0 & 0 & 0 & 0 & 1 & 0 & 0 & 0 & 0 & 0 \\
        0 & 0 & 0 & 0 & 0 & 0 & 0 & 0 & 0 & 0 & 0 & 0 & 1 & 0 & 0 & 0 & 0 \\
        0 & 0 & 0 & 0 & 0 & 0 & 0 & 0 & 0 & 0 & 0 & 0 & 0 & 1 & 0 & 0 & 0 \\
        0 & 0 & 0 & 0 & 0 & 0 & 0 & 0 & 0 & 0 & 0 & 0 & 0 & 0 & 1 & 0 & 0 \\
        0 & 0 & 0 & 0 & 0 & 0 & 0 & 0 & 0 & 0 & 0 & 0 & 0 & 0 & 0 & 1 & 0 \\
        0 & 0 & 0 & 0 & 0 & 0 & 0 & 0 & 0 & 0 & 0 & 0 & 0 & 0 & 0 & 0 & 1 \\
        0 & \tikzmark{l2}0 & 0 & -11 & 10 & 11 & 0 & 0 & 0 & 0 & 0 & 0 & 0 & 0 & 0 & 0 & 0 \\
        0 & 0 & 0 & -21 & 20 & 21 & 0 & 0 & 0 & 0 & 0 & 0 & 0 & 0 & 0 & 0 & 0 \\
        0 & 0 & 0 & -31 & 30 & 31 & 0 & 0 & 0 & 0 & 0 & 0 & 0 & 0 & 0 & 0 & 0 \\
        0 & 0 & 0 & -41 & 40 & 41 & 0 & 0 & 0 & 0 & 0 & 0 & 0 & 0 & 0 & 0 & 0 \\
        0 & 0 & 0 & -51 & 50 & 51 & 0 & 0 & 0 & 0 & 0 & 0 & 0 & 0 & 0 & 0 & 0 \\
        0 & 0 & 0 & -61 & 60 & 61 & 0 & 0 & 0 & 0 & 0 & 0 & 0 & 0 & 0 & 0 & 0 \\
        0 & 0 & 0 & -71 & 70 & 71 & 0 & 0 & 0 & 0 & 0 & 0 & 0 & 0 & 0 & 0 & 0 \\
        0 & 0 & 0 & 0 & 0 & -1 & 0 & 1 & 0 & 0 & 0 & 0 & 0 & 0 & 0 & 0 & 0 \\
        0 & 0 & 0 & 0 & 0 & -11 & 10 & 11\tikzmark{r2} & 0 & 0 & 0 & 0 & 0 & 0 & 0 & 0 & 0
      \end{smallmatrix}\right).
    \DrawBox[Maroon, thick]{l2}{r2}
  \end{equation*}
  Again, the matrices can also be obtained ad hoc, by inserting $2n$ and $2n + 1$ into the
  components and, if needed, component-wise applying the relations
  of~\eqref{eq:artificial-1}. For example, the fourth row of~$A_{1}$
  corresponds to $x\circ(n\mapsto 2n-1)$, i.e., the fourth component
  of~$v$. Inserting $2n + 1$ yields
  $x\circ(n\mapsto 2(2n + 1) -1)= x\circ(n\mapsto 4n + 1)$, which itself is the
  $13$th component of~$v$. Thus, we have a~$1$ in the $13$th column in the
  fourth row of~$A_{1}$.

  The ``interesting'' part of the matrices~$A_{0}$ and $A_{1}$ is given by entries in rows corresponding
  to~$v_{M-1} = v_{2}$ and columns corresponding to $v_{m} = v_{1}$. It is
  marked by the green and red boxes, respectively, and the entries can be obtained
  exactly as described in the previous example. Here the application
  of~\eqref{eq:artificial-1} is indeed needed and again leads to a block of
  shifted copies of the coefficients in the recurrence relations. Also here,
  one can see the ``steps'' in the matrices that were described in
  Example~\ref{ex:odd-pascal-after-main-result}.
\end{example}

Up to now, we have reduced $q$-recursive sequences to $q$-regular
sequences with some offset. Next, we get rid of this offset; Allouche and Shallit implicitly do such an offset correction for offset~$1$ in the proof of~\cite[Lemma~4.1]{Allouche-Shallit:1992:regular-sequences}.

\begin{theorem}
  \label{thm:offset-correction}
  Let $x$ be a $q$-regular sequence with offset~$n_{0}$, and let
  $(A_{0},\ldots,A_{q-1}, v)$ be a $q$-linear
  representation with offset~$n_{0}$ of~$x$. 
  Then~$x$ is $q$-regular and a $q$-linear representation $(\widetilde{A}_{0}, \dots,
  \widetilde{A}_{q-1}, \widetilde{v})$ of~$x$ is
  given as follows:
  \begin{enumerate}[(a)]
    \item The vector-valued sequence~$\widetilde{v}$ is given in block form by
    \begin{equation}
      \label{eq:offset-v}
      \widetilde{v} =
      \begin{pmatrix}
        v\\
        \delta_{0}\\
        \vdots\\
        \delta_{n_{0}-1}
      \end{pmatrix},
    \end{equation}
    where $\delta_{k}\colon \N_{0}\to\C$ is defined by
    $\delta_{k}(n) = \iverson{n = k}$ for all $0\leq k < n_{0}$ and $n\ge 0$.
  
  \item Let~$D\in\N$ be the dimension
    of~$v$. Moreover, for $0\leq r < q$ and $0 \leq k < n_{0}$, let\footnote{We
      set $x(n)=0$ for $n < 0$ to ensure that all vectors are defined.
      Other definitions of values for negative arguments are possible, but
      would result in other values for $W_{r}$.} $w_{r,k}
    \coloneqq v(qk + r) - A_{r}v(k)\in\C^{D}$, and let $W_{r}$ be the $D\times n_{0}$
    matrix which has columns $w_{r,0}$, \ldots, $w_{r,n_{0}-1}$.  Then
    for all $0\leq r < q$, the matrix $\widetilde{A}_{r}$ is given in block form by
  \begin{equation}
    \label{eq:matrices-A-tilde}
    \widetilde{A}_{r} =
    \begin{pmatrix}
      A_{r} & W_{r}\\
      0 & J_{r}
    \end{pmatrix},
  \end{equation}
  where $J_{r}\in\set{0,1}^{n_{0}\times n_{0}}$ is the matrix defined by
  \begin{equation}
    \label{eq:matrix-J-corr-n0}
    J_{r} \coloneqq \bigl(\iverson{jq = k - r}\bigr)_{\substack{0\leq k < n_{0}\\ 0\leq j < n_{0}}}.
  \end{equation}
  The matrix $J_r$ is a lower triangular matrix with diagonal
  \begin{equation*}
    \diag(J_{r}) = \bigl(\iverson{r = 0}, 0, \dots, 0\bigr).
  \end{equation*}
\end{enumerate}
\end{theorem}

\begin{corollary}\label{corollary:main:is-q-regular}
  Every $q$-recursive sequence~$x$ with any offset is $q$-regular and a $q$-linear representation of~$x$ is
  given as the combination of the explicit constructions of the $q$-linear representations from
  Theorem~\ref{thm:remark-2.1-general} and Theorem~\ref{thm:offset-correction}.
\end{corollary}

While Section~\ref{chap:lin-rep} up to this point (in particular Definition~\ref{definition:q-reqursive-sequence}) 
considered homogeneous recursive sequences, 
also inhomogeneities can occur. An example is, as already mentioned, the binary sum of digits,
where the constant sequence appears. In the following corollary,
we deal with such inhomogeneous recursive sequences.

\begin{corollary}
  \label{cor:inhomogeneities}
  Let $q\geq 2$, $M > m \geq 0$, $\ell\leq u$ and $n_{0}\geq \max\set{-\ell/q^{m},0}$ be
  fixed integers.
  Furthermore, let $x$ be a
  sequence such that for all $0\leq s < q^{M}$
  there exist $q$-regular sequences $g_{s}$ and constants
  $c_{s,k}\in\C$ for $\ell\leq k \leq u$ with
  \begin{equation}\label{eq:recursions-with-inhomogeneities}
    x(q^{M}n + s) = \smashoperator{\sum_{\ell\leq k\leq u}} c_{s,k}\,x(q^{m}n + k) + g_{s}(n)
  \end{equation}
  for all $n\geq n_{0}$.

  Then~$x$ is $q$-regular and a $q$-linear representation of~$x$ can
  be constructed straightforwardly by combining the explicit constructions of the
  $q$-linear representations from Theorem~\ref{thm:remark-2.1-general}
  and Theorem~\ref{thm:offset-correction} with $q$-linear
  representations of shifted versions of the sequences~$g_{s}$.
\end{corollary}

\begin{remark}
  The construction of a $q$-linear representation of a $q$-recursive sequence (given by
  recurrence relations as in~\eqref{eq:def-q-recursive} or in~\eqref{eq:recursions-with-inhomogeneities}) with offset
  has been included~\cite{Krenn-Lipnik:2019:trac-regular-sequence-from-recurrences} in SageMath~\cite{SageMath:2021:9.4}.
\end{remark}

\subsection{Reduction to $q$-Regular Sequences in a Special Case}
\label{sec:recursive-special}

We now study a specific case of $q$-recursive sequences,
namely $q$-recursive sequences with exponents~$M=m+1$ and~$m$ and index shift bounds
$\ell=0$ and $u=q^{m}-1$ for some $m\in\N_{0}$. The study of this case
is well-motivated: First of all, it will turn out in
Sections~\ref{sec:pascal} and~\ref{sec:ub} that this choice of parameters is quite natural,
i.e., we will see examples where subsequences of indices modulo~$q^{m+1}$
equal linear combinations of subsequences of indices modulo~$q^{m}$. Moreover, we can give the matrices of the linear
representation in a simpler form than in
Theorem~\ref{thm:remark-2.1-general}, and the upper bound~$u'$ can be
improved significantly. Finally, we show that the asymptotics of the
summatory functions of this special case of sequences can
be obtained directly from the recurrence relations
in~\eqref{eq:def-q-recursive}, without knowing a linear representation of the
sequence explicitly.

Note that in this section we assume the offset to be $n_{0} = 0$, mainly for
the sake of readability. However, we want to emphasize that all results can
be stated for arbitrary offset~$n_{0}\in\N_{0}$ as well, using
Theorem~\ref{thm:offset-correction}.

We start by giving an analogon of Theorem~\ref{thm:remark-2.1-general}
for our special case.

\begin{theorem}
  \label{prop:recursive-special-case}
  Let $x$ be a $q$-recursive sequence with exponents~$M=m+1$ and~$m$ and index shift bounds $\ell=0$ and
  $u=q^{m} - 1$ and coefficients $(c_{s,k})_{0\leq s < q^{m + 1}, 0\leq k < q^m}$.
  We define the matrices
  \begin{equation}
    \label{eq:prop-recursive-Bs}
    B_{r} = (c_{rq^m+d,k})_{\substack{0 \leq d < r q^{m} \\ 0\leq k < q^{m}}}
  \end{equation}
  for $0 \leq r < q$. Then $x$ is $q$-regular and a $q$-linear
  representation~$(A_{0}, \dots, A_{q-1}, v)$ of~$x$ is given as follows:

  \begin{enumerate}[(a)]
  \item The vector-valued sequence~$v$ is given in block form by
  \begin{equation}
    \label{eq:prop-recursive-v}
    v =
    \begin{pmatrix}
      v_{0}\\
      \vdots\\
      v_{m}
    \end{pmatrix},
  \end{equation}
  where for
  $0\leq j \le m$, the block $v_{j}$ has the form
    \begin{equation}
      \label{eq:prop-recursive-vj}
      v_{j} =
      \begin{pmatrix}
        x\circ(n\mapsto q^{j}n)\\
        \vdots\\
        x\circ(n\mapsto q^{j}n + q^{j} - 1)
      \end{pmatrix}.
    \end{equation}
  \item For $0 \leq r < q$, the matrices $A_{r}$ are given in block
  form by
  \begin{equation}
    \label{eq:prop-recursive-matrices}
    A_{r} =
    \begin{pmatrix}
      J_{r0} & J_{r1}\\
      0 & B_{r}
    \end{pmatrix},
  \end{equation}
  where $J_{r0}\in\set{0,1}^{\frac{q^{m} - 1}{q - 1}\times \frac{q^{m} - 1}{q - 1}}$ and
  $J_{r1}\in\set{0,1}^{\frac{q^{m} - 1}{q - 1}\times q^{m}}$. Furthermore, for $0\leq r < q$, the
  matrices~$J_{r0}$ are upper triangular matrices with
  zeros on the diagonal, and the matrices $J_{r0}$ and $J_{r1}$ are
  given explicitly by the first
  case of~\eqref{eq:Ar-explicitly} (with $u'$ replaced by $q^m-1$).
\end{enumerate}
\end{theorem}

\begin{remark}\makeanchor\label{rem:result-special-case}
  \begin{enumerate}
  \item The structure of~$v$ is the same as in
    Theorem~\ref{thm:remark-2.1-general}. In particular, the
    blocks~$v_{j}$ with $0\leq j< m$ coincide with the blocks~$v_{j}$
    from Theorem~\ref{thm:remark-2.1-general} given
    in~\eqref{eq:recursive-components-of-v-1}.
  \item The matrices $J_{r0}$ and $J_{r1}$ can be decomposed into
    blocks of identity matrices and zero matrices of smaller dimensions, which
    are horizontally shifted depending on~$r$. For an illustration we refer to
    Example~\ref{ex:artificial-example-special-case}.
  \item The last component of $v$ is $x\circ(n\mapsto q^mn+q^m-1)$ in contrast
    to $x\circ(n\mapsto q^mn+u')$ when using Theorem~\ref{thm:remark-2.1-general}.
    This means that using Theorem~\ref{prop:recursive-special-case} leads to a linear
    representation whose dimension is $\frac{q^{m+1} - q}{q-1}$ less than the dimension
    achieved by Theorem~\ref{thm:remark-2.1-general}.
  \item In the case $m = 0$, only rather special sequences can be handled by 
    Theorem~\ref{prop:recursive-special-case}. For instance, for $q=2$ and $m=0$,
    only sequences of the form $x(n) = x(0)a^{s(n)}$,    
    where~$s(n)$ is the binary sum of digits of~$n$ and~$a$ is some constant,
    fulfill the assumptions of this theorem. For all other $q$-recursive
    sequences with $m = 0$, Theorem~\ref{thm:remark-2.1-general} has to be
    used.
    \qedhere
  \end{enumerate}
\end{remark}

The following example will allow us to illustrate
Theorem~\ref{prop:recursive-special-case}. For the sake of simplicity, we again choose an artificial example.

\begin{example}
  \label{ex:artificial-example-special-case}
  Let us study the $2$-recursive sequence~$x$ with exponents~$M=3$ and $m=2$ given by the
  recurrence relations
  \begin{align*}
    f(8n) &= f(4n) + f(4n+1) + f(4n+2) + f(4n+3),\\
    f(8n + 1) &= f(4n) + f(4n+1) + f(4n+2) + f(4n+3),\\
    f(8n + 2) &= f(4n) + f(4n+1) + f(4n+2) + f(4n+3),\\
    f(8n + 3) &= f(4n) + f(4n+1) + f(4n+2) + f(4n+3),\\
    f(8n + 4) &= 2f(4n) + 2f(4n+1) + 2f(4n+2) + 2f(4n+3),\\
    f(8n + 5) &= 2f(4n) + 2f(4n+1) + 2f(4n+2) + 2f(4n+3),\\
    f(8n + 6) &= 2f(4n) + 2f(4n+1) + 2f(4n+2) + 2f(4n+3),\\
    f(8n + 7) &= 2f(4n) + 2f(4n+1) + 2f(4n+2) + 2f(4n+3)
  \end{align*}
  for all $n \geq 0$. Then we have
  \begin{equation*}
    B_{0} =
    \begin{pmatrix}
      \tikzmark{l1}1 & 1 & 1 & 1\\
      1 & 1 & 1 & 1\\
      1 & 1 & 1 & 1\\
      1 & 1 & 1 & 1\tikzmark{r1}\\
    \end{pmatrix}
    \qq{and} B_{1} =
    \begin{pmatrix}
      \tikzmark{l2}2 & 2 & 2 & 2\\
      2 & 2 & 2 & 2\\
      2 & 2 & 2 & 2\\
      2 & 2 & 2 & 2\tikzmark{r2}\\
    \end{pmatrix}.
    \DrawBox[ForestGreen, thick]{l1}{r1}
    \DrawBox[Maroon, thick]{l2}{r2}
  \end{equation*}

  By Theorem~\ref{prop:recursive-special-case}, $x$ is $2$-regular
  and a $2$-linear representation~$(A_{0}, A_{1}, v)$ of~$x$ is given
  by
  \begin{equation*}
    v =
    \begin{pmatrix}
      x\\
      x\circ(n\mapsto 2n)\\
      x\circ(n\mapsto 2n + 1)\\
      x\circ(n\mapsto 4n)\\
      x\circ(n\mapsto 4n+1)\\
      x\circ(n\mapsto 4n+2)\\
      x\circ(n\mapsto 4n+3)
    \end{pmatrix}
  \end{equation*}
  as well as
  \begingroup
  \renewcommand*{\arraystretch}{0.9}
  \setlength\arraycolsep{4pt} 
  \setcounter{MaxMatrixCols}{30}
  \begin{equation*}
    A_{0}=
    \begin{pmatrix}
      \tikzmark{l3}0 & \tikzmark{l5}1\tikzmark{r5} & 0 & \tikzmark{l4}0 & 0 & 0 & 0\\
      0 & 0 & 0 & \tikzmark{l6}1 & 0 & 0 & 0\\
      0 & 0 & 0\tikzmark{r3} & 0 & 1\tikzmark{r6} & 0 & 0\tikzmark{r4}\\
      0 & 0 & 0 & \tikzmark{left}1 & 1 & 1 & 1\\
      0 & 0 & 0 & 1 & 1 & 1 & 1\\
      0 & 0 & 0 & 1 & 1 & 1 & 1\\
      0 & 0 & 0 & 1 & 1 & 1 & 1\tikzmark{right}\\
    \end{pmatrix}
    \DrawBox[ForestGreen, thick]{left}{right}
    \DrawBox[lightgray, thick]{l5}{r5}
    \DrawBox[lightgray, thick]{l6}{r6}
    \DrawBox[gray, thick]{l3}{r3}
    \DrawBox[gray, thick]{l4}{r4}
    \qq{and} A_{1} =
    \begin{pmatrix}
      \tikzmark{l1}0 & 0 & \tikzmark{l4}1\tikzmark{r4} & \tikzmark{l2}0 & 0 & 0 & 0\\
      0 & 0 & 0 & 0 & 0 & \tikzmark{l5}1 & 0\\
      0 & 0 & 0\tikzmark{r1} & 0 & 0 & 0 & 1\tikzmark{r2}\\
      0 & 0 & 0 & \tikzmark{l3}2 & 2 & 2 & 2\\
      0 & 0 & 0 & 2 & 2 & 2 & 2\\
      0 & 0 & 0 & 2 & 2 & 2 & 2\\
      0 & 0 & 0 & 2 & 2 & 2 & 2\tikzmark{r3}\\
    \end{pmatrix}.
    \DrawBox[lightgray, thick]{l4}{r4}
    \DrawBox[lightgray, thick]{l5}{r2}
    \DrawBox[gray, thick]{l1}{r1}
    \DrawBox[gray, thick]{l2}{r2}
    \DrawBox[Maroon, thick]{l3}{r3}
  \end{equation*}
  \endgroup The dark gray boxes mark the matrices $J_{r0}$ and $J_{r1}$,
  whereas the smaller, light gray boxes mark the shifted identity matrices
  mentioned in Remark~\ref{rem:result-special-case}.
\end{example}

\section{Asymptotics}
\label{sec:asymp}

We want to study the asymptotic behavior for $q$-recursive sequences
(or, to be precise, of their summatory functions). As we have already
seen that such sequences are $q$-regular, we can apply the results of
\cite{Heuberger-Krenn:2018:asy-regular-sequences}. This is indeed what
we do, however, our set-up here is more specific than $q$-regular
sequences in general, because the sequences are given by particular
recurrence relations. This leads to more specific results here.

We start by briefly discussing the growth of matrix products, in
particular in conjunction with the joint spectral radius. This is one
important quantity determining the asymptotics of a sequence. Beside
that, the eigenvalues of the sum of matrices of a $q$-linear
representation play an important role.

Again, we will distinguish between the general case and the special
case introduced in Section~\ref{sec:recursive-special}.

\subsection{Growth of Matrix Products}
\label{sec:growth-of-matrix-products}

Before presenting previous results and adapting them to our purposes, we recall
the notion of the joint spectral radius
and introduce some related notions.

We fix a vector norm~$\norm{\,\cdot\,}$ on $\C^{D}$ and consider its
induced matrix norm.

\begin{definition}Let $\calG$ be a finite set of $D\times
  D$ matrices over $\C$.

  \begin{enumerate}[(a)]
  \item The joint spectral radius of~$\calG$ is defined as
    \begin{equation*}
      \rho(\calG)\coloneqq \lim_{k\to\infty} \sup\{\norm{G_{1}\ldots
        G_{k}}^{1/k}\mid G_1,\ldots, G_k\in\calG\}.
    \end{equation*}
  \item We say that $\calG$ has the \emph{finiteness property} if there exists a
    $k\in\N$ such that
    \begin{equation*}
      \rho(\calG)=\sup\{\norm{G_{1}\ldots
        G_{k}}^{1/k}\mid G_1,\ldots, G_k\in\calG\}.
    \end{equation*}
  \item We say that $\calG$ has the \emph{simple growth property} if
    \begin{equation*}
      \norm{G_1\ldots G_k}=O(\rho(\calG)^k)
    \end{equation*}
    holds for all $G_1$, \ldots, $G_k\in\calG$ and $k\to\infty$.
  \end{enumerate}
\end{definition}

\begin{remark}\makeanchor\label{remark:joint-spectral-radius-etc}
  \begin{enumerate}
  \item In the definition of the joint spectral radius, the limit can be
    replaced by an infimum over all $k \geq 1$; see Rota and Strang~\cite{Rota-Strang:1960}
    and also~\cite[Section~7.2]{Heuberger-Krenn:2018:asy-regular-sequences}. In
    particular, the limit in the definition of the joint spectral radius always exists.
  \item As any two norms on $\C^{D\times D}$ are equivalent, the definitions of
    the joint spectral radius and the simple growth property do not depend on
    the chosen norm. The finiteness property, however, depends on the chosen
    norm; see
    Remark~\ref{remark:unbordered-factors-finiteness-property} for an example.
  \item The finiteness property implies the simple growth property; see
    \cite[Section~7.2]{Heuberger-Krenn:2018:asy-regular-sequences}.
  \item The set
    \begin{equation*}
      \calG\coloneqq \left\{ \begin{pmatrix}1&1\\0&1\end{pmatrix}\right\}
    \end{equation*}
    has joint spectral radius $1$, but not the simple growth property, because the
    $k$th power of the only matrix in $\calG$ equals
    \begin{equation*}
      \begin{pmatrix}1&1\\0&1\end{pmatrix}^k = \begin{pmatrix}1&k\\0&1\end{pmatrix}.
      \qedhere
    \end{equation*}
  \end{enumerate}
\end{remark}

In Lemma~\ref{lemma:simple-growth-property-block-triangular-matrices}, we will
study sufficient conditions under which sets of block triangular matrices have the
simple growth property.

\subsection{Asymptotics for Regular Sequences}
\label{sec:rec-sec-asymp}

In order to obtain the asymptotics for the summatory function of $q$-recursive
sequences, we now apply a result of the first two authors of this article on the
asymptotic behavior of $q$-regular
sequences~\cite[Theorem~A]{Heuberger-Krenn:2018:asy-regular-sequences}. So
let~$x$ be a $q$-regular sequence with $q$-linear representation
$(A_{0}, \dots, A_{q-1}, v)$, and set
\begin{equation*}
  C\coloneqq \smashoperator{\sum_{0\leq r< q}}A_{r}.
\end{equation*}

For a square matrix~$G$, let $\sigma(G)$ denote the set of eigenvalues of~$G$
and by $m_{G}(\lambda)$ the size of the largest Jordan block of~$G$
associated with some~$\lambda\in\C$. In particular, we
have $m_{G}(\lambda) = 0$ if $\lambda\notin\sigma(G)$.

 Then we choose $R > 0$ as follows: If the set
$\mathcal{A} = \set{A_{0},\dots,A_{q-1}}$ has the simple growth property,
then we set $R=\rho(\mathcal{A})$. Otherwise, we choose
$R > \rho(\mathcal{A})$ such that there is no
eigenvalue~$\lambda\in\sigma(C)$ with $\rho(\mathcal{A}) < \abs{\lambda}\leq
R$. Furthermore, we let
\begin{equation*}
  \mathcal{X}(s) = \sum_{n\geq 1}n^{-s}x(n) \qq{and} \mathcal{V}(s) = \sum_{n\geq 1}n^{-s}v(n)
\end{equation*}
denote the Dirichlet series corresponding to~$x$ and~$v$.
Now we are ready to state the result.

\begin{theorem}[Asymptotic Analysis of $q$-Regular Sequences~{\cite[Theorem~A]{Heuberger-Krenn:2018:asy-regular-sequences}}]
  \label{thm:asymp}
  With the notations above, we have%
  \footnote{We let $\fractional{z} \coloneqq z - \floor{z}$ denote the
    fractional part of a real number~$z$.}
  \begin{multline}\label{eq:main-asymptotic-expansion}
    X(N) = \smashoperator{\sum_{0\leq n < N}} x(n) = \smashoperator{\sum_{\substack{\lambda\in\sigma(C)\\ \abs{\lambda} > R}}} N^{\log_{q}\lambda}
    \quad\smashoperator{\sum_{0\leq k < m_{C}(\lambda)}}\quad\frac{(\log N)^{k}}{k!}\ \Phi_{\lambda k}(\fractional{\log_{q}N})\\
    + O\bigl(N^{\log_{q}R}(\log N)^{\max\set{m_{C}(\lambda)\colon \abs{\lambda} = R}}\bigr)
  \end{multline}
  as $N\to\infty$, where $\Phi_{\lambda k}$ are suitable $1$-periodic functions. If
  there are no eigenvalues $\lambda\in\sigma(C)$ with
  $\abs{\lambda} \le R$, the $O$-term can be omitted.

  For $\abs{\lambda} > R$ and $0\leq k < m_{C}(\lambda)$, the
  function~$\Phi_{\lambda k}$ is Hölder continuous with any exponent
  smaller than~$\log_{q}(\abs{\lambda}/R)$.

  The Dirichlet series $\mathcal{V}(s)$ converges absolutely and
  uniformly on compact subsets of the half plane
  $\Re s > \log_{q}R + 1$ and can be continued to a meromorphic
  function on the half plane $\Re s > \log_{q}R$. It satisfies the
  functional equation
  \begin{equation}
    \label{eq:dirichlet-equation}
    (I - q^{-s}C)\mathcal{V}(s) = \smashoperator{\sum_{1\leq n < q}} n^{-s}v(n) + q^{-s}\smashoperator{\sum_{0\leq r < q}}A_{r}
    \sum_{k\geq 1}\binom{-s}{k}\biggl(\frac{r}{q}\biggr)^{k}\mathcal{V}(s+k)
  \end{equation}
  for $\Re s > \log_{q}R$. The right-hand side
  of~\eqref{eq:dirichlet-equation} converges absolutely and uniformly
  on compact subsets of $\Re s > \log_{q}R$. In particular,
  $\mathcal{V}(s)$ can only have poles where $q^{s}\in\sigma(C)$.

  For $\lambda\in\sigma(C)$ with $\abs{\lambda} > R$ and $0\le k<m_{C}(\lambda)$, the Fourier series
  \begin{equation}
    \Phi_{\lambda k}(u) = \sum_{\mu\in\Z}\varphi_{\lambda k\mu}\exp(2\mu\pi iu)
  \end{equation}
  converges pointwise for $u\in\R$ where the Fourier
  coefficients~$\varphi_{\lambda k\mu}$ are given by the singular
  expansion
  \begin{equation}
    \label{eq:fourier-coeffs-singular-expansion}
    \frac{x(0) + \mathcal{X}(s)}{s}
    \asymp \sum_{\substack{\lambda\in\sigma(C)\\ \abs{\lambda} > R}} \;
    \sum_{\mu\in\Z}\; \sum_{0\leq k< m_C(\lambda)}
    \frac{\varphi_{\lambda k\mu}}{\bigl(s - \log_{q}\lambda - \frac{2\mu\pi i}{\log q}\bigr)^{k+1}}
  \end{equation}
  for $\Re s > \log_{q}R$.  
\end{theorem}

\begin{remark}
  \begin{enumerate}
  \item \cite[Theorem~A]{Heuberger-Krenn:2018:asy-regular-sequences}
    only uses the
    simple growth property implicitly; the full details are contained
    in~\cite[Section~6]{Heuberger-Krenn:2018:asy-regular-sequences}. Note that there,
    the only property of the joint spectral radius used
    is \cite[Equation~(7.1)]{Heuberger-Krenn:2018:asy-regular-sequences}.
  \item The given expressions for the Fourier coefficients allow their
    computation with high precision;
    see~\cite[Part~IV]{Heuberger-Krenn:2018:asy-regular-sequences}. Furthermore,
    an implementation is available
    at \url{https://gitlab.com/dakrenn/regular-sequence-fluctuations}. We
    will use this implementation to compute the Fourier coefficients
    for the examples in Sections~\ref{sec:stern-brocot},~\ref{sec:pascal} and~\ref{sec:ub}.
  \item The motivation for analyzing the summatory function instead of
    the sequence itself is the following: The asymptotic behavior of
    regular sequences is often not smooth (which would imply that in any asymptotic expansion as given in~\cite{Heuberger-Krenn:2018:asy-regular-sequences}, everything is absorbed by the error term), whereas the asymptotic
    behavior of the summatory function is.

    However, it is also possible to apply Theorem~\ref{thm:asymp} to a
    $q$-regular sequence~$x$ itself: Let us write
    \begin{equation*}
      x(N) = x(0) + \smashoperator{\sum_{0\leq n < N}}\bigl(x(n+1) - x(n)\bigr).
    \end{equation*}
    So $x$ can be represented as the summatory function of the
    sequence of differences
    \begin{equation*}
      f(n) \coloneqq x(n+1) - x(n),
    \end{equation*}
    which is again $q$-regular by~\cite[Theorems~2.5 and~2.6]{Allouche-Shallit:1992:regular-sequences}. Consequently,
    applying Theorem~\ref{thm:asymp} to~$f$ yields an asymptotic
    analysis for
    \begin{equation*}
      F(N) = \smashoperator{\sum_{0\leq n< N}}f(n) = x(N) - x(0),
    \end{equation*}
    which differs from the asymptotic behavior of~$x$ only by an additive constant.
    \qedhere
  \end{enumerate}
\end{remark}

\subsection{Spectral Results in the General Case}
In this section, we show that in most cases, the asymptotic behavior of a
regular sequence can be deduced directly from a linear representation which is valid
from some offset $n_0\ge 1 > 0$. In these cases, it is not necessary to use Theorem~\ref{thm:offset-correction}
to construct an augmented linear representation valid for all non-negative integers.
So, we will assume that $n_0\ge 1$ because otherwise, there is nothing to do.

We first consider the significant eigenvalues and then
the significant joint spectral radii (significant with respect to
Theorem~\ref{thm:asymp}).

\begin{proposition}\label{cor:n0-same-eigenvalues}
  Let $A_{0}$, \dots, $A_{q-1}$, $\widetilde{A}_{0}$, \dots,
  $\widetilde{A}_{q-1}$ and $n_{0}$ as in
  Theorem~\ref{thm:offset-correction}. Assume that $n_0\ge 1$. Set
  \begin{equation*}
    C \coloneqq \smashoperator{\sum_{0\leq r < q}}A_{r} \qq{and}
    \widetilde{C} \coloneqq \smashoperator{\sum_{0\leq r < q}}\widetilde{A}_{r}.
  \end{equation*}
  Then $\sigma(\widetilde{C}) \subseteq \sigma(C) \cup \set{0, 1}$
  holds. In particular,
  \begin{enumerate}[(a)]
  \item if $n_{0} = 1$, then
    $\sigma(\widetilde{C}) = \sigma(C) \cup \set{1}$ and for all $\lambda\in\C\setminus\set{1}$, we have $m_C(\lambda) = m_{\widetilde{C}}(\lambda)$; and
  \item if $n_{0} \geq 2$, then
    $\sigma(\widetilde{C}) = \sigma(C) \cup \set{0, 1}$ and for all $\lambda\in\C\setminus\set{0,1}$, we have $m_C(\lambda) = m_{\widetilde{C}}(\lambda)$.
  \end{enumerate}
\end{proposition}

Before stating the second result, we state a lemma dealing with the simple
growth property for sets of block triangular matrices. This is a refinement
of Jungers~\cite[Proposition~1.5]{Jungers:2009:joint-spectral-radius}, which deals with
the joint spectral radius only (here restated as the first statement of the lemma).

\begin{lemma}
  \label{lemma:simple-growth-property-block-triangular-matrices}
  Let $\calG$ be a finite set of $(D_1+D_2+\cdots +D_s)\times (D_1+D_2+\cdots +
  D_s)$ block upper triangular
  matrices. For $G\in\calG$ write
  \begin{equation*}
    G =
    \begin{pmatrix}
      G^{(11)}& G^{(12)}&\ldots&G^{(1s)}\\
      0& G^{(22)}&\ldots&G^{(2s)}\\
      \vdots&\vdots&\ddots&\vdots\\
      0& 0&\ldots& G^{(ss)}
    \end{pmatrix}
  \end{equation*}
  where the block~$G^{(ij)}$ is a $D_i\times D_j$ matrix for $1\le i\le j\le s$.
  Set $\calG^{(i)}\coloneqq \{G^{(ii)} \mid G\in\calG\}$.
  Then $\rho(\calG)=\max_{1\le i\le s}\rho(\calG^{(i)})$.

  If there is a unique
  $i_0\in\{1, \ldots, s\}$ such that $\rho(\calG^{(i_0)})=\rho(\calG)$ and
  $\calG^{(i_0)}$ has the simple growth property, then $\calG$ has the simple
  growth property.
\end{lemma}

We now state the result on the joint spectral radius in the context of
Theorem~\ref{thm:offset-correction}.

\begin{proposition}
  \label{prop:jsr-A}
  Let
  $\mathcal{A} \coloneqq \set{A_{0},\dots, A_{q-1}}$,
  $\widetilde{\mathcal{A}} \coloneqq
      \set{\widetilde{A}_{0},\dots,\widetilde{A}_{q-1}}$ and
  $\mathcal{J} \coloneqq \set{J_{0}, \dots, J_{q-1}}$
    be the sets of matrices and $n_0$ the offset as given
  in Theorem~\ref{thm:offset-correction}, and assume $n_0\ge 1$. Then the joint spectral radii
  of~$\widetilde{\mathcal{A}}$ and~$\mathcal{J}$ satisfy
  \begin{equation}
    \label{eq:prop-jsr-A-1}
    \rho(\widetilde{\mathcal{A}}) = \max\set[\big]{\rho(\mathcal{A}), 1} \qq{and} \rho(\mathcal{J})= 1,
  \end{equation}
  respectively. In particular, if $\rho(\mathcal{A}) \geq 1$ holds, then we have
  $\rho(\widetilde{\mathcal{A}}) = \rho(\mathcal{A})$.

  Furthermore, if $\rho(\mathcal{A}) > 1$ holds and $\mathcal{A}$ has
  the simple growth property, then $\widetilde{\mathcal{A}}$ has the simple
  growth property.
\end{proposition}

Combining Propositions~\ref{cor:n0-same-eigenvalues} and~\ref{prop:jsr-A} with Theorem~\ref{thm:asymp} implies that
the asymptotics can also be determined by using the matrices~$A_{0}$,
\dots, $A_{q-1}$ (which do not contain the correction for the offset; see
Theorem~\ref{thm:offset-correction}) instead of the
matrices~$\widetilde{A}_{0}$, \ldots, $\widetilde{A}_{q-1}$ from the linear
representation.

Note that if $\rho(\mathcal{A})<1$, then the error
in~\eqref{eq:main-asymptotic-expansion} is $o(1)$. This implies that
adding constants (created by correction terms if the recurrence
relation is not valid for some $n\ge 0$) is visible in the asymptotic
expansion.

\subsection{Spectral Results in the Special Case}

Next, we are interested in the eigenvalues of the matrix
$C = \sum_{0\leq r < q}A_{r}$ for the special case. It turns out that
the eigenvalues of~$C$ can be obtained from the recurrence
relations~\eqref{eq:def-q-recursive} more directly than via the detour to
linear representations.

Note that also here we assume the offset to be $n_{0} = 0$ for the sake of
readability, analogous to Section~\ref{sec:recursive-special}. The following
results can be generalized easily for arbitrary offset.

\begin{proposition}
  \label{cor:sprectrum-C-Bs}
  Let $A_{0}$, \ldots, $A_{q-1}$ and $B_{0}$, \ldots, $B_{q-1}$ be the matrices
  as given in Theorem~\ref{prop:recursive-special-case}, let~$M = m+1$ and~$m$ be the
  exponents of the corresponding $q$-recursive sequence with $m \geq 1$ and set
  $C = \sum_{0\leq r < q}A_{r}$. Then we have
  \begin{equation*}
    \sigma(C) = \sigma(B_{0} + \cdots + B_{q-1})\cup\set{0}.
  \end{equation*}
  Moreover, we have
  $m_{C}(\lambda) = m_{B_{0} + \cdots + B_{q-1}}(\lambda)$
  for all $\lambda\in\C\setminus\set{0}$.
\end{proposition}

\begin{proposition}
  \label{cor:jsr-special-case}
  Let
  $\mathcal{A} \coloneqq \set{A_{0}, \dots, A_{q-1}}$,
  $\mathcal{J} \coloneqq \set{J_{00}, \dots, J_{(q-1)0}}$ and
  $\mathcal{B} \coloneqq \set{B_{0}, \dots, B_{q-1}}$
  be the sets of matrices as given in
  Theorem~\ref{prop:recursive-special-case}. Then the joint spectral radii
  of~$\mathcal{A}$ and~$\mathcal{J}$ satisfy
  \begin{equation*}
    \rho(\mathcal{A}) = \rho(\mathcal{B}) \qq{and}
    \rho(\mathcal{J}) = 0,
  \end{equation*}
  respectively.

  Furthermore, if $\rho(\mathcal{B}) > 0$ holds and $\mathcal{B}$ has
  the simple growth property, then $\mathcal{A}$ has the simple
  growth property.
\end{proposition}

The two propositions of this section provide the possibility to obtain the asymptotics of
the summatory function without knowing a linear representation of the sequence;
the asymptotics are fully determined by the matrices~$B_{0}$, \ldots, $B_{q-1}$.

\subsection{Functional Equation for the Dirichlet Series in the Special Case}
Theorem~\ref{thm:asymp} indicates that functional equations for the Dirichlet
series corresponding to the sequence of interest are essential for
computing Fourier coefficients of the periodic fluctuations. We now give a
variant for our special case of a $q$-recursive sequence which does not require constructing the $q$-linear representation first.

\begin{proposition}
  \label{prop:dirichlet}
  Let~$x$ be a $q$-recursive sequence with exponents~$M=m+1$ and~$m$, index shift bounds
  $\ell = 0$ and $u = q^{m} - 1$ and coefficients $(c_{j,k})_{0\leq j < q^{m+1}, 0\leq k < q^{m}}$, and
  let~$B_{0}$, \dots, $B_{q-1}$ be the matrices introduced
  in~\eqref{eq:prop-recursive-Bs}.
  Let $\rho > 0$ be such that
  $x(n) = O(n^{\log_{q}R})$ as $n \to \infty$ holds for all $R > \rho$, and let $\eta \geq 1$ be an integer. We define the Dirichlet
  series
  \begin{equation}
    \label{eq:dirichlet-special-case}
    \mathcal{X}_{j}(s) \coloneqq \smashoperator{\sum_{n \geq \eta}}\ \frac{x(q^{m}n + j)}{(q^{m}n + j)^{s}}
    = \ \smashoperator{\sum_{n\geq q^{m}\eta + j}}\ \frac{x(n)\iverson{n\equiv j \tpmod{q^{m}}}}{n^{s}}
  \end{equation}
  for $0\leq j < q^{m}$ and $\Re s>\log_q \rho + 1$ and set
  \begin{equation*}
    \mathcal{X}(s) \coloneqq
    \begin{pmatrix}
      \mathcal{X}_{0}(s)\\
      \vdots\\
      \mathcal{X}_{q^{m}-1}(s)
    \end{pmatrix}.
  \end{equation*}
  Then the functional equation
  \begin{equation}
    \label{eq:prop-functional-equation}
    \bigl(I - q^{-s}(B_{0} + \cdots + B_{q-1})\bigr)\mathcal{X}(s) =
    \begin{pmatrix}
      \mathcal{Y}_{0}(s)\\
      \vdots\\
      \mathcal{Y}_{q^{m}-1}(s)
    \end{pmatrix}
  \end{equation}
  holds for $\Re s > \log_{q}\rho$ with
  \begin{equation}
    \label{prop:dirichlet-error}
    \mathcal{Y}_{j}(s) = q^{-s}\sum_{k=0}^{q^{m}-1}\Biggl(\sum_{n\geq 1}\binom{-s}{n}
    \biggl(\sum_{\mu=0}^{q-1}c_{\mu q^{m} + j,k}\Bigl(\frac{\mu q^{m} + j}{q} - k\Bigr)^{n}\biggr)\mathcal{X}_{k}(s+n)\Biggr)
    +\ \smashoperator[l]{\sum_{\eta\leq n < q\eta}}\frac{x(q^{m}n + j)}{(q^{m}n + j)^{s}}.
  \end{equation}
  Moreover, $\mathcal{Y}_{j}(s)$ is analytic for $\Re s>\log_q \rho$ and
  all $0\leq j < q^{m}$, and, in particular,
  $\mathcal{X}(s)$ is meromorphic for $\Re s>\log_q \rho$ and can
  only have poles~$s$ where $q^{s}\in\sigma(B_{0} + \cdots + B_{q-1})$.
\end{proposition}



\section{Stern's Diatomic Sequence}
\label{sec:stern-brocot}

\subsection{Introduction of the Sequence}
We start our detailed study of particular sequences
by studying a sequence which has a long history, namely the
so-called\footnote{Note that this sequence has a strong connection to
  Stern--Brocot trees, which were first discovered independently by
  Stern~\cite{Stern:1858:zahlentheoretische-funktion} and
  Brocot~\cite{Brocot:1862:calcul}. This is the reason
  that Stern's diatomic sequence is also known as \emph{Stern--Brocot
    sequence.}}  \emph{Stern's diatomic sequence};
see~\oeis{A002487}. After
its earliest introduction by
Stern~\cite{Stern:1858:zahlentheoretische-funktion} in 1858, the
sequence has been studied thoroughly; see Northshield~\cite{Northshield:stern-diatomic} and the references therein, and also
Coons and Shallit~\cite{Coons-Shallit:2011:pattern-sequence-approach-Stern},
Leroy, Rigo and Stipulanti~\cite{Leroy-Rigo-Stipulanti:2017:digital-sequences-exotic}.

Stern's diatomic sequence $d$ is defined
by\footnote{%
  Strictly speaking, $d(0) = 0$ follows from \eqref{eq:sb-def:odd} for $n = 0$.}
$d(0) = 0$, $d(1) = 1$ and
\begin{subequations}
    \label{eq:sb-def}
  \begin{align}
    d(2n) &= d(n), \label{eq:sb-def:even}\\
    d(2n + 1) &= d(n) + d(n+1) \label{eq:sb-def:odd}
  \end{align}
\end{subequations}
for all $n \geq 0$. The first few terms of~$d$ are given in
Table~\ref{tab:sb}.

\begin{table}[htbp]
  \centering
  \begin{tabular}{c|cccccccccccccccccccc}
    $n$ & $0$ & $1$ & $2$ & $3$ & $4$ & $5$ & $6$ & $7$ & $8$ & $9$ & $10$ & $11$ & $12$ & $13$ & $14$ & $15$\\\hline
    $d(n)$ & $0$ & $1$ & $1$ & $2$ & $1$ & $3$ & $2$ & $3$ & $1$ & $4$ & $3$ & $5$ & $2$ & $5$ & $3$ & $4$
  \end{tabular}
  \caption{First few elements of Stern's diatomic sequence~$d$}
  \label{tab:sb}
\end{table}

\subsection{Combinatorial Interpretations of the Sequence}

There are several combinatorial interpretations of Stern's diatomic
sequence. In the following, we give a short overview of the most
interesting connections to combinatorial settings.
\begin{enumerate}
\item Let us call the word $d_{L-1}\ldots d_{0}$ over the alphabet
  $\set{0,1,2}$ a \emph{hyperbinary representation} of some~$n\in\N_{0}$ if
  $n = \sum_{0\leq i<L}d_{i}2^{i}$ and $d_{L-1}\neq 0$. Then the
  number of different hyperbinary representations of~$n$ is equal
  to~$d(n+1)$ for all $n \geq 0$;
  see Northshield~\cite[Theorem~3.1]{Northshield:stern-diatomic}.
\item Let $\stirling{n}{r}$ denote the \emph{Stirling partition numbers},
  i.e., $\stirling{n}{r}$ is the number of different partitions of the
  set~$\set{1,\dots,n}$ in exactly $r$~non-empty
  subsets. Then~$d(n)$ equals the number of integers $r\in\N_{0}$ such that
  $\stirling{n}{2r}$ is even and non-zero; see Carlitz~\cite{Carlitz:1964:partitions-stirling}.
\item Let $F(n)$ be the $n$th \emph{Fibonacci number}. Then~$d(n)$ is equal
  to the number of different representations of~$n$ as a sum of
  distinct Fibonacci numbers~$F(2k)$ with $k\in\N_{0}$;
  see Bicknell-Johnson~\cite{Bicknell:2003:stern-diatomic-fibonacci}.
\item An \emph{alternating bit set} of some integer~$n\in\N_{0}$ is a subset~$A$ of
  the positions in the binary expansion of~$n$ such that
  \begin{itemize}
  \item the bits of the binary expansion of~$n$ at positions which lie
    in~$A$ are alternating between~$1$ and~$0$,
  \item the most significant bit at a position which lies in~$A$ is
    a~$1$, and
  \item the least significant bit at a position which lies in~$A$ is
    a~$0$.
  \end{itemize}
  In particular, we also allow $A = \emptyset$ to be an alternating
  bit set. Note that this definition implies that every alternating
  bit set has even cardinality.

  Then the number of different alternating bit sets of~$n$ is
  equal to~$d(n+1)$; see Finch~\cite[Section~2.16.3]{Finch:constants:2003}.
\item There is a relation to the well-known \emph{Towers of Hanoi};
  see Hinz, Klav\v{z}ar, Milutinovi\'{c}, Parisse and
  Petr~\cite{Hinz-Klavzar-Milutinovic-Parisse-Petr:metric-properties-hanoi-stern-diatomic}.
\end{enumerate}

Thus, the asymptotic analysis of the summatory function of Stern's
diatomic sequence is indeed well-motivated, also from a combinatorial
point of view.

\subsection{Regularity and a Linear Representation}

In order to be able to apply Theorem~\ref{thm:asymp} for the asymptotic
analysis of the summatory function, the sequence~$d$ has to be
recursive. Due to the definition of the sequence in~\eqref{eq:sb-def},
it is clear that~$d$ is $2$-recursive with exponents~$M=1$ and~$m=0$,
index shift bounds $\ell = 0$ and $u = 1$, and offset $n_0=0$. Thus, it is also $2$-regular by
Theorem~\ref{thm:remark-2.1-general}.
Note that Theorem~\ref{prop:recursive-special-case} is not applicable:
The term $d(n+1)$ appears in~\eqref{eq:sb-def:odd} and therefore, the
upper index shift bound~$u$ needs to be at least $1$, whereas Theorem~\ref{prop:recursive-special-case}
only allows $0$ as an upper index shift bound in the case $m=0$.
So we use
Theorem~\ref{thm:remark-2.1-general} to obtain a $2$-linear
representation~$(A_{0}, A_{1}, v)$ of~$d$: The
vector-valued sequence~$v$ is given by
\begin{equation*}
  v =
  \begin{pmatrix}
    d\\
    d\circ(n\mapsto n+1)\\
    d\circ(n\mapsto n+2)
  \end{pmatrix},
\end{equation*}
and the matrices are given by
\begin{equation*}
  A_{0} =
  \begin{pmatrix}
    1 & 0 & 0\\
    1 & 1 & 0\\
    0 & 1 & 0
  \end{pmatrix}
  \qq{and}
  A_{1} =
  \begin{pmatrix}
    1 & 1 & 0\\
    0 & 1 & 0\\
    0 & 1 & 1
  \end{pmatrix}.
\end{equation*}
The correctness of the recurrence relations $v(2n) = A_{0}v(n)$ and
$v(2n + 1) = A_{1}v(n)$ for all $n \geq 0$ can easily be verified by
using~\eqref{eq:sb-def}.

As in Example~\ref{ex:odd-pascal-after-main-result},
we can see that $d\circ(n\mapsto n+2)$ is actually not needed in the
linear representation, which is due to the fact that the coefficient of
$d(n + 1)$ in the recurrence relation~\eqref{eq:sb-def:even} is
zero. This implies that
$(A_{0}',A_{1}',v')$ with
\begin{equation*}
  v' =
  \begin{pmatrix}
    d\\
    d\circ(n\mapsto n+1)
  \end{pmatrix}
\end{equation*}
as well as
\begin{equation*}
  A_{0}' =
  \begin{pmatrix}
    1 & 0\\
    1 & 1 
  \end{pmatrix}
  \qq{and}
  A_{1}' =
  \begin{pmatrix}
    1 & 1\\
    0 & 1
  \end{pmatrix}
\end{equation*}
is also a $2$-linear representation of~$d$.

By applying the minimization algorithm mentioned in
Remark~\ref{rem:main-thm}~(\ref{item:minimization}), we see that this is the
smallest possible $2$-linear representation of~$d$.

\subsection{Asymptotics}

Let $\mathcal{V}(s)$ denote the Dirichlet series corresponding to~$v'$,
i.e.,
\begin{equation*}
  \mathcal{V}(s) = \sum_{n\geq 1}n^{-s}v'(n),
\end{equation*}
and let
$C = A_{0}' + A_{1}'$. In the following theorem,
we state the main result of this section: We give an asymptotic formula
for the summatory function of~$d$ as well as a functional equation
for~$\mathcal{V}(s)$.

\begin{theorem}[Asymptotics for Stern's Diatomic Sequence]
  \label{thm:asymp-sb}
  The summatory function~$D$ of Stern's diatomic sequence~$d$ satisfies
  \begin{equation}
    \label{eq:asymp-sb}
    D(N) = \smashoperator{\sum_{0\leq n <N}}d(n) = N^{\kappa}\cdot\Phi_{D}(\fractional{\log_{2}N}) + O(N^{\log_{2}\varphi})
  \end{equation}
  as $N \to \infty$, where
  $\kappa = \log_{2}3 = 1.5849625007211\ldots$,
  $\varphi = \frac{1+\sqrt{5}}{2} = 1.6180339887498\ldots$ is the golden ratio,
  $\log_{2}\varphi = 0.69424191363061\ldots$ and
  $\Phi_{D}$ is a $1$-periodic continuous function which is
    Hölder continuous with any exponent smaller
    than~$\kappa-\log_{2}\varphi$. The Fourier coefficients of~$\Phi_{D}$
    can be computed efficiently.

  Furthermore, the Dirichlet series~$\mathcal{V}(s)$ satisfies the
  functional equation
  \begin{equation}
    \label{eq:sb-functional}
    (I - 2^{-s}C)\mathcal{V}(s) = v'(1) + 2^{-s}A_{1}'\sum_{k\geq 1}\frac{1}{2^{k}}\binom{-s}{k}\mathcal{V}(s+k)
  \end{equation}
  for all $\Re s > \log_{2}\varphi$. Both sides of Equation~\eqref{eq:sb-functional}
  are analytic for $\Re s > \log_{2}\varphi$, and, in particular,
  $\mathcal{V}(s)$ is meromorphic for $\Re s>\log_2\varphi$ and can
  only have at most simple poles
  $s = \log_{2}3 + \frac{2i\pi\mu}{\log 2}$ with $\mu\in\Z$.
\end{theorem}
Table~\ref{tab:sb-fourier} shows the first few Fourier coefficients and Figure~\ref{fig:sb-fluc} a plot of the periodic fluctuation of Theorem~\ref{thm:asymp-sb}.

\begin{table}[htbp]
  \centering
  \begin{tabular}{r|l}
    \multicolumn{1}{c|}{$\mu$} & \multicolumn{1}{c}{$\varphi_{\mu}$}\\\hline
    $0$ & $\phantom{-}0.5129922721107177789989881697483$\\
    $1$ & $-0.00572340619479957230984532582323 + 0.00692635056470554871320794780023i$\\
    $2$ & $\phantom{-}0.00024322678681282580951796870908 + 0.00296266191012688412725699259509i$\\
    $3$ & $-0.00145239145783579607592238228126 + 0.00117965322085442917547658711471i$\\
    $4$ & $\phantom{-}0.00111195666191700704424207971541 + 0.00018518355971470343780812186861i$\\
    $5$ & $-0.00046732929957426516792963051204 + 0.00050425058689999021735711128987i$\\
    $6$ & $-0.00044953390461558932213468137492 + 0.00048773649732968174101103217106i$\\
    $7$ & $\phantom{-}0.00036329328164895877338262637843 + 0.00035534416834062145852032394307i$\\
    $8$ & $-0.00016679033186561839463311958967 - 0.00043694014091729453542478927729i$\\
    $9$ & $\phantom{-}0.00030367683575578278808761185183 + 0.00009371153567156005005069054904i$\\
    $10$ & $-0.00009911479960205956796299031716 + 0.00019462735102739460438023334462i$\\
  \end{tabular}
  \caption{First few Fourier Coefficients of~$\Phi_{D}$}
  \label{tab:sb-fourier}
\end{table}

\begin{figure}[htbp]
  \centering
  \includegraphics{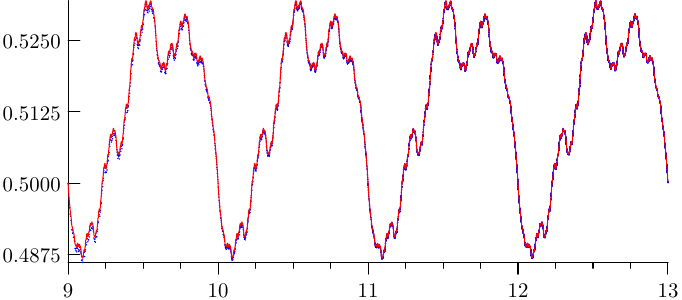}
  \caption{Fluctuation in the main term of the asymptotic expansion of
    the summatory function~$D$. The plot shows the periodic
    fluctuation~$\Phi_{D}(u)$ approximated by its Fourier series of
    degree~$2000$ (red) as well as the
    function~$D(2^{u})/2^{\kappa u}$ (blue).  }
  \label{fig:sb-fluc}
\end{figure}

\begin{proof}[Proof of Theorem~\ref{thm:asymp-sb}]
  We use Theorem~\ref{thm:asymp} with the linear
  representation~$(A_{0}', A_{1}',
  v')$ and work out the parameters needed in the
  theorem. Recall that $C = A_{0}' + A_{1}'$.

  \emph{Joint Spectral Radius.} We determine the joint spectral radius
  of $A_{0}'$ and $A_{1}'$.
  As one matrix is the transpose of the other, the spectral norm
  of each of them equals the square root of the dominant eigenvalue of their product. The maximal
  spectral norm of the matrices is an upper bound for the joint spectral
  radius; the square root of the dominant eigenvalue of their product is a
  lower bound for the joint spectral radius. As both bounds are equal,
  the joint spectral radius equals the spectral norm. It turns out that this
  spectral norm equals~$\varphi = \frac{1 + \sqrt{5}}{2}$.

  \emph{Finiteness Property.} The finiteness
  property for $A_{0}'$ and $A_{1}'$ is
  satisfied with respect to the spectral norm, which can be seen by considering exactly one
  factor~$A_{0}'$ or~$A_{1}'$. Thus, we choose
  $R = \varphi$.
  
  \emph{Eigenvalues.} The spectrum of~$C$ is given by
  $\sigma(C) = \set{1,3}$. Furthermore, it is clear that all
  eigenvalues are simple and thus, $m_{C}(3) = 1$.

  Applying Theorem~\ref{thm:asymp} yields the result.
\end{proof}

It will turn out during the proof of Theorem~\ref{thm:asymp-pascal-gen} that
a slight modification of the summatory function leads to an exact formula:
\begin{corollary}\label{corollary:Stern-Brocot-exact}
  With the notations of Theorem~\ref{thm:asymp-sb}, we have
\begin{equation*}
   \smashoperator{\sum_{0\leq n <N}}d(n) + \frac12 d(N) =
  N^{\kappa}\cdot\Phi_{D}(\fractional{\log_{2}N}).
\end{equation*}
\end{corollary}


\section{Number of Non-Zero Entries in a Generalized Pascal's Triangle}
\label{sec:pascal}

\subsection{Introduction of the Sequence}

The first two authors of this article have studied Pascal's rhombus as one possible
generalization of Pascal's triangle
in~\cite{Heuberger-Krenn:2018:asy-regular-sequences} as well as
in~\cite{Heuberger-Krenn-Prodinger:2018:pascal-rhombus} together with Prodinger. In particular,
they analyzed the asymptotic behavior of the number of odd entries in
the $n$th row of Pascal's rhombus.

Here, we consider a generalization of Pascal's triangle to binomial
coefficients of words. This generalization was first introduced by
Leroy, Rigo and Stipulanti
in~\cite{Leroy-Rigo-Stipulanti:2016:generalized-pascal-triangle}. We
in particular study the sequence counting the number of non-zero
elements in each row (see~\oeis{A007306} except for the initial value), which was investigated in detail by the same
authors
in~\cite{Leroy-Rigo-Stipulanti:2017:non-zero-generalized-pascal-triangle} and~\cite{Leroy-Rigo-Stipulanti:2018:counting-subword-occurences},
and provide an asymptotic result for the summatory function. Our result
coincides with the result
in~\cite{Leroy-Rigo-Stipulanti:2018:counting-subword-occurences}. In
contrast
to~\cite{Leroy-Rigo-Stipulanti:2018:counting-subword-occurences}, we
additionally provide the periodic fluctuation that occurs in the
asymptotics by determining its Fourier coefficients. This completes the full picture
of the summatory function.

We start with the following definition; also
see Lothaire~\cite[Chapter~6]{Lothaire:1997:comb-words} for more details on
binomial coefficients of words.

\begin{definition}[Scattered Subword, Binomial Coefficients of Words]
  Let $u = u_{1}\ldots u_{j}$ and $v = v_{1}\dots v_{k}$ be two words
  over the same alphabet.
  \begin{enumerate}[(a)]
  \item We say that $v$ is a \emph{scattered subword} of $u$ if there
    exists a strictly increasing mapping
    $\pi\colon\set{1,\dots,k}\to\set{1,\dots,j}$ with
    $u_{\pi(i)} = v_{i}$ for all~$1\leq i\leq k$.
    We call~$\pi$ an \emph{occurrence} of~$v$ as a scattered subword of~$u$.
  \item The \emph{binomial coefficient} of $u$ and $v$, denoted
    by~$\binom{u}{v}$, is defined as the number of different
    occurrences of~$v$ as a scattered subword of~$u$.
  \end{enumerate}
\end{definition}

For example, we consider the words
$u = u_{1}u_{2}u_{3}u_{4}u_{5}u_{6} = 110010$ and $v = v_{1}v_{2} = 10$
over the alphabet~$\set{0,1}$. Then we have
\begin{equation}
  \label{eq:example-binom-words}
  \binom{110010}{10} = 7
\end{equation}
because there are exactly seven possibilities to represent~$v$ as a
scattered subword of~$u$, namely
\begin{equation*}
  u_{1}u_{3} = u_{1}u_{4} = u_{1}u_{6} = u_{2}u_{3} = u_{2}u_{4} = u_{2}u_{6} = u_{5}u_{6} = v.
\end{equation*}

Note that the classical binomial coefficient for two
integers~$n$, $k\in\N_{0}$ can be obtained by the identity
\begin{equation*}
  \binom{1^{n}}{1^{k}} = \binom{n}{k},
\end{equation*}
where~$1^{n}$ denotes the word consisting of~$n$ ones.

Next, we define the \emph{generalized Pascal's triangle}~$\mathcal{P}_{2}$
as an infinite matrix as follows: The
entry in the $n$th row and $k$th column of~$\mathcal{P}_{2}$ is
given by $\binom{(n)_{2}}{(k)_{2}}$, where $(n)_{2}$ denotes
the binary expansion of some~$n\in\N_{0}$, i.e.,
\begin{equation*}
  \mathcal{P}_{2} \coloneqq
    \Biggl(\binom{(n)_{2}}{(k)_{2}}\Biggr)_{\substack{n\geq 0\\k\geq 0}}.
\end{equation*}

Observe that
$\binom{(n)_{2}}{(0)_{2}} = 1$ and $\binom{(n)_{2}}{(n)_{2}} = 1$
hold for all $n \geq 0$.  We let~$z$ denote the sequence of
interest and define~$z(n)$ as the number of non-zero elements
in the $n$th row of~$\mathcal{P}_{2}$. The first few values
of~$\mathcal{P}_{2}$ are given in Table~\ref{tab:pascal-gen}, and the
last column shows the first few values of~$z$. Figure~\ref{fig:gen-pascal-triangle-plot} illustrates the non-zero elements in~$\mathcal{P}_{2}$.

\begin{table}[htbp]
  \centering
  \begin{tabular}{c|c|ccccccccc||c}
    & $k$ & $0$ & $1$ & $2$ & $3$ & $4$ & $5$ & $6$ & $7$ & $8$\\\hline
    $n$ & \diagbox{$(n)_{2}$}{$(k)_{2}$} & $\varepsilon$ & $1$ & $10$ & $11$ & $100$ & $101$ & $110$ & $111$ & $1000$ & $z(n)$\\\hline
    $0$ & $\varepsilon$ & \textbf{1} & 0 & 0 & 0 & 0 & 0 & 0 & 0 & 0 & 1\\
    $1$ & $1$ & \textbf{1} & \textbf{1} & 0 & 0 & 0 & 0 & 0 & 0 & 0 & 2\\
    $2$ & $10$ & 1 & 1 & 1 & 0 & 0 & 0 & 0 & 0 & 0 & 3\\
    $3$ & $11$ & \textbf{1} & \textbf{2} & 0 & \textbf{1} & 0 & 0 & 0 & 0 & 0 & 3\\
    $4$ & $100$ & 1 & 1 & 2 & 0 & 1 & 0 & 0 & 0 & 0 & 4\\
    $5$ & $101$ & 1 & 2 & 1 & 1 & 0 & 1 & 0 & 0 & 0 & 5\\
    $6$ & $110$ & 1 & 2 & 2 & 1 & 0 & 0 & 1 & 0 & 0 & 5\\
    $7$ & $111$ & \textbf{1} & \textbf{3} & 0 & \textbf{3} & 0 & 0 & 0 & \textbf{1} & 0 & 4\\
    $8$ & $1000$ & 1 & 1 & 3 & 0 & 3 & 0 & 0 & 0 & 1 & 5\\
  \end{tabular}
  \caption{The first few elements of the generalized Pascal's
    triangle~$\mathcal{P}_2$ as well as the corresponding number of
    non-zero elements in each row. The values of the ordinary Pascal's
    triangle are printed in bold.}
  \label{tab:pascal-gen}
\end{table}

\begin{figure}[htbp]
  \centering
  \includegraphics[width=0.6\textwidth]{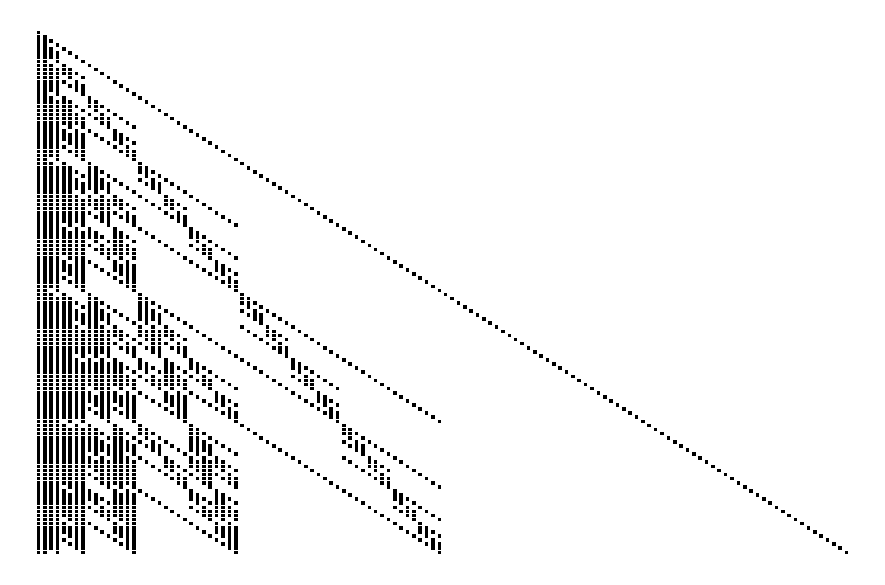}
  \caption{Non-zero elements in the generalized Pascal's triangle~$\mathcal{P}_2$}
  \label{fig:gen-pascal-triangle-plot}
\end{figure}

The following result by Leroy, Rigo and Stipulanti~\cite{Leroy-Rigo-Stipulanti:2017:non-zero-generalized-pascal-triangle} provides a (at least on the first glance) surprising connection between the number of
non-zero elements in~$\mathcal{P}_{2}$ and Stern's diatomic sequence.

\begin{theorem}[Leroy--Rigo--Stipulanti~{\cite[Section~4]{Leroy-Rigo-Stipulanti:2017:non-zero-generalized-pascal-triangle}}]
  \label{thm:connection-pascal-stern}
  The sequence $z$ satisfies the relation
  \begin{equation*}
    z(n) = d(2n + 1)
  \end{equation*}
  for all $n \geq 0$, where $d$ is Stern's diatomic sequence as
  studied in Section~\ref{sec:stern-brocot}.
\end{theorem}

\subsection{Regularity and a Linear Representation}
\label{sec:pascal-regular}

In principle, Theorem~\ref{thm:connection-pascal-stern} and the results on the
asymptotics of Stern's diatomic sequence given in
Section~\ref{sec:stern-brocot} could be used to determine the asymptotics
of~$z$. However, it turns out that the error term in the asymptotic expansion
of~$z$ vanishes. In order to show this, the results of
Section~\ref{sec:stern-brocot} are not sufficient, and we need to have a closer
look at the linear representation of~$z$. 
Theorem~\ref{thm:connection-pascal-stern} does not suffice for this purpose, so instead
we intend to present three different possibilities for
obtaining a linear representation. The first one will give some more details on
the reduction via Theorem~\ref{thm:connection-pascal-stern}, while the others
will be based on the following result, also by Leroy, Rigo and Stipulanti.

\begin{theorem}[Leroy--Rigo--Stipulanti~{\cite[Theorem~21]{Leroy-Rigo-Stipulanti:2017:non-zero-generalized-pascal-triangle}}]
  \label{thm:pascal-recursions}
  The sequence~$z$ satisfies the recurrence relations
  \begin{subequations}
    \label{eq:pascal-rec}
   \begin{align}
    \label{eq:pascal-rec-1}
    z(2n + 1) &= 3z(n) - z(2n),\\
    \label{eq:pascal-rec-2}
    z(4n) &=  - z(n) + 2z(2n),\\
    \label{eq:pascal-rec-3}
    z(4n + 2) &= 4z(n) - z(2n)
  \end{align}
  \end{subequations}
  for all $n \geq 0$.
\end{theorem}

As already mentioned, the previous theorem as well as
Theorem~\ref{thm:connection-pascal-stern} provide several possibilities to find
a linear representation of~$z$, and we discuss three of
them. As a side product of the second approach, it will also be clear why~$z$
is a recursive sequence and therefore fits into our framework.

\paragraph{Approach 1.}

  First of all, it is worth mentioning that we can use
  Theorem~\ref{thm:connection-pascal-stern} to obtain a $2$-linear representation:
  Since Stern's diatomic sequence~$d$ is $2$-regular and $z(n) = d(2n + 1)$
  holds for all $n\in\N_{0}$, the $2$-regularity of~$z$ follows
  by Allouche and Shallit~\cite[Theorem~2.6]{Allouche-Shallit:1992:regular-sequences}.  In the proof
  of~\cite[Theorem~2.6]{Allouche-Shallit:1992:regular-sequences} we also find a
  description for the construction of a $2$-linear representation of~$z$ by using
  the linear representation of~$d$. We do not want to go into detail here.

\paragraph{Approach 2.}

  The recurrence relations in Theorem~\ref{thm:pascal-recursions} are not
  directly in line with the desired relations in the framework of $q$-recursive
  sequences as given in~\eqref{eq:def-q-recursive}. This second
  approach does not only lead to a desired linear representation of~$z$, but it
  also illustrates how recurrence relations as given
  in~\eqref{eq:pascal-rec} can be disentangled in order to obtain appropriate
  recurrence relations for a $q$-recursive sequence. In fact, we will show that
  the sequence~$z$ is $2$-recursive, which directly implies that it is also
  $2$-regular due to Theorem~\ref{prop:recursive-special-case}.

  For this purpose, consider the system of equations
  \begin{equation}\label{eq:pascal-matrix-second-approach-1}
    \begin{pmatrix}
      -3 & 1 & 1 & 0 & 0 & 0 & 0 \\
      1 & -2 & 0 & 1 & 0 & 0 & 0 \\
      -4 & 1 & 0 & 0 & 0 & 1 & 0 \\
      0 & -3 & 0 & 1 & 1 & 0 & 0 \\
      0 & 0 & -3 & 0 & 0 & 1 & 1
    \end{pmatrix}
    \begin{pmatrix}
      z\\
      z\circ (n\mapsto 2n)\\
      z\circ (n\mapsto 2n + 1)\\
      z\circ (n\mapsto 4n)\\
      z\circ (n\mapsto 4n + 1)\\
      z\circ (n\mapsto 4n + 2)\\
      z\circ (n\mapsto 4n + 3)\\
    \end{pmatrix}
    = 0,
  \end{equation}
  where the first three
  rows correspond to the relations given in
  Theorem~\ref{thm:pascal-recursions} and the last two rows arise from~\eqref{eq:pascal-rec-1}
  by replacing~$n$ by $2n$ and by $2n+1$, respectively.

  We want to get a representation of~$z$ as a $2$-recursive sequence.
  It turns out that we can achieve such a sequence with exponents
  $M = 2$ and $m = 1$, so we need explicit expressions for
  $z\circ (n\mapsto 4n)$, $z\circ (n\mapsto 4n+1)$, $z\circ (n\mapsto 4n+2)$
  and $z\circ(n\mapsto 4n+3)$ (corresponding to the last four columns of the matrix). We also
  want these expressions to be free from~$z$ itself (corresponding to the first
  column of the matrix), so we transform the system in such a way that an identity matrix appears in these columns. Indeed, we
  multiply the system from the left with the inverse of the matrix formed by these five columns
  and obtain
  \begingroup
  \renewcommand{\arraystretch}{1.15}
  \begin{equation}\label{eq:pascal-matrix-second-approach-2}
    \begin{pmatrix}
      0 & \tikzmark{l1}-\frac{5}{3} & \frac{1}{3} & 1 & 0 & 0 & 0 \\
      0 & -\frac{4}{3} & -\frac{1}{3}\tikzmark{r1} & 0 & 1 & 0 & 0 \\
      0 & \tikzmark{l2}-\frac{1}{3} & -\frac{4}{3} & 0 & 0 & 1 & 0 \\
      0 & \frac{1}{3} & -\frac{5}{3}\tikzmark{r2} & 0 & 0 & 0 & 1\\
      1 & -\frac{1}{3} & -\frac{1}{3} & 0 & 0 & 0 & 0
    \end{pmatrix}
    \begin{pmatrix}
      z\\
      z\circ (n\mapsto 2n)\\
      z\circ (n\mapsto 2n + 1)\\
      z\circ (n\mapsto 4n)\\
      z\circ (n\mapsto 4n + 1)\\
      z\circ (n\mapsto 4n + 2)\\
      z\circ (n\mapsto 4n + 3)\\
    \end{pmatrix}
    = 0.
    \DrawBox[ForestGreen, thick][0.13em][0][-0.2em][0.2em]{l1}{r1}
    \DrawBox[Maroon, thick][0.13em][0][-0.2em][0.2em]{l2}{r2}
  \end{equation}
  \endgroup
  Here the first four rows give the system
  \begin{align*}
    z(4n) &= \frac{5}{3}z(2n) - \frac{1}{3}z(2n + 1),\\
    z(4n + 1) &= \frac{4}{3}z(2n) + \frac{1}{3}z(2n + 1),\\
    z(4n + 2) &= \frac{1}{3}z(2n) + \frac{4}{3}z(2n + 1),\\
    z(4n + 3) &= -\frac{1}{3}z(2n) + \frac{5}{3}z(2n + 1)
  \end{align*}
  for $n \geq 0$, which is a representation of~$z$ as a $2$-recursive sequence
  with offset~$n_{0} = 0$, exponents
  $M = 2$ and $m = 1$ and index shift bounds $\ell=0$ and $u=1$. The last row of~\eqref{eq:pascal-matrix-second-approach-2}
  can be omitted.

  The matrices~$B_{0}$ and~$B_{1}$ as introduced
  in~\eqref{eq:prop-recursive-Bs} are given by
  \begin{equation*}
    B_{0} = \frac{1}{3}
    \begin{pmatrix}
      \tikzmark{l1}5 & -1\\
      4 & 1\tikzmark{r1}
    \end{pmatrix}
    \qq{as well as}
    B_{1} = \frac{1}{3}
    \begin{pmatrix}
      \tikzmark{l2}1 & 4\\
      -1 & 5\tikzmark{r2}
    \end{pmatrix};
    \DrawBox[ForestGreen, thick][0][0.37em][0][0]{l1}{r1}
    \DrawBox[Maroon, thick][0][0][0][-0.4em]{l2}{r2}
  \end{equation*}
  see also the marked submatrices
  in~\eqref{eq:pascal-matrix-second-approach-2}. We can now apply
  Theorem~\ref{prop:recursive-special-case} and obtain a $2$-linear
  representation~$(A_{0}, A_{1}, v)$ of~$z$ with dimension~$3$: The
  vector-valued sequence~$v$ is given by
  \begin{equation*}
    v =
    \begin{pmatrix}
      z\\
      z\circ(n\mapsto 2n)\\
      z\circ(n\mapsto 2n + 1)
    \end{pmatrix},
  \end{equation*}
  and due to~\eqref{eq:prop-recursive-matrices}, the 
  matrices~$A_{0}$ and $A_{1}$ have the form
  \begin{equation*}
    A_{0} = \frac{1}{3}
    \begin{pmatrix}
      0 & 3 & 0\\
      0 & \tikzmark{l1}5 & -1\\
      0 & 4 & 1\tikzmark{r1}
    \end{pmatrix}
    \qq{and}
    A_{1} = \frac{1}{3}
    \begin{pmatrix}
      0 & 0 & 3\\
      0 & \tikzmark{l2}1 & 4\\
      0 & -1 & 5\tikzmark{r2}
    \end{pmatrix}.
    \DrawBox[ForestGreen, thick][0][0.37em][0][0]{l1}{r1}
    \DrawBox[Maroon, thick][0][0][0][-0.4em]{l2}{r2}
  \end{equation*}
  Note that by~\eqref{eq:pascal-rec-1} we know that the three sequences
  contained in $v$ are linearly dependent. This means that we could replace
  $z\circ(n\mapsto 2n + 1)$ by a linear combination of the other two sequences
  and obtain a linear representation of dimension~2.  However, we prefer to
  derive a linear representation with this vector directly as Approach~3 below.

\paragraph{Approach 3.}
  
  Finally, we can also derive a linear representation of~$z$
  directly from the recurrence relations given in
  Theorem~\ref{thm:pascal-recursions}. By setting
  \begin{equation}
    \label{eq:pascal-gen-v}
    v =
    \begin{pmatrix}
      z\\
      z\circ(n\mapsto 2n)
    \end{pmatrix},
  \end{equation}
  we obtain
  \begin{subequations}
    \label{eq:pascal-A0-A1}
    \begin{equation}
      \label{eq:pascal-A0}
    v(2n) =
    \begin{pmatrix}
      z(2n)\\
      z(4n)
    \end{pmatrix}
    =
    \begin{pmatrix}
      z(2n)\\
      - z(n) + 2z(2n)
    \end{pmatrix}
    = \underbrace{
      \begin{pmatrix}
        0 & 1\\
        -1 & 2
      \end{pmatrix}
    }_{\eqqcolon A_{0}}v(n)
  \end{equation}
  as well as
  \begin{equation}
    \label{eq:pascal-A1}
    v(2n + 1) =
    \begin{pmatrix}
      z(2n + 1)\\
      z(4n + 2)
    \end{pmatrix}
    =
    \begin{pmatrix}
      3z(n) - z(2n)\\
      4z(n) - z(2n)
    \end{pmatrix}
    = \underbrace{
      \begin{pmatrix}
        3 & -1\\
        4 & -1
      \end{pmatrix}
    }_{\eqqcolon A_{1}}v(n)
  \end{equation}
  \end{subequations}
  for all $n\in\N_{0}$. See
  also~\cite[Corollary~22]{Leroy-Rigo-Stipulanti:2017:non-zero-generalized-pascal-triangle}.

  By applying the minimization algorithm mentioned in
  Remark~\ref{rem:main-thm}~(\ref{item:minimization}), we see that this is the
  smallest possible $2$-linear representation of~$z$.

\subsection{Full Asymptotics}

We now come to the main theorem of this section: We provide
an explicit formula for the summatory function
$Z(N) = \sum_{0\leq n < N}z(n)$.

\begin{theorem}[Full Asymptotics for the Number of Non-Zero Elements in
  the generalized Pascal's triangle~$\mathcal{P}_{2}$]
  \label{thm:asymp-pascal-gen}
  The summatory function~$Z$ of the sequence~$z$ equals
  \begin{equation}
    \label{eq:asymp-pascal-gen}
    Z(N) = \smashoperator{\sum_{0\leq n< N}}z(n) = N^{\kappa}\cdot\Phi_{Z}(\fractional{\log_{2}N})
  \end{equation}
  for $N\ge 1$
  with $\kappa = \log_{2}3$ and $\Phi_{Z} = 2\Phi_{D}$, where $\Phi_{D}$ is the
  periodic fluctuation which occurs in the asymptotic expansion of Stern's
  diatomic sequence in Theorem~\ref{thm:asymp-sb}.
\end{theorem}

For a plot of~$\Phi_{Z}$ see Figure~\ref{fig:pascal-fluc}.

\begin{figure}[htbp]
  \centering
  \includegraphics{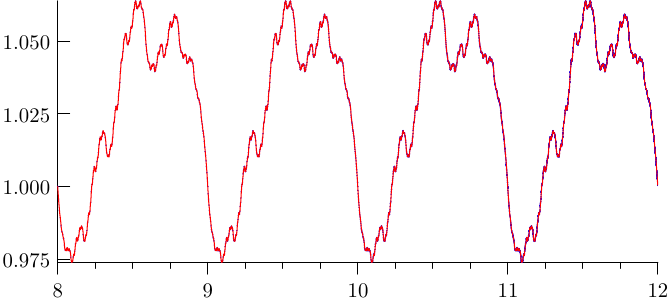}
  \caption{Fluctuation in the asymptotic expansion of
    the summatory function~$Z$. The plot shows the periodic
    fluctuation~$\Phi_{Z}(u) = 2\Phi_{D}(u)$ approximated by its Fourier polynomial of
    degree~$2000$ (red) as well as the
    function~$Z(2^{u})/2^{\kappa u}$ (blue).  }
  \label{fig:pascal-fluc}
\end{figure}

\begin{proof}[Proof of Theorem~\ref{thm:asymp-pascal-gen}]
  We use the $2$-dimensional linear representation $(A_{0}, A_{1}, v)$ given
  in~\eqref{eq:pascal-gen-v} and~\eqref{eq:pascal-A0-A1}, and we want to
  apply~\cite[Theorem~B]{Heuberger-Krenn:2018:asy-regular-sequences} instead of
  Theorem~\ref{thm:asymp} in order to get rid of the error term coming from
  Theorem~\ref{thm:asymp}.

  The left eigenvectors of the matrix
  \begin{equation*}
    C = A_{0} + A_{1} =
    \begin{pmatrix}
      3 & 0\\
      3 & 1
    \end{pmatrix}
  \end{equation*}
  are given by
  $v_{1} = (1, -2/3)$, which corresponds to the
    eigenvalue~$\lambda = 1$, and
  $v_{3} = (1, 0)$, which corresponds to the
    eigenvalue~$\lambda = 3$.
  Since the vector~$v_{3}$ is equal to the left ``choice
  vector''~$e_{1}$ of the linear representation (see~\eqref{eq:xn-lin-rep}), it
  follows from~\cite[Section~6.3]{Heuberger-Krenn:2018:asy-regular-sequences}
  that the eigenvalue~$\lambda = 3$ is the only one which contributes to
  the summatory function.

  By the above observations, the matrix~$C$ has the Jordan decomposition $C = T^{-1}JT$ with
  \begin{equation*}
    J =
    \begin{pmatrix}
      3 & 0\\
      0 & 1
    \end{pmatrix}
    \qq{and}
    T =
    \begin{pmatrix}
      1 & 0\\
      1 & -2/3
    \end{pmatrix}.
  \end{equation*}
  With the notation
  introduced in~\cite[Section~6.2]{Heuberger-Krenn:2018:asy-regular-sequences}, we
  moreover have
  \begin{equation*}
    D =
    \begin{pmatrix}
      1 & 0\\
      0 & 0
    \end{pmatrix} \qq{and}
    C' = T^{-1}DJT =
    \begin{pmatrix}
      3 & 0\\
      9/2 & 0
    \end{pmatrix}
  \end{equation*}
  as well as
  \begin{equation*}
    K = T^{-1}DT(I_{2} - C')^{-1}(I_{2} - A_{0}) =
    \begin{pmatrix}
      -1/2 & 1/2\\
      -3/4 & 3/4
    \end{pmatrix},
  \end{equation*}
  where $I_{2}$ denotes the $2\times 2$ identity matrix. Due to
  Theorem~\ref{thm:connection-pascal-stern} and the second paragraph of the proof of Theorem~\ref{thm:asymp-sb}, we can choose
  $R = \varphi = \frac{1 + \sqrt{5}}{2}$.  Now we are ready to
  apply~\cite[Theorem~B]{Heuberger-Krenn:2018:asy-regular-sequences} for the
  eigenvalue~$\lambda = 3$. By noting $Kv(0) = 0$, we obtain
  \begin{equation}
    \label{eq:formula-for-Z-proof}
    Z(N) = N^{\kappa}\cdot\Phi_{Z}(\fractional{\log_{2}N}),
  \end{equation}
  where~$\Phi_{Z}$ is a $1$-periodic function which is Hölder continuous with any
  exponent smaller than $\kappa-\log_{2}\varphi$.

  Finally, we argue that $\Phi_{Z} = 2\Phi_{D}$. Due to
  Theorem~\ref{thm:connection-pascal-stern} and \eqref{eq:sb-def:odd}, we have
  \begin{align}
    Z(N) &= \smashoperator{\sum_{0\leq n < N}}z(n)\notag\\
    &= \smashoperator{\sum_{0\leq n < N}} d(2n + 1)\notag\\
         &= \smashoperator{\sum_{0\leq n < N}}d(n) + \smashoperator{\sum_{0\leq n < N}}d(n+1)\notag\\
         &= 2D(N) + d(N)\label{eq:pascal-stern-exact-formula}\\
    &= 2D(N) + O(N^{\log_{2}\varphi})\notag
  \end{align}
  as $N \to \infty$. Applying Theorem~\ref{thm:asymp-sb} on $D(N)$ and combining the result
  with~\eqref{eq:formula-for-Z-proof} yields
  \begin{equation*}
    N^{\kappa}\cdot\Phi_{Z}(\fractional{\log_{2}N}) = N^{\kappa}\cdot 2\Phi_{D}(\fractional{\log_{2}N}) + O(N^{\log_{2}\varphi}).
  \end{equation*}
  Since the set $\setm{\fractional{\log_{2}N}}{N\in\N}$ is dense in the
  interval~$[0,1]$ and both $\Phi_{Z}$ and $2\Phi_{D}$ are continuous, they
  have to be equal. This concludes the proof.
\end{proof}

We are now able to prove Corollary~\ref{corollary:Stern-Brocot-exact}.
\begin{proof}[Proof of Corollary~\ref{corollary:Stern-Brocot-exact}]
  The statement of the corollary follows from
  combining~\eqref{eq:pascal-stern-exact-formula}
  with Theorem~\ref{thm:asymp-pascal-gen}.
\end{proof}



\section{Number of Unbordered Factors in the Thue--Morse Sequence}
\label{sec:ub}

\subsection{Introduction of the Sequence}

The Thue--Morse sequence~$t$ is the parity sequence of the
binary sum of digits; see \oeis{A010060}.
In this section, we study the number of
unbordered factors in the Thue--Morse sequence, which are a special
kind of subsequences of~$t$.  Unbordered factors of the Thue--Morse
sequence were first investigated by Currie and Saari
in~\cite{Currie-Saari:2009:least-periods-factors}. Some more
interesting insights concerning this special type of factors can, for
example, be found
in Charlier, Rampersad and Shallit~\cite{Charlier-Rampersad-Shallit:2012:enumeration-properties-automatic-seq}
and Go\v{c}, Henshall and Shallit~\cite{Goc-Henshall-Shallit:2013:automatic-theorem-proving}; see also
Go\v{c}, Rampersad, Rigo and Salimov~\cite{Goc-Rampersad-Rigo-Salimov:2014:abelian-bordered-words}.

Before we formally define the sequence of study, we give the following
definitions from combinatorics on words, which can, for example, also
be found in Lothaire~\cite{Lothaire:2002:algeb}.

\begin{definition}[Unbordered Factor]
  Let $\mathcal{A}$ be an alphabet and $x\colon\N_{0}\to\mathcal{A}$ a
  sequence.
  \begin{enumerate}
  \item For integers $0\leq i \leq j$, we let $x[i\twoldots j]$ denote
    the subword $x(i)x(i+1)\ldots x(j-1)x(j)$ of the infinite word
    $x(0)x(1)x(2)x(3)\dots$.
  \item We say that a word $w\in\mathcal{A}^{\star}$ is a
    \emph{factor of $x$} if there exist integers $0\leq i\leq j$
    such that $w = x[i\twoldots j]$.
  \item A word~$w\in\mathcal{A}^{\star}$ is said to be \emph{bordered}
    if there exists a non-empty word $v\neq w$ which is both a prefix
    and a suffix of~$w$. In this case we call~$v$ a \emph{border}
    of~$w$. If~$w$ is not bordered, then it is said to be
    \emph{unbordered}.
  \end{enumerate}
\end{definition}
In particular, this implies that the empty word~$\varepsilon$ as well
as every word of length~$1$ is unbordered. Moreover, a word $w = ab$
of length~$2$ with $a$, $b\in\mathcal{A}$ is unbordered if and only if
$a\neq b$.

We illustrate the previous definitions by the following example.

\begin{example}[Some (Un)Bordered Factors of the Thue--Morse Sequence]
  If we write the Thue--Morse sequence as an infinite word
  $t = t(0)t(1)t(2)\dots$, then it starts with
  \begin{equation*}
    t = 01101001\,10010110\,10010110\,01101001\dots.
  \end{equation*}
  We visually structured the sequence into blocks of eight letters
  in order to emphasize some of its properties.
  Some bordered factors of~$t$ are given in Table~\ref{tab:b-in-t},
  whereas the factors in Table~\ref{tab:ub-in-t} are unbordered.
  
  \begin{table}[htbp]
    \centering
    \begin{tabular}{l|c|c}
      bordered factor & border & length\\\hline
      $t[5\twoldots 6] = 00$ & $0$ & $2$\\
      $t[1\twoldots 2] = 11$ & $1$ & $2$\\
      $t[3\twoldots 5] = 010$ & $0$ & $3$\\
      $t[2\twoldots 4] = 101$ & $1$ & $3$\\
      $t[2\twoldots 5] = 1010$ & $10$ & $4$\\
      $t[0\twoldots 9] = 0110100110$ & $0110$ & $10$
    \end{tabular}
    \caption{Some bordered factors of the Thue--Morse sequence~$t$}
    \label{tab:b-in-t}
  \end{table}

  \begin{table}[htbp]
    \centering
    \begin{tabular}{l|c}
      unbordered factor & length\\\hline
      $\varepsilon$ & $0$\\
      $t[0\twoldots 0] = 0$ & $1$\\
      $t[1\twoldots 1] = 1$ & $1$\\
      $t[0\twoldots 1] = 01$ & $2$\\
      $t[2\twoldots 3] = 10$ & $2$\\
      $t[0\twoldots 2] = 011$ & $3$\\
      $t[1\twoldots 3] = 110$ & $3$\\
      $t[4\twoldots 6] = 100$ & $3$\\
      $t[5\twoldots 7] = 001$ & $3$\\
    \end{tabular}
    \caption{All unbordered factors of the Thue--Morse sequence~$t$ up to length~$3$}
    \label{tab:ub-in-t}
  \end{table}
  It is easy to check that Table~\ref{tab:ub-in-t} contains a complete
  list of the unbordered factors of the Thue--Morse sequence with
  length smaller than~$4$.
\end{example}

Now the following question may arise: Is there an unbordered factor
of~$t$ with length~$n$ for every $n \geq 0$? Currie and Saarie
showed in~\cite{Currie-Saari:2009:least-periods-factors} that~$t$ has
an unbordered factor of length~$n$ if~$n\not\equiv 1
\tpmod{6}$. However, since
\begin{equation*}
  t[39\twoldots 69] = 0011010010110100110010110100101
\end{equation*}
is an unbordered factor of length~$31$ in~$t$, the given condition is
not necessary. Go\v{c}, Henshall and Shallit finally proved the following
characterization.

\begin{theorem}[Go\v{c}, Henshall and Shallit~{\cite[Theorem~4]{Goc-Henshall-Shallit:2013:automatic-theorem-proving}}]
  \label{thm:f(n)=0}
  There is an unbordered factor of length~$n$ in~$t$ if and only if
  $(n)_{2}\notin 1(01^{*}0)^{*}10^{*}1$, where~$(n)_{2}$ denotes the
  binary digit expansion of~$n\in\N_{0}$ and $1(01^{*}0)^{*}10^{*}1$
  has to be considered as a regular expression.
\end{theorem}

In this article, we are interested in the number of unbordered factors of
length~$n$ in~$t$; cf.~\oeis{A309894}. We let $f$ denote this sequence.
The first few elements of this sequence are given
in Table~\ref{tab:ub}.

\begin{table}[htbp]
  \label{tab:ub-initial-values}
  \centering
  \begin{tabular}{c|cccccccccccccccccccccc}
    $n$ & $0$ & $1$ & $2$ & $3$ & $4$ & $5$ & $6$ & $7$ & $8$ & $9$ & $10$ & $11$ & $12$ & $13$ & $14$ & $15$ & \ldots & $23$\\\hline
    $f(n)$ & $1$ & $2$ & $2$ & $4$ & $2$ & $4$ & $6$ & $0$ & $4$ & $4$ & $4$ & $4$ & $12$ & $0$ & $4$ & $4$ & \ldots & $8$
  \end{tabular}
  \caption{First few elements of the sequence~$f$ which counts the number of unbordered factors in the Thue--Morse sequence of length~$n$}
  \label{tab:ub}
\end{table}

Theorem~\ref{thm:f(n)=0} characterizes the numbers~$n$ with $f(n) = 0$.
However, for our purpose, this is not satisfying
since we are interested in the numbers~$f(n)$ themselves. In particular, we
hope that~$f$ is regular and that we can determine a linear
representation for it, which consequently can be used to derive an asymptotic result for
the summatory function of~$f$.

Finally, before we come to the determination of a linear
representation of~$f$, the authors in~\cite{Goc-Mousavi-Shallit:2013}
also investigated the growth rate of~$f$, which is given as follows.

\begin{theorem}[Go\v{c}, Mousavi and Shallit~{\cite[Theorem~2]{Goc-Mousavi-Shallit:2013}}]
  \label{prop:ub-growth-f}
  The inequality $f(n) \leq n$ holds for all $n\geq 4$, with $f(n) = n$
  infinitely often. This implies
  \begin{equation*}
    \limsup_{n\geq 1}\frac{f(n)}{n} = 1.
  \end{equation*}
\end{theorem}

\subsection{Regularity and a Linear Representation}
\label{sec:ub-lin-rep}

Go\v{c}, Mousavi and Shallit came up with the following recurrence relations,
which will help us to conclude that~$f$ is $2$-regular and to
construct a $2$-linear representation of~$f$.

\begin{theorem}[Go\v{c}, Mousavi and Shallit~{\cite[Proof of Theorem~2]{Goc-Mousavi-Shallit:2013}}]
  \label{prop:ub-recs-goc}
  For the number~$f(n)$ of unbordered factors of length~$n$ in
  the Thue--Morse sequence, we have
  \begin{align*}
    f(4n) &= 2f(2n), &(n\geq 2)\\
    f(4n + 1) &= f(2n + 1), &(n\geq 0)\\
    f(8n + 2) &= f(2n + 1) + f(4n + 3), &(n\geq 1)\\
    f(8n + 3) &= -f(2n + 1) + f(4n + 2), &(n\geq 2)\\
    f(8n + 6) &= -f(2n + 1) + f(4n + 2) + f(4n + 3), &(n\geq 2)\\
    f(8n + 7) &= 2f(2n + 1) + f(4n + 3). &(n\geq 3)
  \end{align*}
\end{theorem}

This directly yields the following corollary.

\begin{corollary}
  \label{cor:ub-q-recursive}
  For the number~$f(n)$ of unbordered factors of length~$n$ in
  the Thue--Morse sequence, we have
  \begin{align*}
    f(8n) &= 2f(4n), &(n\geq 1)\\
    f(8n + 1) &= f(4n + 1),&(n\geq 0)\\
    f(8n + 2) &= f(4n + 1) + f(4n + 3), &(n\geq 1)\\
    f(8n + 3) &= -f(4n + 1) + f(4n + 2), &(n\geq 2)\\
    f(8n + 4) &= 2f(4n + 2), &(n\geq 1)\\
    f(8n + 5) &= f(4n + 3), &(n\geq 0)\\
    f(8n + 6) &= -f(4n + 1) + f(4n + 2) + f(4n + 3), &(n\geq 2)\\
    f(8n + 7) &= 2f(4n + 1) + f(4n + 3). &(n\geq 3)
  \end{align*}
  In particular, $f$ is $2$-recursive with offset~$n_{0}=3$, exponents~$M=3$
  and $m = 2$ and index shift bounds~$\ell = 0$ and $u=3$.
\end{corollary}

Together with the initial values given in
Table~\ref{tab:ub-initial-values}, the recurrence relations stated in
the previous corollary completely describe the sequence~$f$.

In order to obtain a $2$-linear representation of~$f$, we first of all
use Theorem~\ref{prop:recursive-special-case} (instead of
Theorem~\ref{thm:remark-2.1-general} which would lead to a linear
representation of larger dimension). This will give us
matrices~$A_{0}$ and~$A_{1}$ as well as a vector-valued sequence~$v$
such that the relations $v(2n) = A_{0}v(n)$ and $v(2n + 1) = A_{1}v(n)$
hold for all~$n\geq 3$, i.e., a $2$-linear representation with offset~$3$ of~$f$.

With the notation of Theorem~\ref{prop:recursive-special-case}, the matrices~$B_{0}$
and~$B_{1}$ are given by
\begin{equation}
  \label{eq:ub-system-matrices}
  B_{0} =
  \begin{pmatrix}
    \tikzmark{l}2 & 0 & 0 & 0\\
    0 & 1 & 0 & 0\\
    0 & 1 & 0 & 1\\
    0 & -1 & 1 & 0\tikzmark{r}
  \end{pmatrix}
  \qq{and} B_{1} =
  \begin{pmatrix}
    \tikzmark{l2}0 & 0 & 2 & 0\\
    0 & 0 & 0 & 1\\
    0 & -1 & 1 & 1\\
    0 & 2 & 0 & 1\tikzmark{r2}
  \end{pmatrix},
  \DrawBox[ForestGreen, thick]{l}{r}
  \DrawBox[Maroon, thick]{l2}{r2}
\end{equation}
where the relevant spectrum---relevant with respect to
Proposition~\ref{cor:sprectrum-C-Bs}---is
\begin{equation*}
  \sigma(B_{0}+B_{1}) = \set[\big]{-\sqrt{3} + 1, 1,2,\sqrt{3}+1},
\end{equation*}
and all eigenvalues are simple.

We now apply Theorem~\ref{prop:recursive-special-case} and obtain
\begin{equation*}
  v =
  \begin{pmatrix}
    f\\
    f\circ (n\mapsto 2n)\\
    f\circ (n\mapsto 2n+1)\\
    f\circ (n\mapsto 4n)\\
    f\circ (n\mapsto 4n+1)\\
    f\circ (n\mapsto 4n+2)\\
    f\circ (n\mapsto 4n+3)
  \end{pmatrix}
\end{equation*}
as well as the matrices
\setcounter{MaxMatrixCols}{20}
\begin{equation}
  \label{eq:ub-As}
  A_{0} =
  \begin{pmatrix}
    \tikzmark{l}0 & 1 & 0 & \tikzmark{l2}0 & 0 & 0 & 0\\
    0 & 0 & 0 & 1 & 0 & 0 & 0\\
    0 & 0 & 0\tikzmark{r} & 0 & 1 & 0 & 0\tikzmark{r2}\\
    0 & 0 & 0 & \tikzmark{left}2 & 0 & 0 & 0\\
    0 & 0 & 0 & 0 & 1 & 0 & 0\\
    0 & 0 & 0 & 0 & 1 & 0 & 1\\
    0 & 0 & 0 & 0 & -1 & 1 & 0\tikzmark{right}\\
  \end{pmatrix}
  \DrawBox[ForestGreen, thick]{left}{right}
  \DrawBox[gray, thick]{l}{r}
  \DrawBox[gray, thick]{l2}{r2}
  \text{\ \ and\ \ }
   A_{1} =
  \begin{pmatrix}
    \tikzmark{l}0 & 0 & 1 & \tikzmark{l2}0 & 0 & 0 & 0 \\
    0 & 0 & 0 & 0 & 0 & 1 & 0\\
    0 & 0 & 0\tikzmark{r} & 0 & 0 & 0 & 1\tikzmark{r2} \\
    0 & 0 & 0 & \tikzmark{left}0 & 0 & 2 & 0 \\
    0 & 0 & 0 & 0 & 0 & 0 & 1\\
    0 & 0 & 0 & 0 & -1 & 1 & 1\\
    0 & 0 & 0 & 0 & 2 & 0 & 1\tikzmark{right}
  \end{pmatrix}.
  \DrawBox[Maroon, thick]{left}{right}
  \DrawBox[gray, thick]{l}{r}
  \DrawBox[gray, thick]{l2}{r2}
\end{equation}

Next, we use Theorem~\ref{thm:offset-correction} in order to ``adjust'' the initial values
of the recurrence relations of~$f$ given by
Corollary~\ref{cor:ub-q-recursive} for $0\leq n \leq 2$.

For this purpose we recall the notation $\delta_{k}\colon\N_{0}\to\C$ with
$\delta_{k}(n)\coloneqq \iverson{n=k}$ for $0\leq k \leq 2$ as introduced in Theorem~\ref{thm:offset-correction} and set
\begin{equation*}
  w_{r,k} \coloneqq v(2k+r) - A_{r}v(k)
\end{equation*}
for $0\leq k \leq 2$ and $0 \le r < 2$. Then the matrices
$W_{0} = (w_{0,0},w_{0,1},w_{0,2})$ and
$W_{1} = (w_{1,0},w_{1,1},w_{1,2})$ are given by
\begin{equation*}
  W_{0} =
  \begin{pmatrix}
    0 & 0 & 0\\
    0 & 0 & 0\\
    0 & 0 & 0\\
    -1 & 0 & 0\\
    0 & 0 & 0\\
    -4 & 0 & 0\\
    4 & 2 & 0
  \end{pmatrix}
  \qq{and}
  W_{1} =
  \begin{pmatrix}
    0 & 0 & 0\\
    0 & 0 & 0\\
    0 & 0 & 0\\
    -2 & 0 & 0\\
    0 & 0 & 0\\
    2 & 2 & 0\\
    -8 & -4 & -4
  \end{pmatrix}.
\end{equation*}
Moreover, recall that we have introduced matrices $J_{0}$ and $J_{1}$
in~\eqref{eq:matrix-J-corr-n0} by
\begin{equation*}
  J_{r} \coloneqq \bigl(\iverson{2j = k - r}\bigr)_{\substack{0\leq k < n_{0}\\ 0\leq j < n_{0}}},
\end{equation*}
which in this example gives
\begin{equation*}
  J_{0} =
  \begin{pmatrix}
    1 & 0 & 0\\
    0 & 0 & 0\\
    0 & 1 & 0\\
  \end{pmatrix}
  \qq{and}
  J_{1} =
  \begin{pmatrix}
    0 & 0 & 0\\
    1 & 0 & 0\\
    0 & 0 & 0\\
  \end{pmatrix}.
\end{equation*}

Then by Theorem~\ref{thm:offset-correction} we obtain a $2$-linear
representation $(\widetilde{A}_{0}, \widetilde{A}_{1}, \widetilde{v})$
of~$f$ by the vector~$\widetilde{v}$ given in block form as
\begin{equation*}
  \widetilde{v} =
  \begin{pmatrix}
    v\\
    \delta_{0}\\
    \delta_{1}\\
    \delta_{2}
  \end{pmatrix}
\end{equation*}
as well as the block matrices
\begin{equation}
  \label{eq:ub-Atilde}
  \widetilde{A}_{0} =
  \begin{pmatrix}
    A_{0} & W_{0}\\
    0 & J_{0}
  \end{pmatrix}
  \qq{and}
  \widetilde{A}_{1} =
  \begin{pmatrix}
    A_{1} & W_{1}\\
    0 & J_{1}
  \end{pmatrix}.
\end{equation}

This final $2$-linear representation has dimension~$10$. By applying the
minimization algorithm mentioned in
Remark~\ref{rem:main-thm}~(\ref{item:minimization}), we see that the
smallest possible $2$-linear representation of~$f$ has dimension~$8$.




\subsection{Joint Spectral Radius}

Next, we determine the joint spectral radius of~$B_{0}$
and $B_{1}$, which we need to determine the asymptotics of
the summatory function of~$f$ in the next section.

\begin{lemma}
  \label{lem:jsr-ub}
  The joint spectral radius of $\mathcal{B}=\{B_{0}, B_{1}\}$ is~$2$ and
  $\mathcal{B}$ has the simple growth property.
\end{lemma}

\begin{proof}
  As usual, $\norm{\,\cdot\,}_\infty$ denotes the maximum norm of a vector and the
  row sum norm of a matrix.
  Let
  \begin{equation*}
    T\coloneqq\diag(2, 1/2, 1, 1)
  \end{equation*}
  and consider the vector norm $\norm{\,\cdot\,}_T$ defined by
  $\norm{v}_T\coloneqq \norm{T^{-1}v}_\infty$.
  This vector norm
  induces the matrix norm $\norm{\,\cdot\,}_T$ defined by $\norm{G}_T\coloneqq
  \norm{T^{-1}GT}_\infty$.

  We consider all products of two matrices in $\mathcal{B}$ and get
  \begin{align*}
    T^{-1}B_{0}^{2}T &=
    \begin{pmatrix}
      4 & 0 & 0 & 0 \\
      0 & 1 & 0 & 0 \\
      0 & 0 & 1 & 0 \\
      0 & 0 & 0 & 1
    \end{pmatrix},&
    T^{-1}B_{0}B_{1}T &=
    \begin{pmatrix}
      0 & 0 & 2 & 0 \\
      0 & 0 & 0 & 2 \\
      0 & 1 & 0 & 2 \\
      0 & -\frac{1}{2} & 1 & 0
    \end{pmatrix},\\
    T^{-1}B_{1}B_{0}T &=
    \begin{pmatrix}
      0 & \frac{1}{2} & 0 & 1 \\
      0 & -1 & 2 & 0 \\
      0 & -\frac{1}{2} & 1 & 1 \\
      0 & \frac{1}{2} & 1 & 0
    \end{pmatrix},&
    T^{-1}B_{1}^{2}T &=
    \begin{pmatrix}
      0 & -\frac{1}{2} & 1 & 1 \\
      0 & 2 & 0 & 2 \\
      0 & \frac{1}{2} & 1 & 1 \\
      0 & 1 & 0 & 3
    \end{pmatrix}.
  \end{align*}
  We observe that all these matrices have row sum norm at most $4$,
  which implies that
  \[\max\{\norm{G_1G_2}_T\mid G_1, G_2\in \mathcal{B}\}=4\]
  and therefore $\rho(\mathcal{B})\leq2$.
  Furthermore, we observe that $\norm{B_0^k}_T=2^k$ holds for all even
  positive integers~$k$. We conclude that $\rho(\mathcal{B})=2$ and that
  $\mathcal{B}$ has the finiteness property and thus the simple growth property.
\end{proof}

\begin{remark}\label{remark:unbordered-factors-finiteness-property}
  As announced in Remark~\ref{remark:joint-spectral-radius-etc}, the finiteness
  property may depend on the chosen norm. Indeed,
  the set
  $\mathcal{B}$ in Lemma~\ref{lem:jsr-ub} has the finiteness property with
  respect to the norm considered in the proof of Lemma~\ref{lem:jsr-ub}, but
  does not have the finiteness property
  with respect to the row sum norm: By computing the eigendecomposition of $B_1$
  or by induction on $k$, we obtain that
  the last row of $B_1^{k}$ is
  \[
    \Bigl(
      0,\ \frac{2}{3}(2^k-(-1)^k),\ 0,\ \frac13(2\cdot 2^k+(-1)^k)
    \Bigr)
  \]
  for $k\ge 0$. In particular,
  \[
    \norm{B_1^k}_\infty\ge \frac{4}{3}2^k - \frac{1}{3}(-1)^k > 2^k
  \]
  holds for all $k\ge 1$.
\end{remark}

\subsection{Asymptotics}

Let $\eta \geq 1$ be an integer. We define the Dirichlet series
\begin{equation*}
  \mathcal{F}_{j}(s) \coloneqq \sum_{n\geq \eta}\frac{f(4n + j)}{(4n + j)^{s}}
\end{equation*}
for $0\leq j\leq 3$ and set
\begin{equation*}
  \mathcal{F}(s) \coloneqq
  \begin{pmatrix}
    \mathcal{F}_{0}(s)\\
    \mathcal{F}_{1}(s)\\
    \mathcal{F}_{2}(s)\\
    \mathcal{F}_{3}(s)
  \end{pmatrix}.
\end{equation*}

In the following theorem, we state the main result of this section: We
give an asymptotic formula for $F(N) \coloneqq \sum_{0\leq n < N}f(n)$.

\begin{theorem}[Asymptotics for the Number of Unbordered Factors]
  \label{thm:ub-asymp}
  The summatory function~$F$ of the number~$f(n)$ of unbordered
  factors of length~$n$ in the Thue--Morse sequence satisfies
  \begin{equation*}
    F(N) = \smashoperator{\sum_{0\leq n < N}} f(n) = N^{\kappa}\cdot\Phi_{F}(\fractional{\log_{2}N}) + O(N\log N)
  \end{equation*}
  as $N \to \infty$, where $\kappa = \log_{2}(1+\sqrt{3}) = 1.44998431347650\ldots$ and~$\Phi_{F}$ is a
  $1$-periodic continuous function which is Hölder continuous
  with any exponent smaller than $\kappa - 1$. The Fourier
  coefficients of~$\Phi_{F}$ can be computed efficiently.

  Furthermore, the Dirichlet series $\mathcal{F}$ satisfies the
  functional equation
  \begin{equation}
    \label{ub-functional-equation}
    \bigl(I - 2^{-s}(B_{0} + B_{1})\bigr)\mathcal{F}(s) =
    \begin{pmatrix}
      \mathcal{G}_{0}(s)\\
      \mathcal{G}_{1}(s)\\
      \mathcal{G}_{2}(s)\\
      \mathcal{G}_{3}(s)
    \end{pmatrix}
  \end{equation}
  for $\Re s > 1$, where $B_{0}$ and $B_{1}$ are the matrices
  given in~\eqref{eq:ub-system-matrices}, and
  \begin{equation*}
    \mathcal{G}_{j}(s) = 2^{-s}\sum_{k=0}^{3}\Biggl(\sum_{n\geq 1}
    \binom{-s}{n}\biggl(c_{j,k}\Bigl(\frac{j}{2} - k\Bigr)^{n}
    + c_{j+4,k}\Bigl(\frac{j+4}{2} - k\Bigr)^{n}\biggr)\mathcal{F}_{k}(s+n) \Biggr)
    + \ \smashoperator[l]{\sum_{\eta\leq n < 2\eta}}\frac{f(4n + j)}{(4n + j)^{s}},
  \end{equation*}
  where $(c_{j,k})_{0\le j<8, 0\le k\le 3}$ are the coefficients of the
  $2$-recursive sequence in Corollary~\ref{cor:ub-q-recursive}.
  Moreover, $\mathcal{G}_{j}(s)$ is analytic for $\Re s > 1$ and $0\leq j \leq 3$, and, in particular, $\mathcal{F}(s)$ is meromorphic for $\Re s > 1$ and
  can only have poles $s = \log_{2}(\sqrt{3}+1) + \frac{2\pi i\mu}{\log 2}$
  with $\mu\in\Z$.
\end{theorem}
Table~\ref{tab:ub-fourier} and Figure~\ref{fig:ub-fluc} show the first few Fourier
coefficients and a plot of the periodic
fluctuation of Theorem~\ref{thm:ub-asymp}, respectively.

\begin{table}[htbp]
  \centering
  \begin{tabular}{r|l}
    \multicolumn{1}{c|}{$\mu$} & \multicolumn{1}{c}{$\varphi_{\mu}$}\\\hline
    $0$ & $\phantom{-}1.081200224751780$\\
    $1$ & $-0.0012296808157996 + 0.0157152473714320i$\\
    $2$ & $-0.0013742386970566 - 0.0110033266904103i$\\
    $3$ & $\phantom{-}0.0083338522036749 + 0.0034850861320328i$\\
    $4$ & $-0.0042230458157050 - 0.0017461310727764i$\\
    $5$ & $\phantom{-}0.0064055951023042 - 0.0018152583716649i$\\
    $6$ & $\phantom{-}0.0004060978191922 + 0.0003312870598610i$\\
    $7$ & $\phantom{-}0.0000421719576760 - 0.0039180258068148i$\\
    $8$ & $-0.0001216226961034 + 0.0017930566948364i$\\
    $9$ & $-0.0010960361908015 - 0.0012530549651823i$\\
    $10$ & $\phantom{-}0.0024882524350784 - 0.0009385940552233i$\\
  \end{tabular}
  \caption{First few Fourier coefficients of~$\Phi_{F}$}
  \label{tab:ub-fourier}
\end{table}

\begin{figure}[htbp]
  \centering
  \includegraphics{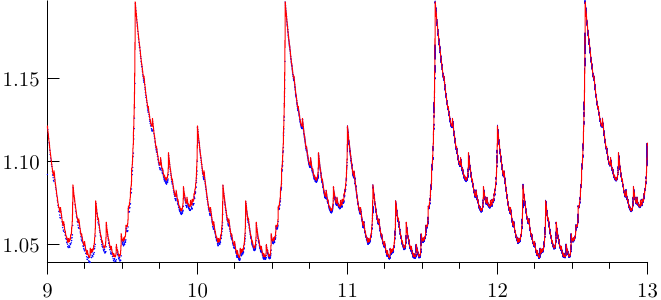}
  \caption{Fluctuation in the main term of the asymptotic expansion of
    the summatory function~$F$. The plot shows the periodic
    fluctuation~$\Phi_{F}(u)$ approximated by its Fourier series of
    degree~$2000$ (red) as well as the
    function~$F(2^{u})/2^{\kappa u}$ (blue).}
  \label{fig:ub-fluc}
\end{figure}

\begin{proof}[Proof of Theorem~\ref{thm:ub-asymp}]
  We use Theorem~\ref{thm:asymp} with the linear
  representation~$(\widetilde{A}_{0},\widetilde{A}_{1},\widetilde{v})$
  of~$f$ as obtained in Section~\ref{sec:ub-lin-rep}. For this
  purpose, we work out the parameters needed in the theorem.

  \emph{Joint Spectral Radius.} Due to
  Lemma~\ref{lem:jsr-ub},
  Proposition~\ref{cor:jsr-special-case} and
  Proposition~\ref{prop:jsr-A}, the
  joint spectral radius of~$\widetilde{A}_{0}$ and~$\widetilde{A}_{1}$
  is~$2$. Moreover, the simple growth property holds.
  So we set $R = 2$.

  \emph{Eigenvalues.} Let
  $\widetilde{C} = \widetilde{A}_{0} + \widetilde{A}_{1}$. We could of
  course directly calculate the eigenvalues
  of~$\widetilde{C}$. However, we want to emphasize the relation to
  the matrices~$B_{0}$ and~$B_{1}$ as given in~\eqref{eq:ub-system-matrices} and thus, to the recurrence
  relations given by the property that the sequence~$f$ is
  $2$-recursive.  Let~$A_{0}$ and~$A_{1}$ be the matrices in~\eqref{eq:ub-As}
  and set $C \coloneqq A_{0} + A_{1}$. Then applying
  Proposition~\ref{cor:sprectrum-C-Bs} yields
  \begin{equation*}
    \sigma(C) = \sigma(B_{0} + B_{1})\cup\set{0}
  \end{equation*}
  as well as $m_{C}(\lambda) = m_{B_{0} + B_{1}}(\lambda)$ for
  $\lambda\in\sigma(B_{0}+B_{1})$, and we recall that
  \begin{equation*}
    \sigma(B_{0}+B_{1}) = \set[\big]{-\sqrt{3} + 1, 1,2,\sqrt{3}+1};
  \end{equation*}
  see Section~\ref{sec:ub-lin-rep}. Moreover, all eigenvalues of
  $B_{0} + B_{1}$ are simple and thus, we have $m_{C}(\lambda) = 1$
  for all $\lambda\in\sigma(B_{0}+B_{1})$.

  Due to Proposition~\ref{cor:n0-same-eigenvalues}, we have
  \begin{equation*}
    \sigma(\widetilde{C}) = \sigma(C) \cup\set{0,1}
  \end{equation*}
  and $m_{\widetilde{C}}(\lambda) = m_{C}(\lambda)$ for $\lambda\in\sigma(B_{0}+B_{1})\setminus\set{1}$.
  All in all, this means for the spectrum of $\widetilde{C}$ that
  \begin{equation*}
    \sigma(\widetilde{C}) = \sigma(C) \cup\set{0,1} = \sigma(B_{0} + B_{1})\cup\set{0,1} = \set{-\sqrt{3} +1, 0, 1, 2, \sqrt{3}+1}.
  \end{equation*}
  So the only relevant eigenvalues of $\widetilde{C}$, i.e., the only eigenvalues~$\lambda$ with
  $\abs{\lambda} \geq R$, are~$\sqrt{3}+1$ with
  $m_{\widetilde{C}}(\sqrt{3}+1) = 1$ and $2$ with $m_{\widetilde{C}}(2) = 1$.

  Now applying Theorem~\ref{thm:asymp} yields the result for the
  asymptotic expansion. Finally, Proposition~\ref{prop:dirichlet}
  implies the correctness of the functional equation as stated
  in~\eqref{ub-functional-equation}.
\end{proof}




\section{Proofs}
\label{chap:proofs}

\subsection{Proofs of the Reductions to $q$-Regular Sequences in the General Case}
\label{sec:proof-main-result}

For the proof of Theorem~\ref{thm:remark-2.1-general}, we need the following lemma.

\begin{lemma}
  \label{lem:bounds}
  Let $q\geq 2$, $M > m \geq 0$ and $\ell\leq u$ be integers,
  and let $\ell'$ and $u'$ be as defined in~\eqref{eq:ell-and-u}.
  Furthermore, let~$d$ be an integer with
  $\ell'\leq d\leq q^{M-1}-q^{m}+u'$ and write $d = d'q^{M} + r'$ with
  $0\leq r' < q^{M}$ and $d'\in\Z$. Then we have
  \begin{equation}
    \label{eq:lemma-bounds-case-1}
    \ell' \leq q^{m}d' + \ell \qq{and} q^{m}d' + u \leq u'.
  \end{equation}
  If additionally $r' \geq q^{M-1}$ holds, then
  \begin{equation}
    \label{eq:lemma-bounds-case-2}
    q^{m}d' + q^{m} + u \leq u'.
  \end{equation}
\end{lemma}

\begin{proof}
  First of all, we note that the equalities
  \begin{equation}
    \label{eq:graham-floor}
    \floor[\Big]{\frac{\floor{x} + a}{b}} = \floor[\Big]{\frac{x + a}{b}}
    \qq{and}
    \ceil[\Big]{\frac{\ceil{x} + a}{b}} = \ceil[\Big]{\frac{x + a}{b}}
  \end{equation}
  hold for all $x\in\R$ and $a$, $b\in\Z$ with $b > 0$; see Graham, Knuth and Patashnik~\cite[p.~72]{Graham-Knuth-Patashnik:1994}.

  We now show that the left inequality of~\eqref{eq:lemma-bounds-case-1} is true. If
  $\ell \ge 0$, then $\ell' = 0$ and the inequality follows. If $\ell < 0$,
  then we have
  \begin{align*}
    q^{m}d' + \ell
    &= \frac{d - r'}{q^{M-m}} + \ell\\
    &> \frac{\ell' - q^{M} + \ell q^{M-m}}{q^{M-m}}\\
    &= \frac{ \floor[\big]{\frac{(\ell+1) q^{M-m} - q^{M}}{q^{M-m}-1}} - q^{M} + \ell q^{M-m}}{q^{M-m}}\\
    &\geq \floor[\Bigg]{ \frac{\floor[\big]{\frac{(\ell+1) q^{M-m} - q^{M}}{q^{M-m}-1}} - q^{M} + \ell q^{M-m}}{q^{M-m}}}\\
    &\downtoeq{\eqref{eq:graham-floor}} \floor[\Bigg]{\frac{\frac{(\ell+1) q^{M-m} - q^{M}}{q^{M-m}-1} - q^{M} + (\ell+1) q^{M-m}}{q^{M-m}}-1}\\
    &= \floor[\bigg]{\frac{(\ell + 1)q^{M-m} - q^{M}}{q^{M-m} - 1}} - 1 = \ell' - 1,
  \end{align*}
  which is equivalent to $\ell' \leq q^{m}d' + \ell$.

  For the remaining two inequalities, we distinguish between negative and
  non-negative upper bounds~$u$.

  \begin{description}
  \item[Case 1:] $u \leq -1$. This implies $u' = q^{m} - 1$ and as a
    consequence, we have $d \leq q^{M-1} - 1$ as well as $d' \leq 0$. The right
    inequality of~\eqref{eq:lemma-bounds-case-1} as well as~\eqref{eq:lemma-bounds-case-2} follow directly.

  \item[Case 2:] $u \geq 0$. This implies $u' = q^{m} - 1 + \ceil[\big]{\frac{uq^{M-m}}{q^{M-m}-1}}$, and we have
  \begin{align*}
    q^{m}d' + u
    &= \frac{d - r'}{q^{M-m}} + u \leq \frac{u' + q^{M-1} - q^{m} + uq^{M-m}}{q^{M-m}}\\
    &< \frac{u' + q^{M} - q^{m} + uq^{M-m}}{q^{M-m}} & \refstepcounter{equation}\tag{\theequation}\label{eq:stern}\\
    &= \frac{q^{m} - 1 + \ceil[\big]{\frac{uq^{M-m}}{q^{M-m}-1}} + q^{M} - q^{m} + uq^{M-m}}{q^{M-m}}\\
    &< q^{m} + \frac{\ceil[\big]{\frac{uq^{M-m}}{q^{M-m}-1}} + uq^{M-m}}{q^{M-m}}\\
    &\leq q^{m} + \ceil[\Bigg]{\frac{\ceil[\big]{\frac{uq^{M-m}}{q^{M-m}-1}} + uq^{M-m}}{q^{M-m}}}\\
    &\downtoeq{\eqref{eq:graham-floor}} q^{m} + \ceil[\Bigg]{\frac{\frac{uq^{M-m}}{q^{M-m}-1} + uq^{M-m}}{q^{M-m}}}\\
    &= q^{m} + \ceil[\bigg]{\frac{uq^{M-m}}{q^{M-m} - 1}} = u' + 1.
  \end{align*}
  This is equivalent to $q^{m}d' + u \leq u'$,
  and~\eqref{eq:lemma-bounds-case-1} is shown. Now let $r'\geq q^{M-1}$,
  then it follows that
  \begin{align*}
    q^{m}d' + q^{m} + u &= \frac{d - r'}{q^{M-m}} + q^{m} + u\\
    &\leq \frac{u' + q^{M - 1} - q^{m} - q^{M-1} + q^{M} + uq^{M-m}}{q^{M-m}}\\
    &= \frac{u' + q^{M} - q^{m} + uq^{M-m}}{q^{M-m}}.
  \end{align*}
  We can now use the same steps as above (from \eqref{eq:stern} on)
  and obtain $q^{m}d' + q^{m} + u \leq u'$, which
  shows that~\eqref{eq:lemma-bounds-case-2} holds. This completes the proof.\qedhere
\end{description}
\end{proof}

\begin{proof}[Proof of Theorem~\ref{thm:remark-2.1-general}]
  Let~$v$ be
  as stated in Theorem~\ref{thm:remark-2.1-general} and let $(c_{r,k})_{0\le
    r<q^M, \ell\le k\le u}$ be the coefficients of the $q$-recursive sequence~$x$. We have to show that there are
  matrices $A_{0}$, $\ldots$, $A_{q-1}$ such that
  $v(qn + r) = A_{r}v(n)$ holds for all $n\geq n_{1}$ and
  $0\leq r < q$, as required in the definition of regular sequences.
  This is equivalent to the property that
  each component of $v(qn + r)$ is a linear combination of the entries
  in~$v(n)$, where all sequences are restricted to~$n\geq n_{1}$.

  We fix $0\leq r < q$ and split the proof into three parts depending
  on the indices of the blocks~$v_{j}$ of~$v$ as introduced
  in~\eqref{eq:recursive-blocks-of-v}. In each part,
  the choice of the components of~$v$ as defined
  in~\eqref{eq:recursive-components-of-v-1} and~\eqref{eq:recursive-components-of-v-2}
  allows to represent $v_{j}(qn+r)$ as a linear combination of the entries
  of~$v$.

  Furthermore, we let $\ind_{v}(x\circ(n\mapsto q^{j}n+d))$ denote the
  absolute position of some sequence~$x\circ(n\mapsto q^{j}n+d)$
  in the vector~$v$, absolute in the sense that we disregard the block
  form of~$v$.\footnote{We assume that the operator $\ind_{v}$ is
    ``intelligent'' enough to recognize the
    sequence~$x\circ(n\mapsto q^{j}n+d)$ in~$v$ by the parameters~$j$ and~$d$, even if there are
    other sequences in~$v$ that are equal
    to~$x\circ(n\mapsto q^{j}n+d)$ (in the sense that all elements of the sequences are equal). For reasons of simplicity and
    readability, we refrain from develop this more explicitly at this
    point.} Moreover, we write~$v_{j,d}$ for the sequence
  $x\circ(n\mapsto q^{j}n+d)$ where this notation is convenient.

  \paragraph{Part 1.}

    At first, we consider blocks~$v_{j}$ of~$v$ with
    $0\leq j < m$.

    \emph{Components of~$v_{j}$:} Let $0\leq d < q^{j}$ and consider the component
    $v_{j,d} = x\circ(n\mapsto q^{j}n + d)$ of~$v$. Then we get
    \begin{align*}
      v_{j,d}(qn+r) &= x\bigl(q^{j}(qn + r) + d\bigr)\\
      &= x(q^{j+1}n + q^{j}r + d).
    \end{align*}

    We claim that $x\circ(n\mapsto q^{j+1}n+q^{j}r + d)$ is a component
    of $v_{j+1}$. To show this, we first observe that
    \begin{equation}\label{eq:proof-main-theorem-part-1}
      0\leq q^{j}r + d \leq q^{j}(q-1) + q^{j} - 1 = q^{j + 1} - 1.
    \end{equation}
    For $j \leq m-2$, this immediately proves the claim.
    If $j=m-1$, then the estimates $\ell'\le 0$ (see Remark~\ref{rem:main-thm}) and $q^{j+1}-1 = q^{m}-1 \leq u'$ together
    with~\eqref{eq:proof-main-theorem-part-1} ensure the validity of the claim.

    \emph{Corresponding Rows of~$A_{r}$:} The previous
    considerations imply that the row~$\ind_{v}(v_{j,d})$ of~$A_{r}$
    has zeros in every entry except one. Specifically, the entry in
    column~$\ind_{v}\bigl(v_{j+1, q^{j}r + d}\bigr)$ is~$1$.

  \paragraph{Part 2.}

    Secondly, we consider blocks~$v_{j}$ of~$v$ with
    $m\leq j \leq M-2$.

    \emph{Components of~$v_{j}$:} Let
    $\ell'\leq d \leq q^{j} - q^{m} + u'$. Then the sequence
    $x\circ(n\mapsto q^{j}n + d)$ is a component of~$v_{j}$ and it
    holds that
    \begin{equation*}
      x\bigl(q^{j}(qn+r) + d\bigr) = x(q^{j+1}n + q^{j}r + d)
    \end{equation*}
    with
    \begin{equation*}
      \ell'\leq d \leq q^{j}r + d \leq q^{j}(q - 1) + q^{j} - q^{m} + u'
      = q^{j+1} - q^{m} + u'.
    \end{equation*}
    This implies that the sequence $x\circ(n\mapsto q^{j}(qn + r) + d)$ is
    a component of~$v$, namely in the block~$v_{j+1}$.

    \emph{Corresponding Rows of~$A_{r}$:} Also in this part,
    the row~$\ind_{v}(v_{j,d})$ of~$A_{r}$ consists of zeros except for
    position~$\ind_{v}(v_{j+1,q^{j}r + d})$ where we have a~$1$.

  \paragraph{Part 3.}

    Finally, we have a look at the last block~$v_{M-1}$
    of~$v$.

    \emph{Components of~$v_{M-1}$:} Let
    $\ell' \leq d \leq q^{M-1} - q^{m} + u'$ and consider the
    component $v_{M-1,d} = x\circ(n\mapsto q^{M-1}n + d)$
    of~$v_{M-1}$. Write $d=d'q^{M} + r'$ with $0\leq r' \leq
    q^{M}-1$ and $d'\in\Z$.
    The component of $v(qn+r)$ which corresponds to~$v_{M-1,d}$
    is given by
    \begin{equation*}
      x\bigl(q^{M-1}(qn+r) + d\bigr) = x\bigl(q^{M}(n+d') + q^{M-1}r + r'\bigr) = x\bigl(q^{M}(n + d') + \tilde{r}\bigr)
    \end{equation*}
    with $\tilde{r} \coloneqq q^{M-1}r + r'$. Note that
    $\tilde{r}\le q^{M-1}(q-1)+q^{M}-1=2q^M-q^{M-1}-1$.
    We distinguish the
    following two cases with respect to the parameter~$r$ for determining
    the rows of~$A_r$ corresponding to the block~$v_{M-1}$.

    \begin{description}
    \item[Case 1:] $0\leq \tilde{r} < q^{M}$. We have
      \begin{equation}\label{eq:proof-of-thm-a-part-3-1}
        \begin{split}
        x\bigl(q^{M}(n + d') + \tilde{r}\bigr)
        &= \smashoperator{\sum_{\ell\leq k\leq u}}c_{\tilde{r}, k}\,x\bigl(q^{m}(n + d') + k\bigr)\\
        &= \smashoperator{\sum_{\ell\leq k\leq u}}c_{\tilde{r}, k}\,x\bigl(q^{m}n + q^{m}d' + k)
        \end{split}
      \end{equation}
      for $n + d'\geq n_{0}$ due to~\eqref{eq:def-q-recursive}.
      As $d'=\floor{d/q^{M}}\ge \floor{\ell'/q^{M}}$ holds, the condition $n + d'\geq n_{0}$ is fulfilled
      for $n\ge n_{1}$ by definition of $n_{1}$.
      We need to show that
      $\ell' \leq q^{m}d' + k \leq u'$ holds for all
      $\ell\leq k\leq u$. We apply Lemma~\ref{lem:bounds} and directly
      obtain both inequalities, i.e., $\ell' \leq q^{m}d' + k$ and $q^{m}d' + k \leq u'$,
      by~\eqref{eq:lemma-bounds-case-1}. This implies that the
      sequence $x\circ (n\mapsto q^{M-1}(qn + r) + d)$ is a linear
      combination of sequences in~$v$, where all sequences are restricted to $n\geq n_{1}$.

      \emph{Corresponding rows of~$A_{r}$:} We have a look at
      row~$\ind_{v}(v_{M-1,d})$ of~$A_{r}$. The first couple of entries are
      zero, up to position~$\ind_{v}(v_{m,q^{m}d'+\ell}) - 1$. On
      position~$\ind_{v}(v_{m,q^{m}d'+\ell})$ up to
      position~$\ind_{v}(v_{m,q^{m}d'+u})$ we have the entries
      $c_{\tilde{r},\ell}$, \dots, $c_{\tilde{r},u}$, followed
      by zeros up to the end of the row.

    \item[Case 2:] $q^{M} \leq \tilde{r} < 2q^{M} - q^{M-1}
      -1$. This implies $r' \geq q^{M-1}$ and $0\leq \tilde{r} - q^M < q^M$.
      We obtain
      \begin{equation}\label{eq:proof-of-thm-a-part-3-2}
        \begin{split}
        x\bigl(q^{M}(n + d') + \tilde{r}\bigr) &=
        x\bigl(q^{M}(n + d' + 1) + \tilde{r} - q^{M}\bigr)\\
        &= \smashoperator{\sum_{\ell\leq k\leq u}}c_{\tilde{r}-q^{M}, k}\,x\bigl(q^{m}(n + d' + 1) + k\bigr)\\
        &= \smashoperator{\sum_{\ell\leq k\leq u}}c_{\tilde{r}-q^{M}, k}\,x\bigl(q^{m}n + q^{m}d' + q^{m} + k)
        \end{split}
      \end{equation}
      for $n \geq n_{1} - 1$, again due to~\eqref{eq:def-q-recursive}. Thus, we need to argue
      that $\ell'\leq q^{m}d' + q^{m} + k \leq u'$ holds for all
      $\ell\leq k\leq u$. For this purpose, we again use
      Lemma~\ref{lem:bounds}. The inequality
      $q^{m}d' + q^{m} + k \leq u'$ directly follows
      from~\eqref{eq:lemma-bounds-case-2}. For the lower bound we
      obtain
      \begin{equation*}
        q^{m}d' + q^{m} + k \geq q^{m}d' + q^{m} + \ell > q^{m}d' + \ell
        \downtosymb{\eqref{eq:lemma-bounds-case-1}}{\geq} \ell'.
      \end{equation*}
      This shows that also in this case, the sequence
      $x\circ (n\mapsto q^{M-1}(qn + r) + d)$ is a linear combination
      of sequences in~$v$, where again all sequences are restricted to~$n\geq n_{1}$.

      \emph{Corresponding rows of~$A_{r}$:} We again have a
      look at row~$\ind_{v}(v_{M-1,d})$ of~$A_{r}$. The first couple of entries
      are zero, up to
      position~$\ind_{v}(v_{m,q^{m}d'+q^{m}+\ell}) - 1$. On
      position~$\ind_{v}(v_{m,q^{m}d'+q^{m}+\ell})$ up to
      position~$\ind_{v}(v_{m,q^{m}d'+q^{m}+u})$ we have the
      entries~$c_{\tilde{r}-q^M,\ell}$, \dots,
      $c_{\tilde{r}-q^M,u}$, followed by zeros up to the end of the
      row.
    \end{description}

  \paragraph{Finalizing the Proof.}

  At this point, we have completed constructing matrices $A_r$ for $0\leq r < q$ row by row
  and have shown that $(A_0, \ldots, A_{r-1}, v)$ is a linear representation of $x$.
  We now gather the information on $A_r$ which was discussed in the proof. We let $D$ denote the dimension of $A_r$. If $a_{i}$ is the $i$th row of~$A_{r}$,
  define $j\in\N_0$ and $d\in\Z$ by $i=\ind_{v}(v_{j, d})$. Then we have
  \begin{equation}
    \label{eq:Ar-explicitly}
    a_{i}^{\top} =
    \begin{cases}
      \bigl(\iverson[\big]{k = \ind_{v}(v_{j+1,q^{j}r+d})}\bigr)_{1\leq k \leq
          D} & \text{if } 0\leq j \leq M-2,\\
      \bigl(c_{\tilde{r},k + k_{1}}\cdot\iverson{\ell \leq k + k_{1} \leq u}\bigr)_{1\leq k \leq D} & \text{if } j = M-1 \text{ and } \tilde{r} < q^{M},\\
            \bigl(c_{\tilde{r}-q^M,k + k_{2}}\cdot\iverson{\ell \leq k + k_{2} \leq u}\bigr)_{1\leq k \leq D} & \text{if } j = M-1 \text{ and } \tilde{r} \geq q^{M},
    \end{cases}
  \end{equation}
  where
  \begin{itemize}
  \item $\tilde{r} \coloneqq q^{M-1}r + r'$
    with $d = d'q^{M} + r'$, $d'\in\Z$ and $0\leq r' < q^{M}$,
  \item $k_{1} \coloneqq \ell -\ind_{v}(v_{m,q^{m}d'+\ell})$ and
  \item $k_{2} \coloneqq \ell -\ind_{v}(v_{m,q^{m}d'+q^{m}+\ell})$.
  \end{itemize}
  Note that all these parameters depend on~$i$.
\end{proof}

\begin{proof}[Proof of Theorem~\ref{thm:offset-correction}]
  The definitions of $\delta_{k}$ and $w_{r,k}$ in Theorem~\ref{thm:offset-correction} imply that
  \begin{equation*}
    v(qn + r) = A_{r}v(n) + \smashoperator{\sum_{0\leq k < n_{0}}}w_{r,k}\,\delta_{k}(n)
  \end{equation*}
  holds for all $0\leq r < q$ and $n \geq 0$. We further have
  \begin{equation*}
    \delta_{k}(qn + r) = \iverson{k \geq r \text{ and } k\equiv r \tpmod{q}}
    \cdot \delta_{(k-r)/q}(n)
  \end{equation*}
  for all $0\leq k < n_{0}$, $0\leq r < q$ and $n\geq 0$. With $\widetilde{v}$
  and $\widetilde{A}_{r}$ as defined in~\eqref{eq:offset-v}
  and~\eqref{eq:matrices-A-tilde}, it follows that
  $\widetilde{v}(qn + r) = \widetilde{A}_{r}\widetilde{v}(n)$ holds
  for all $0\leq r < q$ and $n\ge 0$. Consequently,
  $(\widetilde{A}_{0}, \ldots, \widetilde{A}_{q-1},\widetilde{v})$ is a
  $q$-linear representation of~$x$ and thus,~$x$ is $q$-regular. The statement
  about the shape of $J_r$ follows from the definition of $J_r$. This
  completes the proof of the theorem.
\end{proof}

Corollary~\ref{corollary:main:is-q-regular} directly follows by
combining Theorem~\ref{thm:remark-2.1-general} and
Theorem~\ref{thm:offset-correction}.

\begin{proof}[Proof of Corollary~\ref{cor:inhomogeneities}]
  First of all, note that for a $q$-regular sequence~$g$, its
  shifted version $g\circ(n\mapsto n + d)$ is $q$-regular
  for all integers $d\in\Z$; see Allouche and
  Shallit~\cite[Theorem~2.6 and the subsequent
  remark]{Allouche-Shallit:1992:regular-sequences}.

  We use the notation introduced in
  Theorem~\ref{thm:remark-2.1-general}. In order to prove the
  corollary, we construct a $q$-linear representation with
  offset~$n_{1}$ of~$x$ along the lines of the proof of
  Theorem~\ref{thm:remark-2.1-general}: Parts~1 and~2 of the proof can
  be adopted unchanged, but we have to carefully adapt Part~3:
  In~\eqref{eq:proof-of-thm-a-part-3-1}
  and~\eqref{eq:proof-of-thm-a-part-3-2}, we have to
  apply~\eqref{eq:recursions-with-inhomogeneities} instead
  of~\eqref{eq:def-q-recursive} and obtain
  \begin{equation}\label{eq:proof-corr-d-1}
    x\bigl(q^{M}(n + d') + \tilde{r}\bigr)
    = \smashoperator{\sum_{\ell\leq k\leq u}}c_{\tilde{r}, k}\,x\bigl(q^{m}n + q^{m}d' + k) + g_{\tilde{r}}(n+d')
  \end{equation}
  in the first case and
  \begin{equation}\label{eq:proof-corr-d-2}
    x\bigl(q^{M}(n + d' + 1) + \tilde{r} - q^{M}\bigr) =
    \smashoperator{\sum_{\ell\leq k\leq u}}c_{\tilde{r}-q^{M}, k}\,x\bigl(q^{m}n + q^{m}d' + q^{m} + k) +
    g_{\tilde{r}-q^{M}}(n+d'+1)
  \end{equation}
  in the second case, where $n\geq n_{1}$ and
  $\floor{\ell'/q^{M}}\leq d' \leq \floor{(q^{M-1} - q^{m} +
    u')/q^{M}}$. Thus, if we append the vectors of the $q$-linear
  representations of the $q$-regular sequences
  $g_{s}\circ(n\mapsto n + d'')$ for $0\leq s < q^{M}$ and
  $\floor{\ell'/q^{M}}\leq d'' \leq \floor{(q^{M-1} - q^{m} + u')/q^{M}}
  + 1$ to the vector~$v$ as given by~\eqref{eq:recursive-blocks-of-v}
  and analogously join the corresponding matrices together---with some
  additional entries~$1$ caused by $g_{\tilde{r}}(n+d')$ and
  $g_{\tilde{r}-q^{M}}(n+d'+1)$ in~\eqref{eq:proof-corr-d-1}
  and~\eqref{eq:proof-corr-d-2}, respectively---then we obtain a
  $q$-linear representation with offset~$n_{1}$ of~$x$.

  Finally, we correct the offset of the $q$-linear representation by
  applying Theorem~\ref{thm:offset-correction}. Consequently, $x$ is
  $q$-regular.
\end{proof}

\subsection{Proof of the Reduction to $q$-Regular Sequences in the Special Case}

\begin{proof}[Proof of Theorem~\ref{prop:recursive-special-case}]
  We split the proof into two parts depending on the indices of the
  blocks of~$v$.

  \paragraph{Part 1.}
    At first, we consider blocks $v_{j}$ with
    $0\leq j<m$. These blocks coincide with~\eqref{eq:recursive-components-of-v-1} in
    Theorem~\ref{thm:remark-2.1-general}.
    Moreover, these blocks correspond to the first
    $\sum_{0\leq j <m}q^{j} = (q^{m} - 1)/(q-1)$ rows of the matrices~$A_{0}$, \dots, $A_{q-1}$ and thus, also to the rows of~$J_{r0}$ and~$J_{r1}$ for $0\leq r < q$. By the
    proof of Theorem~\ref{thm:remark-2.1-general}, each of these rows
    consists of zeros and exactly one~$1$, which in particular means
    that the only entries of $J_{r0}$ and $J_{r1}$
    are zeros and ones as stated in the theorem. It remains to
    show that the matrices~$J_{r0}$ are upper triangular with zeros on
    the diagonal. This follows from the fact that
    \begin{equation*}
      \ind_{v}\bigl(x\circ(n\mapsto q^{j}(qn + r) + d)\bigr) > \ind_{v}\bigl(x\circ(n\mapsto q^{j}n + d)\bigr)
    \end{equation*}
    holds for all $0\leq j < m$, $0\leq r < q$ and $0\leq d < q^{j}$.
  \paragraph{Part 2.}
    We take a look at the block~$v_{m}$. The definition of~$B_{r}$ in~\eqref{eq:prop-recursive-Bs} implies that
    \begin{equation*}
      v_{m}(qn + r) = B_{r}v_{m}(n)
    \end{equation*}
    holds for all $0\leq r < q$, which, together with Part~1, directly
    implies that  $(A_{0}, \dots, A_{q-1}, v)$ as given in
    \eqref{eq:prop-recursive-v} and~\eqref{eq:prop-recursive-matrices} is indeed a linear representation of~$x$.
\end{proof}

\subsection{Proofs of the Spectral Results}

We start with the proof of Proposition~\ref{cor:n0-same-eigenvalues}.

\begin{proof}[Proof of Proposition~\ref{cor:n0-same-eigenvalues}]
  Due to~\eqref{eq:matrices-A-tilde}, we have
  \begin{equation*}
    \widetilde{C} =
    \begin{pmatrix}
      C & W\\
      0 & J
    \end{pmatrix}
  \end{equation*}
  with the $n_{0}\times n_{0}$ matrix $J = \sum_{0\leq r < q}J_{r}$ and
  $W = \sum_{0\leq r < q}W_{r}$, and thus,
  $\sigma(\widetilde{C}) = \sigma(C)\cup\sigma(J)$. By the properties of $J_r$
  noted in Theorem~\ref{thm:offset-correction}, $J$ is a lower triangular matrix with
  $\diag(J) = (1,0,\dots,0)$. This yields
  $\sigma(J)\subseteq\set{0,1}$. Furthermore, we can say the
  following. If $n_{0} = 1$, then $J=(1)$ and thus,
  $\sigma(J) = \set{1}$.  If $n_{0} \geq 2$, then
  $\diag(J) = (1,0,\dots,0)$ holds with $n_{0}- 1 \geq 1$ zeros which
  implies $\sigma(J) = \set{0,1}$. This yields the result as stated.
\end{proof}

Before proving
Lemma~\ref{lemma:simple-growth-property-block-triangular-matrices},
we need another lemma.

\begin{lemma}\label{lemma:growth-block-triangular-matrices}
  Let $\calG$ be a finite set of $(D_1+D_2)\times (D_1+D_2)$ block upper
  triangular matrices. For $G\in\calG$ write
  \begin{equation*}
    G =
    \begin{pmatrix}
      G^{(11)}& G^{(12)}\\0&G^{(22)}
    \end{pmatrix}
  \end{equation*}
  where the block~$G^{(ij)}$ is a $D_i\times D_j$ matrix for $1\le i\le j\le 2$.

  Let $R_1$ and $R_2$ be non-negative constants such that
  \begin{equation*}
    \norm{G^{(ii)}_1\ldots G^{(ii)}_k}=O(R_i^k)
  \end{equation*}
  holds for all $G_1$, \ldots, $G_k\in\calG$, $i\in\{1, 2\}$ and
  $k\to\infty$.

  Assume that $R_1\neq R_2$ and set $R\coloneqq \max\{R_1, R_2\}$. Then
  \begin{equation*}
    \norm{G_1\ldots G_k}=O(R^k)
  \end{equation*}
  holds for all $G_1$, \ldots, $G_k\in\calG$ and
  $k\to\infty$.
\end{lemma}
\begin{proof}
  As both the assumption and the statement of the lemma do not depend on the
  particular norm by the norm equivalence theorem, it suffices to consider
  the maximum norm on vectors and thus the row sum norm on matrices.

  It is easily shown by induction on $k$ that
  \begin{equation*}
    G_1\ldots G_k =
    \begin{pmatrix}
      G^{(11)}_1\ldots G^{(11)}_k &
      \sum_{j=1}^{k}\bigl(G^{(11)}_1\ldots G^{(11)}_{j-1}\bigr) G^{(12)}_j \bigl(G^{(22)}_{j+1}\ldots G^{(22)}_k\bigr)\\
      0 &
      G^{(22)}_1\ldots G^{(22)}_k
    \end{pmatrix}.
  \end{equation*}
  If one of $R_1$, $R_2$ equals zero, all except at most one of the upper right
  summands vanish and the result follows easily.

  Otherwise, assume without loss of generality that $R_1<R_2$. Then
  \begin{align*}
    \norm[\bigg]{\sum_{j=1}^{k}\bigl(G^{(11)}_1\ldots G^{(11)}_{j-1}\bigr)
      G^{(12)}_j \bigl(G^{(22)}_{j+1}\ldots G^{(22)}_k\bigr)}&=
    O\biggl(\sum_{j=1}^{k} R_1^{j-1}R_2^{k-j}\biggr)\\&=
    O\biggl(R_2^{k} \sum_{j=1}^{k} \Bigl(\frac{R_1}{R_2}\Bigr)^j\biggr)
    =O(R_2^{k})
    \end{align*}
  as $k \to \infty$, because the last sum can be estimated from above by a bounded geometric
  series.

  Adding the contributions of the rows leads to the result.
\end{proof}

We are now able to prove
Lemma~\ref{lemma:simple-growth-property-block-triangular-matrices}.

\begin{proof}[Proof of Lemma~\ref{lemma:simple-growth-property-block-triangular-matrices}]
  As the joint spectral radius as well as the simple growth property do not
  depend on the choice of the norm, we only consider the row sum norm in this
  proof.

  By definition of the joint spectral radius, for every choice of
  constants $R_i>\rho(\calG^{(i)})$,
  \begin{equation*}
    \norm{G^{(ii)}_1\ldots G^{(ii)}_k}=O(R_i^k)
  \end{equation*}
  holds for $G_1$, \ldots, $G_k\in\calG$, $1\le i\le s$ and $k\to\infty$ (where
  the implicit constant depends on the $R_i$).
  If $\calG^{(i)}$ has the simple growth property, we may also choose $R_i=\rho(\calG^{(i)})$.
  Without loss of generality, we can choose the $R_i$ to be pairwise distinct.
  Then repeated application of
  Lemma~\ref{lemma:growth-block-triangular-matrices} shows that
  \begin{equation}\label{eq:proof:simple-growth-property-block-triangular-matrices:norm-of-products-bounded}
    \norm{G_1\ldots G_k} = O((\max_{1\le i\le s}R_i)^k)
  \end{equation}
  holds for all $G_1$, \dots, $G_k\in\mathcal{G}$ and $k\to\infty$.
  Consequently, we obtain $\rho(\calG)\le
  \max_{1\le i\le s}R_i$. Taking the infimum over all possible choices of $R_i$
  shows that $\rho(\calG)\le
  \max_{1\le i\le s}\rho(\calG^{(i)})$. As the diagonal blocks of the product
  of block triangular matrices are the products of the corresponding diagonal
  blocks, it follows that $\rho(\calG)\ge
  \max_{1\le i\le s}\rho(\calG^{(i)})$.

  If there is a unique $i_0\in\{1, \ldots, s\}$ such that
  $\rho(\calG^{(i_0)})=\rho(\calG)$ and $\calG^{(i_0)}$ has the simple growth
  property, we can choose $R_{i_0}=\rho(\calG^{(i_0)})$ and $R_i<\rho(\calG^{(i_0)})$ for $i\neq i_0$
  and the simple growth property of $\calG$ follows
  from~\eqref{eq:proof:simple-growth-property-block-triangular-matrices:norm-of-products-bounded}.
\end{proof}

\begin{proof}[Proof of Proposition~\ref{prop:jsr-A}]
  As the matrices $J_r$ are lower triangular matrices with diagonal elements
  $0$ and $1$,
  Lemma~\ref{lemma:simple-growth-property-block-triangular-matrices} (or
  Jungers~\cite[Proposition~1.5]{Jungers:2009:joint-spectral-radius}) implies
  that $\rho(\mathcal{J})\le 1$. As $J_0$ has a one on the diagonal, we
  actually have $\rho(\mathcal{J})=1$.

  Since all matrices $\widetilde{A}_{0}$, \dots, $\widetilde{A}_{q-1}$
  are upper triangular block matrices, we can
  apply Lemma~\ref{lemma:simple-growth-property-block-triangular-matrices} (or Jungers~\cite[Proposition~1.5]{Jungers:2009:joint-spectral-radius})
  once more to obtain the first part
  of~\eqref{eq:prop-jsr-A-1}.
  The ``in particular''-statement follows directly
  from~\eqref{eq:prop-jsr-A-1}.

  Finally, Lemma~\ref{lemma:simple-growth-property-block-triangular-matrices}
  shows that $\widetilde{\mathcal{A}}$ has the simple growth property under the
  assumptions stated in the last part of Proposition~\ref{prop:jsr-A}.
\end{proof}

\begin{proof}[Proof of Proposition~\ref{cor:sprectrum-C-Bs}]
  Observe that we have
  \begin{equation*}
    C =
    \begin{pmatrix}
      J_{00} + \cdots + J_{(q-1)0} & J_{01} + \cdots + J_{(q-1)1}\\
      0 & B_{0} + \cdots + B_{q-1}
    \end{pmatrix}.
  \end{equation*}
  So $C$ is an upper triangular block matrix, which implies
  \begin{equation*}
    \sigma(C) = \sigma(J_{00} + \cdots + J_{(q-1)0}) \cup\sigma(B_{0} + \cdots
    + B_{q-1}).
  \end{equation*}
  By Theorem~\ref{prop:recursive-special-case},
  $J_{00} + \cdots + J_{(q-1)0}$ is a triangular matrix with zero diagonal,
  so its only eigenvalue is~$0$, which yields the result.
\end{proof}

\begin{proof}[Proof of Proposition~\ref{cor:jsr-special-case}]
  The statement $\rho(\mathcal{J}) = 0$ follows by
  Lemma~\ref{lemma:simple-growth-property-block-triangular-matrices} and the
  fact that every matrix in~$\mathcal{J}$ is an upper triangular matrix with zeros on the diagonal. The
  rest of the proposition can be proven in analogy to
  Proposition~\ref{prop:jsr-A}.
\end{proof}

\subsection{Proof of the Functional Equation in the Special Case}

\begin{proof}[Proof of Proposition~\ref{prop:dirichlet}]
  For $0\leq j < q^{m}$ and $\Re s>\log_q \rho +1$, replacing $n$ by $qn+\mu$ yields
  \begin{align*}
    \mathcal{X}_{j}(s) &= \sum_{n\geq \eta}\frac{x(q^{m}n + j)}{(q^{m}n + j)^{s}}\\
                       &= \sum_{\mu =0}^{q-1}\;\sum_{n\geq \eta}\frac{x(q^{m+1}n +
                         \mu q^{m} + j)}{(q^{m+1}n + \mu q^{m} + j)^{s}} + \overbrace{\sum_{\eta\leq n < q\eta}\frac{x(q^{m}n + j)}{(q^{m}n + j)^{s}}}^{\eqqcolon \sigma}.\\
  \intertext{We now use \eqref{eq:def-q-recursive} and obtain}
    \mathcal{X}_{j}(s)&= \sum_{\mu =0}^{q-1}\;\sum_{n\geq \eta}\frac{\sum_{k=0}^{q^{m}-1}c_{\mu q^{m} +
      j,k}\,x(q^{m}n + k)}{(q^{m+1}n + \mu q^{m} + j)^{s}} + \sigma\\
                       &= q^{-s}\sum_{\mu =0}^{q-1}\;\sum_{k=0}^{q^{m}-1}c_{\mu q^{m} +
                         j,k}\sum_{n\geq \eta}\frac{x(q^{m}n +
                         k)}{\bigl(q^{m}n + \frac{\mu q^{m} + j}{q}\bigr)^{s}} + \sigma\\
                       &= q^{-s}\sum_{k=0}^{q^{m}-1}\Bigl(\sum_{\mu =0}^{q-1}c_{\mu q^{m} +
                         j,k}\Bigr)\mathcal{X}_{k}(s) + \mathcal{Y}_{j}(s),
  \end{align*}
  where $\mathcal{Y}_{j}(s)$ is given by
  \begin{equation}
    \label{eq:ub-functional-eqation-rho}
    \mathcal{Y}_{j}(s) =
    q^{-s}\sum_{k=0}^{q^{m}-1}\;\sum_{\mu =0}^{q-1}c_{\mu q^{m}+j,k}\biggl(\sum_{n\geq
      \eta}\frac{x(q^{m}n + k)}{\bigl(q^{m}n + \frac{\mu q^{m} + j}{q}\bigr)^{s}} - \mathcal{X}_{k}(s)\biggr) + \sigma.
  \end{equation}
  Note that $\sum_{k=0}^{q^{m}-1}\sum_{\mu =0}^{q-1}c_{\mu q^{m} +
                         j,k}\mathcal{X}_{k}(s)$ is the $j$th row of $(B_0+\cdots +B_{q-1})\mathcal{X}(s)$.
  It is easy to see that
  \begin{equation*}
    \abs[\Big]{\frac{\mu q^{m} + j}{q} - k} < q^{m} \leq q^{m}\eta + k
  \end{equation*}
  holds for all $0\leq \mu < q$, $0\leq j < q^{m}$ and $0\leq k < q^{m}$. Consequently, we
  apply~\cite[Lemma~6.3]{Heuberger-Krenn:2018:asy-regular-sequences}
  with $n_{0} = q^{m}\eta + k$ (see~\eqref{eq:dirichlet-special-case}) as well as $\beta = \frac{\mu q^{m}+j}{q} - k$ and obtain
  \begin{equation}
    \label{eq:proof-prop-dirichlet:rhs-analytic}
    \sum_{n\geq \eta}\frac{x(q^{m}n + k)}{\bigl(q^{m}n + \frac{\mu q^{m} + j}{q}\bigr)^{s}} - \mathcal{X}_{k}(s)
    = \sum_{n \geq 1}\binom{-s}{n}\Bigl(\frac{\mu q^{m} + j}{q} - k\Bigr)^{n}\mathcal{X}_{k}(s + n),
  \end{equation}
  again for all $0\leq \mu < q$, $0\leq j < q^{m}$ and $0\leq k < q^{m}$. Furthermore, the
  right-hand side of~\eqref{eq:proof-prop-dirichlet:rhs-analytic}
  is analytic for $\Re s > \log_{q}\rho$. Plugging \eqref{eq:proof-prop-dirichlet:rhs-analytic}
  into~\eqref{eq:ub-functional-eqation-rho} and reordering
  terms yields the result.
  The functional equation~\eqref{eq:prop-functional-equation} implies that $\mathcal{X}_j(s)$
  is meromorphic for $\Re s>\log_q\rho$ and can only have poles where
  $q^{s}\in\sigma(B_{0} + \cdots + B_{q-1})$.
\end{proof}



\newcommand{\MR}[1]{}
\bibliography{cheub-bib/cheub}

\providecommand{\Submitted}{Submitted} \providecommand{\availableat}{ available
  at } \providecommand{\alsoavailableat}{ also available at }
  \providecommand{\evavailableat}{earlier version available at }
  \providecommand{\toappearin}{To appear in } \providecommand{\toappear}{to
  appear} \providecommand{\inpreparation}{in preparation}
  \providecommand{\doi}[1]{\href{http://dx.doi.org/#1}{\path{doi:#1}}}
  \providecommand{\lowercaseforams}{}
  \providecommand{\etc}{\emph{etc.}}\def\cprime{$'$}
\providecommand{\bysame}{\leavevmode\hbox to3em{\hrulefill}\thinspace}
\providecommand{\MR}{\relax\ifhmode\unskip\space\fi MR }
\providecommand{\MRhref}[2]{%
  \href{http://www.ams.org/mathscinet-getitem?mr=#1}{#2}
}
\providecommand{\href}[2]{#2}
\begin{thebibliography}{10}

\bibitem{Allouche-Shallit:1992:regular-sequences}
Jean-Paul Allouche and Jeffrey Shallit,
  \href{http://dx.doi.org/10.1016/0304-3975(92)90001-V}{\emph{The ring of
  $k$-regular sequences}}, Theoret. Comput. Sci. \textbf{98} (1992), no.~2,
  163--197. \MR{1166363}

\bibitem{Allouche-Shallit:2003:autom}
\bysame, \href{http://dx.doi.org/10.1017/CBO9780511546563}{\emph{Automatic
  sequences: Theory, applications, generalizations}}, Cambridge University
  Press, Cambridge, 2003. \MR{1997038 (2004k:11028)}

\bibitem{Allouche-Shallit:2003:regular-sequences-2}
\bysame, \href{http://dx.doi.org/10.1016/S0304-3975(03)00090-2}{\emph{The ring
  of $k$-regular sequences, {II}}}, Theoret. Comput. Sci. \textbf{307} (2003),
  no.~1, 3--29. \MR{2014728}

\bibitem{Bell:2005:regular-sequences-values}
Jason~P. Bell, \href{http://dx.doi.org/10.1016/j.aam.2004.11.004}{\emph{On the
  values attained by a {$k$}-regular sequence}}, Adv. in Appl. Math.
  \textbf{34} (2005), no.~3, 634--643. \MR{2123552}

\bibitem{Berstel-Reutenauer:2011:noncommutative-rational-series}
Jean Berstel and Christophe Reutenauer, \emph{Noncommutative rational series
  with applications}, Encyclopedia of Mathematics and its Applications, vol.
  137, Cambridge University Press, Cambridge, 2011. \MR{2760561}

\bibitem{Bicknell:2003:stern-diatomic-fibonacci}
Marjorie Bicknell-Johnson, \emph{Stern's diatomic array applied to {F}ibonacci
  representations}, Fibonacci Quart. \textbf{41} (2003), no.~2, 169--180.
  \MR{1990526}

\bibitem{Brocot:1862:calcul}
Achille Brocot,
  \href{http://catalogue.bnf.fr/ark:/12148/cb30164108s}{\emph{Calcul des
  rouages par approximation, nouvelle m{\'e}thode}}, L'auteur, 1862, Source:
  Biblioth{\`e}que nationale de France.

\bibitem{Carlitz:1964:partitions-stirling}
Leonard Carlitz, \emph{A problem in partitions related to the {S}tirling
  numbers}, Riv. Mat. Univ. Parma (2) \textbf{5} (1964), 61--75. \MR{0200259}

\bibitem{Charlier-Rampersad-Shallit:2012:enumeration-properties-automatic-seq}
\'{E}milie Charlier, Narad Rampersad, and Jeffrey Shallit,
  \href{http://dx.doi.org/10.1142/S0129054112400448}{\emph{Enumeration and
  decidable properties of automatic sequences}}, Internat. J. Found. Comput.
  Sci. \textbf{23} (2012), no.~5, 1035--1066. \MR{2983367}

\bibitem{Coons-Shallit:2011:pattern-sequence-approach-Stern}
Michael Coons and Jeffrey Shallit,
  \href{http://dx.doi.org/10.1016/j.disc.2011.07.029}{\emph{A pattern sequence
  approach to {S}tern's sequence}}, Discrete Math. \textbf{311} (2011), no.~22,
  2630--2633. \MR{2832129}

\bibitem{Coons-Spiegelhofer:2018:number-theoretic-regular-sequences}
Michael Coons and Lukas Spiegelhofer,
  \href{http://dx.doi.org/10.1007/978-3-319-69152-7\_2}{\emph{Number theoretic
  aspects of regular sequences}}, Sequences, groups, and number theory, Trends
  Math., Birkh\"{a}user/Springer, Cham, 2018, pp.~37--87. \MR{3799924}

\bibitem{Currie-Saari:2009:least-periods-factors}
James~D. Currie and Kalle Saari,
  \href{http://dx.doi.org/10.1051/ita:2008006}{\emph{Least periods of factors
  of infinite words}}, Theor. Inform. Appl. \textbf{43} (2009), no.~1,
  165--178. \MR{2483449}

\bibitem{Dumas:2013:joint}
Philippe Dumas, \href{http://dx.doi.org/10.1016/j.laa.2012.10.013}{\emph{Joint
  spectral radius, dilation equations, and asymptotic behavior of
  radix-rational sequences}}, Linear Algebra Appl. \textbf{438} (2013), no.~5,
  2107--2126.

\bibitem{Finch:constants:2003}
Steven~R. Finch,
  \href{http://dx.doi.org/10.1017/CBO9780511550447}{\emph{Mathematical
  constants}}, Encyclopedia of Mathematics and its Applications, vol.~94,
  Cambridge University Press, Cambridge, 2003. \MR{2003519 (2004i:00001)}

\bibitem{Fine:1947:binomial-coefficients-modulo-p}
Nathan~J. Fine, \href{http://dx.doi.org/10.2307/2304500}{\emph{Binomial
  coefficients modulo a prime}}, Amer. Math. Monthly \textbf{54} (1947),
  589--592. \MR{23257}

\bibitem{Goc-Henshall-Shallit:2013:automatic-theorem-proving}
Daniel Go\v{c}, Dane Henshall, and Jeffrey Shallit,
  \href{http://dx.doi.org/10.1142/S0129054113400182}{\emph{Automatic
  theorem-proving in combinatorics on words}}, Internat. J. Found. Comput. Sci.
  \textbf{24} (2013), no.~6, 781--798. \MR{3158968}

\bibitem{Goc-Mousavi-Shallit:2013}
Daniel Go\v{c}, Hamoon Mousavi, and Jeffrey Shallit,
  \href{https://doi.org/10.1007/978-3-642-37064-9_27}{\emph{On the number of
  unbordered factors}}, Language and automata theory and applications, Lecture
  Notes in Comput. Sci., vol. 7810, Springer, Heidelberg, 2013, pp.~299--310.
  \MR{3090327}

\bibitem{Goc-Rampersad-Rigo-Salimov:2014:abelian-bordered-words}
Daniel Go\v{c}, Narad Rampersad, Michel Rigo, and Pavel Salimov,
  \href{http://dx.doi.org/10.1142/S0129054114400267}{\emph{On the number of
  abelian bordered words (with an example of automatic theorem-proving)}},
  Internat. J. Found. Comput. Sci. \textbf{25} (2014), no.~8, 1097--1110.
  \MR{3315809}

\bibitem{Graham-Knuth-Patashnik:1994}
Ronald~L. Graham, Donald~E. Knuth, and Oren Patashnik, \emph{Concrete
  mathematics. {A} foundation for computer science}, second ed.,
  Addison-Wesley, 1994.

\bibitem{Heuberger-Krenn:2018:asy-regular-sequences}
Clemens Heuberger and Daniel Krenn,
  \href{http://dx.doi.org/10.1007/s00453-019-00631-3}{\emph{Asymptotic analysis
  of regular sequences}}, Algorithmica \textbf{82} (2020), no.~3, 429--508.
  \MR{4058416}

\bibitem{Heuberger-Krenn-Lipnik:2022:minimisation-notes}
Clemens Heuberger, Daniel Krenn, and Gabriel~F. Lipnik,
  \href{https://arxiv.org/abs/2201.13446}{\emph{A note on the relation between
  recognisable series and regular sequences, and their minimal linear
  representations}}, arXiv:2201.13446 [math.CO], 2022.

\bibitem{Heuberger-Krenn-Prodinger:2018:pascal-rhombus}
Clemens Heuberger, Daniel Krenn, and Helmut Prodinger,
  \href{http://dx.doi.org/10.4230/LIPIcs.AofA.2018.27}{\emph{Analysis of
  summatory functions of regular sequences: Transducer and {P}ascal's
  rhombus}}, Proceedings of the 29th International Conference on Probabilistic,
  Combinatorial and Asymptotic Methods for the Analysis of Algorithms
  (James~Allen Fill and Mark~Daniel Ward, eds.), Leibniz International
  Proceedings in Informatics (LIPIcs), vol. 110, Schloss
  Dagstuhl--Leibniz-Zentrum f\"{u}r Informatik, 2018, pp.~27:1--27:18.
  \MR{3826146}

\bibitem{Hinz-Klavzar-Milutinovic-Parisse-Petr:metric-properties-hanoi-stern-diatomic}
Andreas~M. Hinz, Sandi Klav\v{z}ar, Uro\v{s} Milutinovi\'{c}, Daniele Parisse,
  and Ciril Petr,
  \href{http://dx.doi.org/10.1016/j.ejc.2004.04.009}{\emph{Metric properties of
  the {T}ower of {H}anoi graphs and {S}tern's diatomic sequence}}, European J.
  Combin. \textbf{26} (2005), no.~5, 693--708. \MR{2127690}

\bibitem{Hwang-Janson-Tsai:2017:divide-conquer-half}
Hsien-Kuei Hwang, Svante Janson, and Tsung-Hsi Tsai,
  \href{http://dx.doi.org/10.1145/3127585}{\emph{Exact and asymptotic solutions
  of a divide-and-conquer recurrence dividing at half: Theory and
  applications}}, ACM Trans. Algorithms \textbf{13} (2017), no.~4, Art.~47,
  43~pp.

\bibitem{Jungers:2009:joint-spectral-radius}
Rapha\"el Jungers, \href{http://dx.doi.org/10.1007/978-3-540-95980-9}{\emph{The
  joint spectral radius. {T}heory and applications}}, Lecture Notes in Control
  and Information Sciences, vol. 385, Springer-Verlag, Berlin, 2009.
  \MR{2507938}

\bibitem{Krenn-Lipnik:2019:trac-regular-sequence-from-recurrences}
Daniel Krenn and Gabriel~F. Lipnik, \emph{$k$-{R}egular sequences in
  {SageMath}: {C}onstruction of $q$-recursive sequences},
  \url{https://trac.sagemath.org/ticket/27940},
  \url{https://trac.sagemath.org/ticket/31787},
  \url{https://trac.sagemath.org/ticket/32198},
  \url{https://trac.sagemath.org/ticket/32921},
  \url{https://trac.sagemath.org/ticket/33158}, 2022.

\bibitem{Leroy-Rigo-Stipulanti:2016:generalized-pascal-triangle}
Julien Leroy, Michel Rigo, and Manon Stipulanti,
  \href{http://dx.doi.org/10.1016/j.aam.2016.04.006}{\emph{Generalized {P}ascal
  triangle for binomial coefficients of words}}, Adv. in Appl. Math.
  \textbf{80} (2016), 24--47. \MR{3537237}

\bibitem{Leroy-Rigo-Stipulanti:2017:digital-sequences-exotic}
\bysame, \href{http://dx.doi.org/10.37236/6581}{\emph{Behavior of digital
  sequences through exotic numeration systems}}, Electron. J. Combin.
  \textbf{24} (2017), no.~1, Paper No. 1.44, 36. \MR{3625921}

\bibitem{Leroy-Rigo-Stipulanti:2017:non-zero-generalized-pascal-triangle}
\bysame, \href{http://dx.doi.org/10.1016/j.disc.2017.01.003}{\emph{Counting the
  number of non-zero coefficients in rows of generalized {P}ascal triangles}},
  Discrete Math. \textbf{340} (2017), no.~5, 862--881. \MR{3612418}

\bibitem{Leroy-Rigo-Stipulanti:2018:counting-subword-occurences}
\bysame, \emph{Counting subword occurrences in base-{$b$} expansions}, Integers
  \textbf{18A} (2018), Paper No. A13, 32~pp. \MR{3777535}

\bibitem{Lothaire:1997:comb-words}
M.~Lothaire,
  \href{http://dx.doi.org/10.1017/CBO9780511566097}{\emph{Combinatorics on
  words}}, Cambridge Mathematical Library, Cambridge University Press,
  Cambridge, 1997. \MR{1475463}

\bibitem{Lothaire:2002:algeb}
\bysame, \emph{Algebraic combinatorics on words}, Encyclopedia of Mathematics
  and its Applications, vol.~90, Cambridge University Press, Cambridge, 2002.
  \MR{1 905 123}

\bibitem{Northshield:stern-diatomic}
Sam Northshield,
  \href{http://dx.doi.org/10.4169/000298910X496714}{\emph{Stern's diatomic
  sequence {$0,1,1,2,1,3,2,3,1,4,\dots$}}}, Amer. Math. Monthly \textbf{117}
  (2010), no.~7, 581--598. \MR{2681519}

\bibitem{OEIS:2021}
\emph{The {O}n-{L}ine {E}ncyclopedia of {I}nteger {S}equences},
  \url{http://oeis.org}, 2021.

\bibitem{Rota-Strang:1960}
Gian-Carlo Rota and Gilbert Strang, \emph{A note on the joint spectral radius},
  Indag. Math. \textbf{22} (1960), 379--381. \MR{0147922}

\bibitem{SageMath:2021:9.4}
{The SageMath Developers}, \emph{{SageMath} {M}athematics {S}oftware ({V}ersion
  9.4)}, 2021, \url{http://www.sagemath.org}.

\bibitem{Stern:1858:zahlentheoretische-funktion}
Moritz~A. Stern, \href{http://dx.doi.org/10.1515/crll.1858.55.193}{\emph{Ueber
  eine zahlentheoretische {F}unktion}}, J. Reine Angew. Math. \textbf{55}
  (1858), 193--220. \MR{1579066}

\end{thebibliography}
\bibliographystyle{bibstyle/amsplainurl}

\end{document}